\documentclass[11pt]{amsart}
\usepackage{amsmath,amsmath,latexsym,amssymb,amsfonts,amsbsy, amsthm}
\usepackage{mathrsfs}
\usepackage{fancyhdr}
\usepackage{latexsym}
\usepackage{bbding}
\usepackage{wasysym}
\usepackage{cite}
\usepackage{multicol,graphics}
\usepackage{amssymb}
\usepackage{graphicx}
\usepackage{color}
\usepackage{bm}
\usepackage{enumitem}
\allowdisplaybreaks

\newtheorem{theorem}{Theorem}[section]
\newtheorem{lemma}[theorem]{Lemma}
\newtheorem{proposition}[theorem]{Proposition}
\newtheorem{remark}[theorem]{Remark}
\newtheorem{corollary}[theorem]{Corollary}
\newtheorem{definition}{Definition}[section]

\numberwithin{equation}{section}

\begin{document}

\title[Boundary Control of Ericksen-Leslie]{On optimal boundary control of Ericksen-Leslie system in dimension two}
\author{Qiao Liu,  Changyou Wang, Xiaotao Zhang, Jianfeng Zhou}
\date{}
\begin{abstract}
In this paper, we consider the boundary value problem of a simplified Ericksen-Leslie system in dimension two
with non-slip boundary condition for the velocity field $u$ and time-dependent boundary 
condition  for the director field $d$ of unit length. For such a system, we first establish the existence of
a global weak 
solution that is smooth away from finitely many singular times, then establish the existence of a unique global strong 
solution that is smooth for $t>0$ under the assumption that the image of boundary data  is contained in the hemisphere $\mathbb S^2_+$. 
Finally, we apply these theorems to the problem of optimal boundary control of the simplified
Ericksen-Leslie system and show both the existence and a necessary condition
of an optimal boundary control.
\end{abstract}



\maketitle

\tableofcontents

\section{Introduction}
\setcounter{equation}{0}
\setcounter{theorem}{0}

In this paper, we will consider the boundary value problem of a simplified Ericksen-Leslie system,  first proposed
by Lin \cite{Lin1}, that models the hydrodynamic motion of nematic liquid crystals under the non-slip boundary condition
on the velocity field and time-dependent
boundary condition on the liquid crystal director field in dimension two.
More precisely, for a bounded smooth domain $\Omega\subseteq \mathbb{R}^{2}$ with $\Gamma=\partial\Omega$ and $0<T<\infty$,
set $Q_T=\Omega\times {(0, T)}$ and $\Gamma_T=\Gamma\times {(0,T)}$. 
We seek a $({\bf u},{\bf d}):Q_T\mapsto \mathbb R^2\times\mathbb S^2$ that solves
\begin{align}
   \label{eq1.1}
\begin{cases}
\mathbf{u}_{t}-\nu\Delta \mathbf{u} +(\mathbf{u}\cdot\nabla)\mathbf{u}+\nabla P=-\lambda\operatorname{div}(\nabla \mathbf{d} \odot\nabla \mathbf{d}),\\
\qquad \nabla\cdot \mathbf{u}=0,  \\
 \mathbf{d}_{t}+(\mathbf{u}\cdot\nabla)\mathbf{d}=\mu (\Delta \mathbf{d}+|\nabla \mathbf{d}|^{2}\mathbf{d}),
\end{cases} 
\ {\rm{in}}\ Q_T
\end{align}
subject to the  boundary and initial conditions: 
 \begin{align} \label{eq1.2}
 &\mathbf{u}({x},t)={0},\ \mathbf{d}({x},t)=\mathbf{h}({x},t),\ ({x}, t)\in \Gamma_T,\\
   \label{eq1.3}
 &\mathbf{u}|_{t=0}=\mathbf{u}_0(x), \
 \mathbf{d}|_{t=0}=\mathbf{d}_0(x),\ {x}\in \Omega,
 \end{align}
where ${\bf u}_0:\Omega\mapsto\mathbb R^2$ and ${\bf d_0}:\Omega\mapsto\mathbb R^3$
are given maps such that $\nabla\cdot \mathbf{u}_0(x)=0$ and $|\mathbf{d}_{0}(x)| = 1$ for
$x\in\Omega$. Here $\mathbf{u}:Q_T\mapsto
\mathbb{R}^{2}$ represents the fluid velocity field,
$P:Q_T\mapsto \mathbb{R}$ represents the
pressure function, $\mathbf{d}:\Omega\times
(0,T)\mapsto \mathbb{S}^{2}=\{y\in\mathbb{R}^{3}: \, |y|=1\}$ 
represents the averaged orientation field of the nematic liquid crystal molecules,
and $\nu$, $\mu$ and $\lambda$ are three positive constants that represent viscosity,
the competition between kinetic energy and potential energy, and microscopic
elastic relaxation time for the molecular orientation field. The
notation $\nabla \mathbf{d}\odot\nabla \mathbf{d}$ stands for the $2\times 2$ matrix
whose ($i,j$)-th entry is given by $\nabla_i {\bf d}\cdot \nabla_j {\bf d}$  for $1\le i, j\le 2$.

 The system \eqref{eq1.1} is a simplified version of the general
Ericksen-Leslie system, which was introduced by Ericksen \cite{Ericksen} and Leslie \cite{Leslie}
between 1958  and  1968, that represents a macroscopic continuum description of the time evolution of the 
material under the influence of both the fluid velocity field ${\bf u}$ and the macroscopic
description of the microscopic orientation configuration field ${\bf d}$ of rod-like liquid crystal molecules.
Mathematically,  \eqref{eq1.1} is a system that strongly couples between the non-homogeneous Navier-Stokes
equation and the transported heat flow of harmonic maps to $\mathbb S^2$ (see \cite{CL1993} for the Dirichlet
problem of harmonic heat flows).

There have been many interesting results on the system \eqref{eq1.1}, \eqref{eq1.2}, and \eqref{eq1.3},
when the boundary data ${\bf h}(x,t)={\bf h}(x): \Gamma_T\mapsto \mathbb S^2$  is $t$-independent. 
In a series of papers, Lin \cite{Lin1} and Lin-Liu \cite{LL1, LL2} initiated the mathematical analysis of \eqref{eq1.1}--\eqref{eq1.3} in 1990's.
More precisely, they considered in \cite{LL1, LL2} the case of the so-called Ericksen's variable degree of orientation, that replaces the Dirichlet
energy $\frac12\int_\Omega |\nabla {\bf d}|^2\,dx$ for ${\bf d}:\Omega\mapsto\mathbb S^2$ by the Ginzburg-Landau energy
$\int_\Omega (\frac12|\nabla {\bf d}|^2+\frac{1}{4\epsilon^2}(1-|{\bf d}|^2)^2)\,dx$ for ${\bf d}:\Omega\mapsto\mathbb R^3$.
It has been established by \cite{LL1, LL2} the global existence of strong and weak solutions in dimensions $2$ and $3$ respectively, along with
some partial regularity results analogous to \cite{CKN} on the Navier-Stokes equation. 
However, since the arguments and crucial estimates in \cite{LL1,LL2} depend on the parameter $\epsilon$, it is challenging to utilize
such obtained solutions as approximate solutions to \eqref{eq1.1}, \eqref{eq1.2}, and \eqref{eq1.3} because it remains
a very difficult question to study their convergence as $\epsilon$ tends to zero.  In dimension two, it has been recently  
shown by Lin-Lin-Wang \cite{LLW} and Lin-Wang \cite{LW0} that \eqref{eq1.1}, \eqref{eq1.2}, and \eqref{eq1.3} admits a unique global
``almost regular" weak solution that is smooth away from finitely many singular times when ${\bf h}(={\bf d}_0|_{\partial\Omega})
\in C^{2,\alpha}(\partial\Omega, \mathbb S^2)$, see Hong \cite{Hong}, Xu-Zhang \cite{XZ}, Hong-Xin \cite{HX}, Huang-Lin-Wang \cite{HLW}, and
Lei-Li-Zhang \cite{LLZ} for related works. In a very recent article, Lin-Wang \cite{LW2} have proved 
the existence of a global weak solution to \eqref{eq1.1}, 
\eqref{eq1.2}, and \eqref{eq1.3} in dimension $3$ when ${\bf d}_0(\Omega)\subset\mathbb S^2_+$, see  Huang-Lin-Liu-Wang \cite{HLLW} for related works. 
The interested reader can also consult the survey article \cite{LW1} and the papers by 
Lin-Liu \cite{LL3},  Wang-Zhang-Zhang \cite{WZZ}, Cavaterra-Rocca-Wu \cite{CRW2013}, 
and Wu-Lin-Liu \cite{WLL} for related works on the general Ericksen-Leslie system.

Turning to the technically more challenging case of $t$-dependent boundary data ${\bf h}:Q_T\mapsto \mathbb S^2$ for ${\bf d}$, 
to the best of our knowledge there has been no previous work addressing \eqref{eq1.1}, \eqref{eq1.2}, and \eqref{eq1.3} 
available in the literature yet.  For the Ericksen-Leslie system in the case of Ericksen's variable degree of orientation or the so-called Ginzburg-Landau approximate version of \eqref{eq1.1}, \eqref{eq1.2}, and \eqref{eq1.3}, there has been several interesting works by \cite{Bosia, CGR, CGM, GW, GRR} extending the main theorems by Lin-Liu \cite{LL1, LL2} to $t$-dependent boundary data for ${\bf d}$. In particular,  Cavaterra-Rocca-Wu \cite{CRW2017} have recently studied the optimal boundary control issue for such a system in dimension $2$ (see the books \cite{ATF, HPU, T2010}
for more discussions on optimal control of PDEs).
The motivation for the study we undertake in this paper is two fold: 
\begin{itemize}
\item [(i)] We are interested in extending the previous theorems by Lin-Lin-Wang \cite{LLW} and
establish the theory of global weak and strong solutions of \eqref{eq1.1}, \eqref{eq1.2}, and \eqref{eq1.3} for $t$-dependent boundary data of ${\bf d}$ in dimension $2$; and
\item [(ii)] We plan to employ the existence of a unique, global strong solutions to establish the existence of an optimal boundary control of \eqref{eq1.1}, \eqref{eq1.2}, and \eqref{eq1.3} and characterize a necessary condition of such an optimal
boundary control in a spirit similar to \cite{CRW2017}.
\end{itemize}

Now let us briefly set up the boundary control problem. Denote  ${\bf e}_3=(0,0,1)$, and set
\begin{align*}
&\mathbf{H}= \text{ Closure of } \{{\bf u}\in C_{0}^{\infty}(\Omega,\mathbb{R}^{2}) : \nabla\cdot\mathbf{u}=0\} \text{ in } L^{2}(\Omega,\mathbb{R}^{2}),\nonumber\\
&\mathbf{V}= \text{ Closure of } \{{\bf u}\in C_{0}^{\infty}(\Omega,\mathbb{R}^{2}):\nabla\cdot\mathbf{u}=0\} \text{ in } H^{1}_{0}(\Omega,\mathbb{R}^{2}),\\
&H^k(\Omega, \mathbb S^2)=\big\{{\bf d}\in H^k(\Omega,\mathbb R^3): |{\bf d}(x)|=1 \ \mbox{a.e.}\ x\in\Omega\big\}, \ k\ge 0.
\end{align*}
For any given $\beta_i\in \mathbb R_+$, $1\le i\le 5$, not all zeroes, and four target maps
$${\bf u}_{Q_T}\in L^2([0,T], {\bf H}),  {\bf d}_{Q_T}\in L^2(Q_T, \mathbb S^2), 
 {\bf u}_\Omega\in {\bf H},  {\bf d}_{\Omega}\in L^2(\Omega,\mathbb S^2),$$
our goal is to study the minimization problem of the cost functional
\begin{align}\label{CF}
2\mathcal{C}(({\bf u}, {\bf d}), {\bf h})&=\beta_1\|{\bf u}-{\bf u}_{Q_T}\|_{L^2(Q_T)}^2
+\beta_2\|{\bf d}-{\bf d}_{Q_T}\|_{L^2(Q_T)}^2\nonumber\\
&+\beta_3\|{\bf u}(T)-{\bf u}_{\Omega}\|_{L^2(\Omega)}^2
+ \beta_4\|{\bf d}(T)-{\bf d}_{\Omega}\|_{L^2(\Omega)}^2\nonumber\\
&+\beta_5 \|{\bf h}-{\bf e}_3\|_{L^2(\Gamma_T)}^2,
\end{align}
for any boundary data ${\bf h}:\Gamma_T\mapsto\mathbb S^2$ that lies in a bounded, closed, and geometric
convex set $\widetilde{\mathcal{U}}_M$ of
a function space $\mathcal{U}$, to be specified in the section 4, and $({\bf u}, {\bf d})$ is the unique global strong solution to the state problem
\eqref{eq1.1}, with the boundary condition $({\bf u}, {\bf d})=(0, {\bf h})$  and the initial condition
$({\bf u}_0, {\bf d}_0)$  for some fixed maps ${\bf u}_0\in {\bf H}$ and ${\bf d}_0\in H^1(\Omega,\mathbb S^2)$,
subject to the compatibility condition:
\begin{equation}\label{comp_cond}
{\bf h}(x,0)={\bf d}_0(x), \ x\in\Gamma.
\end{equation} 

In order to study both the existence and a necessary condition for an optimal boundary control ${\bf h}$ of the cost functional 
$\mathcal{C}(({\bf u}, {\bf d}), {\bf h})$ over $\widetilde{\mathcal{U}}_M$, we first need to establish the existence of a unique global strong solution
to \eqref{eq1.1}, \eqref{eq1.2}, and \eqref{eq1.3}. This turns out to be a challenging task, since the existence theorems by Lin-Lin-Wang \cite{LLW} and
Huang-Lin-Liu-Wang \cite{HLLW} indicate that the short time smooth solution may develop finite time singularity even for a $t$-independent
boundary value ${\bf d}_0$. There are several new difficulties that we need to overcome,
when we try to establish  Theorem \ref{thm1} 
extending the main theorem of Lin-Lin-Wang \cite{LLW}
to  \eqref{eq1.1}, \eqref{eq1.2}, and \eqref{eq1.3} for a $t$-dependent boundary value ${\bf h}$: 
\begin{itemize}
\item[1)] the energy $\mathcal{E}({\bf u}, {\bf d})(t)
=\frac12\int_\Omega (|{\bf u}|^2+|\nabla{\bf d}|^2)(t)$ may grow along the flow, which makes it difficult to
estimate the total number of singular times;
\item[2)] the boundary local energy inequality involves contributions by ${\bf h}$ that need to be carefully studied; and
\item[3)] the boundary $\epsilon_0$-regularity property is more delicate to establish than the case of $t$-independent boundary condition.
\end{itemize}

Under the additional assumption that the initial and boundary values ${\bf d}_0(\Omega)\subset\mathbb S^2_+$
and ${\bf h}(\Gamma_T)\subset\mathbb S^2_+$, where
$\mathbb{S}_{+}^{2}=\{{y}=(y^{1},y^{2},y^{3})\in \mathbb{S}^{2} \, \mid\, y^{3}\geq 0\}$, we are able to show in Theorem \ref{globalW}
that the weak solutions $({\bf u}, {\bf d})$ 
obtained in Theorem \ref{thm1} satisfy both ${\bf d}\in L^2([0,T], H^2(\Omega,\mathbb S^2_+))$ and  the smoothness property that
$({\bf u}, {\bf d})\in C^\infty(Q_T)\cap C^\alpha(\overline{\Omega}\times (0,T])$
for any $\alpha\in (0,1)$. In particular, if, in addition, $({\bf u}_0, {\bf d}_0)\in {\bf V}\times H^2(\Omega, \mathbb S^2_+)$
and ${\bf h}\in H^{\frac52, \frac54}(\Gamma_T,\mathbb S^2_+)$, then $({\bf u}, {\bf d})$ is the unique, strong solution
to \eqref{eq1.1}, \eqref{eq1.2}, and \eqref{eq1.3} that enjoys a priori estimate:
$$\|{\bf u}\|_{L^\infty_tH^1_x(Q_T)} +\|{\bf d}\|_{L^\infty_tH^2_x(Q_T)} 
+\|{\bf u}\|_{L^2_tH^2_x(Q_T)} +\|{\bf d}\|_{L^2_tH^3_x(Q_T)} \le C(T).$$
This estimate is established in the section 3, which turns out to be the crucial estimate in order to 
establish the Fr\'echet differentiability of the control to state map 
${\mathcal S}({\bf h})=({\bf u}, {\bf d})$ over appropriate function spaces, that is to be discussed in the section 4, 
by which we can obtain a necessary condition of an optimal boundary
control ${\bf h}: \Gamma_T\mapsto\mathbb S^2_+$. 

There have been many research articles on the optimal boundary control for parabolic equations, the Navier-Stokes equation, 
and the Cahn-Hilliard-Navier-Stokes system. See H\"omberg-Krumbiegel-Rehberg\cite{HKR}, Kunisch-Vexler\cite{KV},  Fattorini-Sritharan\cite{FS1, FS2}, Fursikov-Gunzburger-Hou \cite{FGH1, FGH2}, Hinze-Kunisch \cite{HK}, Frigeri-Rocca-Sprekels \cite{FRS}, Hinterm\"uller-Wedner \cite{HW}, Colli-Sprekels\cite{CP},
and Tr\"oltzsch \cite{T2010}, Alekseev-Tikhomirov-Fomin\cite{ATF}, and Hinze-Pinnau-Ulberich\cite{HPU}.

Since the exact  values of  $\nu$, $\mu$ and $\lambda$  don't play a role,
we henceforth assume $\nu=\mu=\lambda=1$.


\section{Existence of weak solutions}

In this section, we will establish the existence of a global weak solution to \eqref{eq1.1}, \eqref{eq1.2}, and \eqref{eq1.3}. 
First, let us recall a few notations. For any nonnegative number $k\ge 0$, recall the Sobolev spaces
$$H^k(\Gamma,\mathbb S^2)=\big\{f\in H^k(\Gamma, \mathbb R^3): \ f(x)\in\mathbb S^2\ {\rm{a.e.}}\ x\in\Gamma\big\},$$
$$H^{k, \frac{k}2}(\Gamma_T, \mathbb S^2)=\big\{f\in H^{k,\frac{k}2}(\Gamma_T,\mathbb R^3): \ f(x,t) \in \mathbb S^2,
 \ {\rm{a.e.}}\ (x,t)\in\Gamma_T\big\},$$
and the dual space of $H^k(\Gamma,\mathbb R^3)$,  $H^{-k}(\Gamma, \mathbb R^3)=(H^k(\Gamma,\mathbb R^3))'$.  
Our first theorem, which is an extension of \cite{LLW} Theorem 1.3 to time dependent boundary data,  states as follows. 

\begin{theorem}\label{thm1} For any $0<T<\infty$ and any bounded, smooth domain $\Omega\subset\mathbb R^2$ with boundary $\Gamma$,
assume that
\begin{align*}
\mathbf{h}\in L^2([0,T], H^{\frac{3}{2}}(\Gamma, \mathbb{S}^{2})), \ \partial_{t}\mathbf{h}\in 
L^{2}_tH^{\frac{3}{2}}_x(\Gamma_T),
\end{align*}
and
$(\mathbf{u}_{0}, \mathbf{d}_{0})\in \mathbf{H}\times H^{1}(\Omega,\mathbb{S}^{2})$
satisfies the compatibility condition \eqref{comp_cond}.
Then there exists a weak solution $(\mathbf{u}, \mathbf{d}):Q_T\mapsto \mathbb{R}^{2}\times\mathbb{S}^{2}$ of the
system \eqref{eq1.1}, with initial and boundary condition \eqref{eq1.2} and \eqref{eq1.3}, such that
\begin{align} \label{W-space}
\mathbf{u}\in L^{\infty}([0,T],\mathbf{H})\cap L^{2}([0,T],\mathbf{V}),
\ {\rm{and}}\ \mathbf{d} \in L^{\infty}_t H^{1}_x(Q_T,\mathbb{S}^{2}).
\end{align}
Furthermore, there exist $L\in\mathbb{N}$, depending only on $(\mathbf{u}_{0},\mathbf{d}_{0})$, and $0<T_{1}<\cdots<T_{L}\leq T$
such that $(\mathbf{u},\mathbf{d})$ is regular away from $\cup_{i=1}^L\{T_i\}$ in the sense 
that for any $0<\alpha<1$, 
$$(\mathbf{u},\mathbf{d})\in C^\infty(\Omega\times ((0,T]\setminus \cup_{i=1}^L \{T_i\}))
\cap C^{\alpha, \frac{\alpha}2}(\overline\Omega\times ((0,T]\setminus \cup_{i=1}^L \{T_i\})).$$
Moreover, there is a universal constant $\varepsilon_{1}>0$ such that 
for each singular time $T_{i}$, $1\leq i\leq L$, it holds that 
\begin{align}\label{blowup}
\limsup_{t\uparrow T_{i}} \max_{{x}\in\overline\Omega}
 \int_{\Omega\cap B_{r}({x})}(|\mathbf{u}|^{2}+|\nabla \mathbf{d}|^{2})(\cdot,t)\geq \varepsilon_{1}^2, \quad \forall r>0.
\end{align}
\end{theorem}

\medskip
A few remarks are in the order: 
\begin{remark} {\rm 1) Theorem \ref{thm1} was first established by \cite{LLW} for any time independent boundary data
${\bf h}(x,t)={\bf d}_0(x)$, $(x,t)\in\Gamma_T$, with ${\bf d}_0\in C^{2,\beta}(\Gamma, \mathbb S^2)$ for some $\beta\in (0,1)$.\\
2) By the Sobolev embedding theorem,  ${\bf h}\in L^2_tH^\frac32_x(\Gamma_T)$ and 
$\partial_t{\bf h}\in L^2_tH^\frac32_x(\Gamma_T)$
imply that ${\bf h}\in {{\rm Lip}}(\Gamma_T)$. We will present a new proof of the boundary $\epsilon_0$-regularity theorem on \eqref{eq1.1}
and \eqref{eq1.2} for any Lipschitz continuous boundary value ${\bf h}:\Gamma_T\mapsto\mathbb S^2$, which plays a crucial role in
the proof of Theorem \ref{thm1}. \\
3) In contrast with the autonomous boundary condition studied by \cite{LLW}, 
the system \eqref{eq1.1}, \eqref{eq1.2}, and \eqref{eq1.3} no longer enjoys the energy dissipation inequality 
for a time dependent boundary value ${\bf h}$.  However, under the  assumption that both ${\bf h}$ and
$\partial_{t}\mathbf{h}$ belong to $L^{2}_tH^{\frac{3}{2}}_x(\Gamma_T)$, we are able to estimate the growth
rate of the energy $\mathcal{E}({\bf u}(t), {\bf d}(t))$ by
\begin{align}\label{growth-E}
\mathcal{E}({\bf u}(t), {\bf d}(t))\leq\exp\big(C\int_0^t \|\partial_t{\bf h}\|_{H^\frac32(\Gamma)}\big) 
\big[\mathcal{E}({\bf u}_0, {\bf d}_0)+C\|({\bf h}, \partial_t{\bf h})\|^2_{L^2_tH^\frac32_x(\Gamma_T)}\big].
\end{align}
This turns to be sufficient for establishing the finiteness of singular times.}
\end{remark}

Applying the maximum principle, we can show that if $\mathbf{d}_{0}:\Omega\mapsto \mathbb{S}_{+}^{2}$ and $\mathbf{h}:\Gamma_T\mapsto \mathbb{S}_{+}^{2}$, then any weak solution $(\mathbf{u}, \mathbf{d})$ to problem \eqref{eq1.1}, \eqref{eq1.2}, and \eqref{eq1.3} obtained by Theorem \ref{thm1} 
satisfies (see Lemma \ref{rem-S+} below)
\begin{align*}
\mathbf{d}({x},t):Q_T\mapsto \mathbb{S}_{+}^{2}.
\end{align*}
This, together with Lemma \ref{vanish}, ensures that \eqref{blowup} never occurs in the interval $(0,T]$.
Hence, we obtain the following theorem.

\begin{theorem}\label{globalW} For  any $T>0$ and a bounded smooth domain $\Omega\subset\mathbb{R}^{2}$, assume that 
\begin{align*}
\mathbf{h}\in L^{2}_tH^{\frac{3}{2}}_x(\Gamma_T,\mathbb{S}_{+}^{2}) 
\ {\rm{and}}\ \partial_{t}\mathbf{h}\in L^{2}_tH^{\frac{3}{2}}_x(\Gamma_T)
\end{align*}
and
$(\mathbf{u}_{0}, \mathbf{d}_{0})\in \mathbf{H}\times H^{1}(\Omega,\mathbb{S}_{+}^{2})$
satisfies the compatibility condition \eqref{comp_cond}.
Then there is a weak solution $(\mathbf{u}, \mathbf{d}):Q_T\mapsto \mathbb{R}^{2}\times\mathbb{S}_{+}^{2}$ of 
the system \eqref{eq1.1} with the initial and boundary conditions \eqref{eq1.2} and \eqref{eq1.3} such that
\begin{align*}
&\mathbf{u}\in L^{\infty}([0,T], \mathbf{H})\cap L^{2}([0,T], \mathbf{V}),\\
&\mathbf{d} \in L^{\infty}_tH^{1}_x(Q_T,\mathbb{S}_{+}^{2})\cap L^2_tH^2_x(Q_T,\mathbb S^2_+),
\end{align*}
and $({\bf u}, {\bf d})\in C^\infty(Q_T)\cap C^{\alpha, \frac{\alpha}2}(\overline\Omega \times (0,T))$ 
for any $\alpha\in (0,1)$. 
In particular,  for $\varepsilon_{1}$ given by Theorem \ref{thm1} there exists $r_0>0$ such that

\begin{align}\label{small_cond}
\sup_{(x,t)\in\overline{\Omega}\times[0,T]}\int_{\Omega\cap B_{r_0}(x)}
(|\mathbf{u}|^{2}+|\nabla\mathbf{d}|^{2})(y,t)\,dy < \varepsilon_{1}^2.
\end{align}
\end{theorem}

\medskip
The smoothness of a weak solution $({\bf u}, {\bf d})$ to \eqref{eq1.1} and \eqref{eq1.2} relies on  both the interior and
boundary $\varepsilon_1$-regularity properties for $({\bf u}, {\bf d})$, provided 
${\bf d}\in L^2_tH^2_x(Q_T)$
and the condition \eqref{small_cond} holds.  This will be discussed in the following subsection. 

\subsection{Regularity of weak solutions}
In this subsection, we will show both the interior and boundary regularity
for weak solutions $({\bf u}, {\bf d})$ to \eqref{eq1.1} and \eqref{eq1.2} that satisfies 
${\bf d}\in L^2_tH^2_x(Q_T)$
and  \eqref{small_cond}.

\begin{theorem}\label{thm3} For  a $T>0$ and a bounded smooth domain $\Omega\subset\mathbb{R}^{2}$,
assume $\mathbf{h}\in L^{2}_tH^{\frac{3}{2}}_x(\Gamma_T,\mathbb{S}^{2})$ and 
$\partial_{t}\mathbf{h}\in L^{2}_tH^{\frac{3}{2}}_x(\Gamma_T, \mathbb{S}^{2})$.
If $(\mathbf{u}, {\bf d})\in L^{\infty}([0,T],\mathbf{H})\cap L^{2}([0,T],\mathbf{V})\times
L^{\infty}_tH^{1}_x(Q_T, \mathbb{S}^{2})\cap L^{2}_tH^{2}_x(Q_T,\mathbb{S}^{2})$ 
is a weak solution of the system \eqref{eq1.1} with the boundary condition \eqref{eq1.2},
then $(\mathbf{u},\mathbf{d})\in C^\infty(\Omega\times [\delta, T])
\cap C^{\alpha, \frac{\alpha}2}(\overline\Omega\times [\delta, T])$ for any $\alpha\in (0, 1)$ and $0<\delta<T$. 
\end{theorem}

In order to prove Theorem \ref{thm3} and the existence of short time smooth solutions 
to \eqref{eq1.1}, \eqref{eq1.2}, and \eqref{eq1.3}, we recall the definition of H\"older spaces
in $Q_T$. First, define the parabolic distance in $Q_T$ by
$$\delta({z}_{1}, {z}_{2})=|{x}_{1}-{x}_{2}|+\sqrt{|t_{1}-t_{2}|},
\ \ {z}_{i}=({x}_{i},t_{i})\in Q_{T}, \ i=1,2.$$
For  $\alpha\in(0,1]$ and $U\subset Q_T$, a continuous function $f: U\mapsto \mathbb R$  belongs to
the H\"older space $C^{\alpha, \frac{\alpha}2}(U)$, if $\displaystyle\|f\|_{C^{\alpha, \frac{\alpha}2}(U)}
=\|f\|_{C^0(U)}+ \big[f\big]_{C^{\alpha, \frac{\alpha}2}(U)}<\infty$, where 
\begin{align*}
\big[f\big]_{C^{\alpha,\frac{\alpha}{2}}(U)}\! =&\sup_{{z}_{1},{z}_{2}\in U,
{z}_{1}\neq {z}_{2}} \frac{|f({z}_{1})-f({z}_{2})|}{\delta({z}_{1}, {z}_{2})^{\alpha}}.
\end{align*}
For any positive integer $k\ge 1$, a continuous function $f:U\mapsto \mathbb R$ belongs to $C^{k+\alpha, \frac{k+\alpha}2}(U)$,  if
\begin{align*}
\big\|f\big\|_{C^{k+\alpha,\frac{k+\alpha}{2}}(U)}=\sum_{0\leq r+2s\leq k} \|\partial_{t}^{s}\partial_{{x}}^{r} f\|_{C^{0}(U)}
+\big[f\big]_{C^{k+\alpha,\frac{k+\alpha}{2}}(U)}<\infty,
\end{align*}
where 
\begin{align*}
\big[f\big]_{C^{k+\alpha,\frac{k+\alpha}{2}}(U)}\!
 =&\!
 \begin{cases}\displaystyle
\!\sum_{r+2s=k} [\partial_{t}^{s}\partial_{{x}}^{r} f]_{C^{\alpha,\frac{\alpha}{2}}(U)}, \ k \text{ is even},\\
\displaystyle\!\sum_{r+2s=k} [\partial_{t}^{s}\partial_{{x}}^{r} f]_{C^{\alpha,\frac{\alpha}{2}}(U)}
\!+\!\sum_{r+2s=k-1} [\partial_{t}^{s}\partial_{{x}}^{r} f]_{C_{t}^{\frac{1+\alpha}{2}}(U)},\\
\ \ k \text{ is odd},
 \end{cases}
\end{align*}
and
$$\big[f\big]_{C^{\frac{1+\alpha}2}_{t}(U)}
=\sup_{(x,t_1), (x,t_2)\in U,  t_{1}\neq t_{2}}
\frac{|f({x},t_{1})-f({x},t_{2})|}{|t_{1}-t_{2}|^{\frac{1+\alpha}{2}}}.$$

\medskip

For $z_0=({x}_{0},t_0)\in \Omega\times (0,T]$ and $0<r<\min\{\sqrt{t_0}, {\rm{dist}}(x_0,\Gamma)\}$, set
\begin{align*}
B_{r}({x}_{0})=\big\{{x}\in \mathbb{R}^{2}\, \mid\, |{x}-{x}_{0}|\leq r\big\}
,\  Q_{r}({z}_{0}) = B_{r}({x}_{0})\times[t_{0}-r^{2},t_{0}],
\end{align*}
and the parabolic boundary of $Q_{r}({z}_{0})$ by
\begin{align*}
\partial_p Q_{r}({z}_{0})=\big(B_{r}({x}_{0})\times\{t_{0}-r^{2}\}\big)\cup 
\big(\partial B_{r}({x}_{0})\times[t_{0}-r^{2},t_{0}]\big).
\end{align*}
Denote $B_{r}({0}),Q_{r}({0},0)$ and $\partial_p Q_{r}({0},0)$ 
by  $B_{r}, Q_{r}$ and $\partial_p Q_{r}$ respectively, if $z_0=(0,0)$.
For $f\in L^{1}(Q_{r}({z}_{0}))$, denote by
\begin{align*}
&f_{{z}_{0},r}=\frac{1}{|Q_{r}({z}_{0})|} \int_{Q_{r}({z}_{0})}
f({x},t), \\
&f_{{x}_{0},r}(t)=\frac{1}{|B_{r}({x}_{0})|} \int_{B_{r}({x}_{0})}
f({x},t), \ t \in [t_0-r^2, t_0],
\end{align*}
as the average of $f$ over $Q_{r}({z}_{0})$ and $B_{r}({x}_{0})$ respectively.
\medskip

For $p,q\in (1, \infty)$ and $U\subset Q_T$, 
define $W^{1,0}_{p,q}(U)
=L^{q}_t W^{1,p}_x(U)$, with the norm
\begin{align*}
\big\|f\big\|_{W^{1,0}_{p,q}(U)}
=\big\|f\big\|_{L^{q}_tL^{p}_x(U)}
+\big\|\nabla f\big\|_{L^{q}_t L^{p}_x(U)},
\end{align*}
and
$$W^{2,1}_{p,q}(U)
=\big\{f\in W^{1,0}_{p,q}(U)\ | \ \nabla^{2}f,\ \partial_{t}f\in L^{q}_tL^{p}_x(U)\big\},$$ 
with the norm
\begin{align*}
\|f\|_{W^{2,1}_{p,q}(U)}
=\|f\|_{W^{1,0}_{p,q}(U)}+[f]_{W^{2,1}_{p,q}(U)}
\end{align*}
where
\begin{align*}
\big[f\big]_{W^{2,1}_{p,q}(U)}=\|\nabla^{2}f\|_{L^{q}_tL^{p}_x(U)}
+\|\partial_{t} f\|_{L^{q}_tL^{p}_x(U)}.
\end{align*}
For $p=q$, denote $L^{p}(U)=L^{p}_tL^{p}_x(U)$ and
$W^{r,s}_{p}(U)=W^{r,s}_{p,p}(U)$. 

\medskip

We begin with an interior $\varepsilon_0$-regularity result, whose proof follows exactly from 
Lin-Lin-Wang \cite{LLW} Lemma 2.1.

\begin{lemma}\label{lem4}
For any $\alpha\in(0,1)$, there exists $\varepsilon_{0}>0$ such that for ${z}_{0}=({x}_{0},t_{0})\in Q_T$ 
and $0<r<\min\{\sqrt{t_0}, {\rm{dist}}(x_0,\Gamma)\}$, 
if $(\mathbf{u},\mathbf{d})\in W^{1,0}_{2}(Q_T, \mathbb{R}^{2}\times\mathbb{S}^{2})$, 
$P\in W^{1,0}_{\frac{4}{3}}(Q_T)$ is a weak solution to \eqref{eq1.1} satisfying
\begin{align}\label{L4small}
\int_{Q_{r}({z}_{0})}(|\mathbf{u}|^{4}+|\nabla \mathbf{d}|^{4})\leq \varepsilon_{0}^{4},
\end{align}
then $(\mathbf{u},\mathbf{d})\in C^{\alpha,\frac{\alpha}{2}}(Q_{\frac{r}{2}}({z}_{0}),
\mathbb{R}^{2}\times \mathbb{S}^{2})$, and there holds that
\begin{align*}
&\big[\mathbf{d}\big]_{C^{\alpha,\frac{\alpha}{2}}(Q_{\frac{r}{2}}({z}_{0}))}
\leq C \big(\|\mathbf{u}\|_{L^{4}(Q_{r}({z}_{0}))}
+\|\nabla \mathbf{d}\|_{L^{4}(Q_{r}({z}_{0}))}\big),
\nonumber\\
&\big[\mathbf{u}\big]_{C^{\alpha,\frac{\alpha}{2}}(Q_{\frac{r}{2}}({z}_{0}))}
\leq C \big(\|\mathbf{u}\|_{L^{4}(Q_{r}({z}_{0}))}
+\|\nabla \mathbf{d}\|_{L^{4}(Q_{r}({z}_{0}))}
+ \|\nabla P\|_{L^{\frac{4}{3}}(Q_{r}({z}_{0}))}\big).
\end{align*}
\end{lemma}

\medskip

For $r>0$ and ${z}_{0}=({x}_{0},t_{0})$ with ${x}_{0}\in\Gamma$ and $t_{0}>0$, 
set 
$$B_{r}^{+}({x}_{0})=B_{r}({x}_{0})\cap \Omega, \ Q_{r}^{+}({z}_{0})=B_{r}^{+}({x}_{0})\times [t_{0}-r^{2},t_{0}],$$ 
and
$$\Gamma_{r}({x}_{0})=\partial B_{r}^{+}({x}_{0})\cap \Gamma,\ \ 
S_{r}^{+}({x}_{0})=\partial B_{r}^{+}({x}_{0})
\cap \Omega$$
so that
$$\partial B_{r}^{+}({x}_{0})\!=\!\Gamma_{r}({x}_{0})\cup S_{r}^{+}({x}_{0}),$$
and
\begin{align*}
\partial_p Q_{r}^{+}({z}_{0})\!=\!\big(\partial B_{r}^{+}({x}_{0})\!\times\! [t_{0}\!-\!r^{2},t_{0}]\big)\cup 
\big(B_{r}^{+}({x}_{0})\!\times\!\{t_{0}\!-\!r^{2}\}\big).
\end{align*}
If $({x}_{0},t_{0})= ({0},0)$, simply denote  
$$B_{r}^{+}=B_{r}^{+}({0}), \ Q_{r}^{+}=Q_{r}^{+}({0},0), \ \Gamma_{r}=\Gamma_{r}({0}), 
\ S_{r}^{+}=S_{r}^{+}({0}),$$
and
$$\partial B_{r}^{+}=\partial B_{r}^{+}({0}),\ \partial_p Q_{r}^{+}=\partial_p Q_{r}^{+}({0},0).$$

\medskip

Next we will establish a corresponding boundary $\varepsilon_0$ regularity for \eqref{eq1.1} and \eqref{eq1.2}, which is
a highly nontrivial extension of \cite{LLW} Lemma 2.2, where a time independent boundary data for ${\bf d}$ is assumed.

\begin{lemma}\label{lem5}
For any $\alpha\in(0,1)$, ${\bf h}\in L^2_tH^\frac32_x(\Gamma_T,\mathbb S^2)$ with
$\partial_t{\bf h}\in L^2_tH^\frac32_x(\Gamma_T)$,  assume that
$({\bf u}, {\bf d})\in W^{1,0}_2(Q_T, \mathbb R^2\times\mathbb S^2)$,
$P\in W^{1,0}_{\frac43}(Q_T)$, is a weak solution of \eqref{eq1.1} and \eqref{eq1.2}.
Then there exist $r_0\in (0, \sqrt{t_0})$ depending on $\Gamma$
and $\varepsilon_{1}>0$ depending on $\alpha$ and
$\|(\mathbf{h}, \partial_t {\bf h})\|_{L^2_tH^\frac32_x(\Gamma_T)}$
such that for any ${z}_{0}=({x}_{0},t_{0})\in \Gamma_T$,  if
\begin{align}\label{L4smallb}
\int_{Q_{r_0}^{+}({z}_{0})}(|\mathbf{u}|^{4}+|\nabla \mathbf{d}|^{4})\leq \varepsilon_{1}^{4},
\end{align}
then $(\mathbf{u},\mathbf{d})\in C^{\alpha,\frac{\alpha}{2}}(Q_{\frac{r_0}{2}}^{+}({z}_{0}),
\mathbb{R}^{2}\times \mathbb{S}^{2})$, and there holds that
\begin{align}\label{bdyestimate1}
\big [\mathbf{d}\big]_{C^{\alpha,\frac{\alpha}{2}}(Q_{\frac{r_0}{2}}^{+}({z}_{0}))}
\leq& C \big[\|\mathbf{u}\|_{L^{4}(Q_{r_0}^{+}({z}_{0}))}
+\|\nabla \mathbf{d}\|_{L^{4}(Q_{r_0}^{+}({z}_{0}))}\nonumber\\
&\ +\|(\mathbf{h},r_0^2\partial_t{\bf h})\|_{L^2_tH^\frac32_x(\Gamma_{r_0}(x_0)\times [t_0-r_0^2, t_0])}\big],\\
   \label{bdyestimate2}
\big[\mathbf{u}\big]_{C^{\alpha,\frac{\alpha}{2}}(Q_{\frac{r_0}{2}}^{+}({z}_{0}))}
\leq& C \big[\|\mathbf{u}\|_{L^{4}(Q_{r_0}^{+}({z}_{0}))}
+\|\nabla \mathbf{d}\|_{L^{4}(Q_{r_0}^{+}({z}_{0}))}\nonumber\\
&+ \|\nabla P\|_{L^{\frac{4}{3}}(Q_{r_0}^{+}({z}_{0}))}\nonumber\\
&
+\|(\mathbf{h},r_0^2\partial_t{\bf h})\|_{L^2_tH^\frac32_x(\Gamma_{r_0}(x_0)\times [t_{0}-r_0^{2},t_{0}])}\big].
\end{align}
\end{lemma}

\begin{proof} The proof of Lemma \ref{lem5} is more delicate than
\cite{LLW} Lemma 2.2, because the boundary value ${\bf h}$ is time dependent.
Here we will give a detailed argument.

Choosing a sufficiently small $r_0>0$ and applying the standard boundary flatten technique, 
we may, for simplicity, assume that $x_0=0, t_0=1, r_0<1$ so that
$$\Omega\cap B_{r_0}(0)=\mathbb R^2_+\cap B_{r_0}(0)
=B_{r_0}^+, \  {\rm{and}}\ Q_{r_0}^+(0,1)=B_{r_0}^+\times [1-r_0^2, 1].$$ 

First, observe that \eqref{L4smallb} and \eqref{eq1.1} imply 
$$\partial_t {\bf d}-\Delta{\bf d}=|\nabla {\bf d}|^2{\bf d}-{\bf u}\cdot\nabla{\bf d}\in L^2(Q_{r_0}^+).$$
Hence, by the $W^{2,1}_2$-theory on parabolic equations,  we have that
$\partial_t {\bf d}\in L^2(Q_{\frac{3r_0}4}^+)$, ${\bf d}\in L^2_tH^2_x(Q_{\frac{3r_0}4}^+)$, and
\begin{align}\label{L2H2}
&\big\|(\partial_t {\bf d}, \nabla^2{\bf d})\big\|_{L^2(Q_{\frac{3r_0}4}^+)}\nonumber\\
&\le C\big[\big\|({\bf u}, \nabla{\bf d})\big\|_{L^4(Q_{r_0}^+)}
+\big\|({\bf h}, r_0^2\partial_t{\bf h})\big\|_{L^2_tH^\frac32_x(\Gamma_{r_0})} \big].
\end{align}
For ${z}_{1}=({x}_{1},t_{1})\in \Gamma_{\frac{r_0}{2}}\times[1-\frac{r_0^2}{4}, 1]$, $0<r\leq\frac{r_0}{4}$, 
let $\mathbf{d}^{1}:Q_{r}^{+}({z}_{1})\mapsto \mathbb{R}^{3}$ solve
\begin{align}\label{d1}
\begin{cases}
\partial_{t}\mathbf{d}^{1}-\Delta \mathbf{d}^{1}=0 & \text{ in } Q_{r}^{+}({z}_{1}),\\
\mathbf{d}^{1}=\mathbf{h}& \text{ on } \Gamma_{r}({x}_{1})\times [t_{1}-r^{2},t_{1}],\\
\mathbf{d}^{1}=\mathbf{d}& \text{ on } \partial_p Q_r^+(z_1)\setminus (\Gamma_r(x_1)\times [t_1-r^2, t_1]).
\end{cases}
\end{align}
Then $\mathbf{d}^{2}={\bf d}-{\bf d}^1:Q_{r}^{+}({z}_{1})\mapsto \mathbb{R}^{3}$ solves
\begin{align}\label{d2}
\begin{cases}
\partial_{t}\mathbf{d}^{2}-\Delta \mathbf{d}^{2}=-\mathbf{u}\cdot\nabla \mathbf{d}+|\nabla \mathbf{d}|^{2}\mathbf{d} &\text{ in } Q_{r}^{+}({z}_{1}),\\
\qquad\qquad\mathbf{d}^{2}={0}& \text{ on } \partial_p Q_{r}^{+}({z}_{1}).
\end{cases}
\end{align}
From \eqref{L4smallb} and an argument similar to \eqref{L2H2},
we have that $\partial_{t}\mathbf{d}^{2},\nabla^{2}\mathbf{d}^{2}\in L^{2}(Q_{r}^{+}({z}_{1}))$.
Hence, by multiplying $\eqref{d2}_{1}$ by $\Delta \mathbf{d}^{2}$ and integrating over $B_{r}^{+}({x}_{1})$, 
we obtain that
\begin{align*}
\frac{1}{2}\frac{d}{dt}\int_{B_{r}^{+}({x}_{1})} |\nabla \mathbf{d}^{2}|^{2} +
\int_{B_{r}^{+}({x}_{1})}|\Delta \mathbf{d}^{2}|^{2}=\int_{B_{r}^{+}({x}_{1})}(-\mathbf{u}\cdot\nabla \mathbf{d}+|\nabla \mathbf{d}|^{2}\mathbf{d})\cdot\Delta \mathbf{d}^{2}.
\end{align*}
Integrating over $t\in [t_{1}-r^{2},t_{1}]$ and  applying H\"{o}lder's inequality, this yields
\begin{align*}
&\|\nabla \mathbf{d}^{2}\|_{L^\infty_tL^{2}_x(Q_{r}^{+}({z}_{1}))}^{2}
+\|\Delta \mathbf{d}^{2}\|_{L^{2}(Q_{r}^{+}({z}_{1}))}^{2}\\
&\leq C\big[\||\mathbf{u}||\nabla\mathbf{d}|\|_{L^{2}(Q_{r}^{+}({z}_{1}))}^{2}
+\|\nabla \mathbf{d}\|_{L^{4}(Q_{r}^{+}({z}_{1}))}^{4}\big]\\
&\leq C\big[\|\mathbf{u}\|_{L^{4}(Q_{r}^{+}({z}_{1}))}^{2}
+\|\nabla\mathbf{d}\|_{L^{4}(Q_{r}^{+}({z}_{1}))}^{2}\big]
\|\nabla \mathbf{d}\|_{L^{4}(Q_{r}^{+}({z}_{1}))}^{2}.
\end{align*}
This,  together with the Ladyzhenskaya inequality (see Lemma \ref{lem14} below),  yields that
\begin{align}\label{d2-est}
&\|\nabla \mathbf{d}^{2}\|_{L^{4}(Q_{r}^{+}({z}_{1}))}^{4}+\|\nabla \mathbf{d}^{2}\|_{L^\infty_tL^{2}_x(Q_{r}^{+}({z}_{1}))}^{4}
+\|\Delta \mathbf{d}^{2}\|_{L^{2}(Q_{r}^{+}({z}_{1}))}^{4}\nonumber\\
&\leq C\big[\|\mathbf{u}\|_{L^{4}(Q_{r}^{+}({z}_{1}))}^{4}
+\|\nabla\mathbf{d}\|_{L^{4}(Q_{r}^{+}({z}_{1}))}^{4}\big]
\|\nabla \mathbf{d}\|_{L^{4}(Q_{r}^{+}({z}_{1}))}^{4}\nonumber\\
&\leq C\varepsilon_1^4\|\nabla \mathbf{d}\|_{L^{4}(Q_{r}^{+}({z}_{1}))}^{4}. 
\end{align}
Now we  estimate ${\bf d}^1$. First observe that ${\bf h}, \partial_t{\bf h}
\in L^2_tH^\frac32_x(\Gamma_T)$ implies that ${\bf h}\in L^\infty_tH^\frac32_x(\Gamma_T)$
and
\begin{equation}\label{h-bound1}
\|{\bf h}\|_{L^\infty_tH^\frac32_x(\Gamma_T)}
\leq C\big(\frac{1}{T}+ \|\partial_t {\bf h}\|_{L^2_tH^\frac32_x(\Gamma_T)}\big)\|{\bf h}\|_{L^2_tH^\frac32_x(\Gamma_T)}.
\end{equation}
This, combined with $\partial_t {\bf h}\in L^2_tH^\frac32_x(\Gamma_T)$ and the Sobolev embedding theorem,
implies that ${\bf h}\in C^{\alpha, \frac{\alpha}2}(\Gamma_T)$ for any $\alpha\in (0,1)$, and
\begin{equation}\label{h-bound2}
\big\|{\bf h}\big\|_{C^{\alpha, \frac{\alpha}2}(\Gamma_T)}
\leq C(\alpha)\big(\frac{1}{T}+ \|\partial_t {\bf h}\|_{L^2_tH^\frac32_x(\Gamma_T)}\big)
\big(1+\|{\bf h}\|_{L^2_tH^\frac32_x(\Gamma_T)}\big).
\end{equation}
It follows from \eqref{d1}, \eqref{h-bound2} and the boundary regularity theory for parabolic equations that 
${\bf d}^1\in C^{\alpha,\frac{\alpha}2}(Q_{\frac{3r}4}^+(z_1))$ and
\begin{align}\label{d1-bound1}
\big\|{\bf d}^1\big\|_{C^{\alpha,\frac{\alpha}2}(Q_{\frac{3r}4}^+(z_1))}
&\leq C\big(\big\|{\bf h}\big\|_{C^{\alpha, \frac{\alpha}2}(Q_r^+(z_1))}+\|\nabla{\bf d}\|_{L^4(Q_r^+(z_1))}\big)\nonumber\\
&\leq C\big(\alpha, T, \varepsilon_1, \|({\bf h}, \partial_t{\bf h})\|_{L^2_tH^\frac32_x(\Gamma_T)}\big). 
\end{align}
For $t\in (0,T]$, let ${\bf h}_E(\cdot, t):\Omega\mapsto\mathbb R^3$ be the harmonic extension of
${\bf h}(\cdot, t):\Gamma\mapsto\mathbb S^2$. Then we have that
\begin{align*}
&\|{\bf h}_E\|_{L^2_tH^2_x(Q_T)}\le C \|{\bf h}\|_{L^2_tH^\frac32_x(\Gamma_T)},\\
&\|\partial_t{\bf h}_E\|_{L^2_tH^2_x(Q_T)}\le C \|\partial_t{\bf h}\|_{L^2_tH^\frac32_x(\Gamma_T)}.
\end{align*}
Furthermore, \eqref{h-bound1} and \eqref{h-bound2} imply that
${\bf h}_E\in L^\infty_tH^2_x(Q_T)\cap C^{\alpha, \frac{\alpha}2}(Q_T)$ for any $\alpha\in (0,1)$, and
\begin{align}\label{he-bound}
&\max\big\{\big\|{\bf h}_E\big\|_{L^\infty_tH^2_x(Q_T)},\ \big\|{\bf h}_E\big\|_{C^{\alpha, \frac{\alpha}2}(Q_T)}\big\}
\nonumber\\
&\leq C(\alpha)\big(\frac{1}{T}+ \|\partial_t {\bf h}\|_{L^2_tH^\frac32_x(\Gamma_T)}\big)
\big(1+\|{\bf h}\|_{L^2_tH^\frac32_x(\Gamma_T)}\big)\nonumber\\
&\leq C\big(\alpha, T, \|({\bf h}, \partial_t{\bf h})\|_{L^2_tH^\frac32_x(\Gamma_T)}\big).
\end{align}
Observe that ${\bf d}^1-{\bf h}_E$ solves
\begin{equation}\label{d1he1}
\partial_t({\bf d}^1-{\bf h}_E)-\Delta ({\bf d}^1-{\bf h}_E)=-\partial_t {\bf h}_E
\ {\rm{in}}\ Q_r^+(z_1).
\end{equation}
Let $\eta\in C_0^\infty(B_{\frac{3r}4}(x_1))$ be a cut-off function of $B_{\frac{r}2}(x_1)$, i.e.,
$0\le\eta\le 1$, $\eta=1$ in $B_{\frac{r}2}(x_1)$,  and $|\nabla\eta|\le 8r^{-1}$.
Multiplying \eqref{d1he1} by $({\bf d}^1-{\bf h}_E)\eta^2$, integrating over $B_{r}^+(x_1)$, and 
applying \eqref{d1-bound1} and \eqref{he-bound}, we obtain
\begin{align*}
&\frac{d}{dt}\int_{B_r^+(x_1)}|{\bf d}^1-{\bf h}_E|^2\eta^2
+2\int_{B_r^+(x_1)}|\nabla({\bf d}^1-{\bf h}_E)|^2\eta^2\\
&=-2\int_{B_r^+(x_1)} \langle\nabla({\bf d}^1-{\bf h}_E), {\bf d}^1-{\bf h}_E\rangle \nabla\eta^2
-2\int_{B_r^+(x_1)} \langle\partial_t{\bf h}_E, {\bf d}^1-{\bf h}_E\rangle \eta^2\\
&\le \int_{B_r^+(x_1)}|\nabla({\bf d}^1-{\bf h}_E)|^2\eta^2
+\int_{B_r^+(x_1)}(4|{\bf d}^1-{\bf h}_E|^2|\nabla\eta|^2+2|\partial_t{\bf h}_E| |{\bf d}^1-{\bf h}_E|)\\
&\le  \int_{B_r^+(x_1)}|\nabla({\bf d}^1-{\bf h}_E)|^2\eta^2+Cr^{2\alpha}
+Cr^\alpha \int_{B_r^+(x_1)}|\partial_t{\bf h}_E|.
\end{align*}
Integrating this inequality over $t\in [t_1-\frac{r^2}{4}, t_1]$ yields
\begin{align}\label{d1he2}
\int_{Q_{\frac{r}{2}}^+(z_1)}|\nabla({\bf d}^1-{\bf h}_E)|^2
&\le Cr^{2+2\alpha}+Cr^\alpha\int_{Q_r^+(z_1)}|\partial_t{\bf h}_E|\nonumber\\
&\le  Cr^{2+2\alpha}+Cr^{2+2\alpha}\|\partial_t{\bf h}_E\|_{L^2_tL^{\frac{2}{1-\alpha}}_x(Q_T)}\nonumber\\
&\le C\big(1+\|\partial_t{\bf h}_E\|_{L^2_tH^2_x(Q_T)}\big)r^{2+2\alpha}\nonumber\\
&\le C\big(1+\|\partial_t{\bf h}\|_{L^2_tH^\frac32_x(\Gamma_T)}\big)r^{2+2\alpha}\le Cr^{2+2\alpha}.
\end{align}
Let $\eta_1\in C^\infty_0(B_{\frac{r}2}(x_1))$ be  a cut-off function of $B_{\frac{3r}8}(x_1)$,
i.e. $\eta_1=1$ in $B_{\frac{3r}8}(x_1)$, and $|\nabla\eta_1|\le 16r^{-1}$.
Multiplying \eqref{d1he1} by $\Delta({\bf d}^1-{\bf h}_E)\eta_1^2$ and integrating over $B_r^+(x_1)$, 
and using $\partial_t {\bf d}^1=\Delta({\bf d}^1-{\bf h}_E)$, we get
\begin{align}\label{d1he3}
&\frac{d}{dt}\int_{B_r^+(x_1)}|\nabla({\bf d}^1-{\bf h}_E)|^2\eta_1^2
+2\int_{B_r^+(x_1)}|\Delta({\bf d}^1-{\bf h}_E)|^2\eta_1^2\nonumber\\
&=2\int_{B_r^+(x_1)}\big(\langle\partial_t{\bf h}_E, \Delta({\bf d}^1-{\bf h}_E)\rangle\eta_1^2\nonumber\\
&\ \ \ -\langle\partial_t({\bf d}^1-{\bf h}_E), \nabla({\bf d}^1-{\bf h}_E)\rangle\nabla\eta_1^2\big)\nonumber\\
&\le \frac12\int_{B_r^+(x_1)}|\Delta({\bf d}^1-{\bf h}_E)|^2\eta_1^2+8\int_{B_r^+(x_1)}|\partial_t {\bf h}_E|^2\eta_1^2
\nonumber\\
&\ \ \ +\frac12\int_{B_r^+(x_1)}|\partial_t {\bf d}^1|^2\eta_1^2
+8\int_{B_r^+(x_1)}|\nabla({\bf d}^1-{\bf h}_E)|^2|\nabla\eta_1|^2\nonumber\\
&\le \int_{B_r^+(x_1)}|\Delta({\bf d}^1-{\bf h}_E)|^2\eta_1^2+C\int_{B_r^+(x_1)}|\partial_t {\bf h}_E|^2\\
&\ \ \ +Cr^{-2}\int_{B_{\frac{r}2}^+(x_1)}|\nabla({\bf d}^1-{\bf h}_E)|^2.\nonumber
\end{align}
By Fubini's theorem, there exists $t_*\in [t_1-\frac{r^2}4, t_1-\frac{r^2}{16}]$ such that 
$$\int_{B_{\frac{r}2}^+(x_1)\times\{t_*\}}|\nabla({\bf d}^1-{\bf h}_E)|^2
\le \frac{32}{r^2}\int_{Q_{\frac{r}2}^+(z_1)}|\nabla({\bf d}^1-{\bf h}_E)|^2\le Cr^{2\alpha}.
$$
This, combined with \eqref{d1he2} and integration of \eqref{d1he3} over $t\in [t_*, t_1]$, yields
\begin{align}\label{d1he4}
\|\nabla({\bf d}^1-{\bf h}_E)\|^2_{L^\infty_tL^2_x(Q_{\frac{3r}8}^+(z_1))}
+\int_{Q_{\frac{3r}8}^+(z_1)}|\Delta({\bf d}^1-{\bf h}_E)|^2\le Cr^{2\alpha}.
\end{align}
This, combined with the Ladyzhenskaya inequality (see Lemma \ref{lem14} below), implies
that
\begin{align}\label{d1he5}
&\int_{Q_{\frac{r}4}^+(z_1)}|\nabla({\bf d}^1-{\bf h}_E)|^4\nonumber\\
&\le C\|\nabla({\bf d}^1-{\bf h}_E)\|^2_{L^\infty_tL^2_x(Q_{\frac{3r}8}^+(z_1))}
\big[\|\nabla({\bf d}^1-{\bf h}_E)\|^2_{L^\infty_tL^2_x(Q_{\frac{3r}8}^+(z_1))}\nonumber\\
&\ \ \ +\int_{Q_{\frac{3r}8}^+(z_1)}|\Delta({\bf d}^1-{\bf h}_E)|^2\big]\nonumber\\
&\le Cr^{4\alpha}.
\end{align}
Since ${\bf h}_E\in L^\infty_tH^2_x(Q_T)$,  we have that  for all $4<p<\infty$,
\begin{align}\label{he-est1}
\int_{Q_r^+(x_1)}|\nabla{\bf h}_E|^4&\le r^2\sup_{t\in [t_1-r^2, t_1]}\int_{B_r^+(x_1)}|\nabla{\bf h}_E|^4\nonumber\\
&\le Cr^{4-\frac{8}{p}} \sup_{t\in [t_1-r^2, t_1]}(\int_{B_r^+(x_1)}|\nabla{\bf h}_E|^p)^{\frac{4}{p}}\nonumber\\
&\le C\big(T, \|({\bf h}, \partial_t{\bf h})\|_{L^2_tH^\frac32_x(\Gamma_T)}\big) r^{4-\frac{8}{p}}\nonumber\\
&\le Cr^{4-\frac{8}{p}}.
\end{align}
Without loss of generality, we may only consider $\alpha\in (\frac12,1)$.
By choosing $p=\frac{2}{1-\alpha}>4$, \eqref{d1he5} and \eqref{he-est1} imply that
\begin{align}\label{d1-est}
\int_{Q_{\frac{r}4}^+(z_1)}|\nabla{\bf d}^1|^4\le Cr^{4\alpha}.
\end{align}
Putting \eqref{d2-est} and \eqref{d1-est} together, we obtain that
\begin{align}\label{L4d-est}
\int_{Q_{\frac{r}4}^+(z_1)}|\nabla\mathbf{d}|^{4}
\leq& Cr^{4\alpha}+C\varepsilon_1^4\int_{Q_r^+(z_1)}
|\nabla \mathbf{d}|^{4}.
\end{align}
It is well known that by iterations of \eqref{L4d-est}, we can conclude that
\begin{align}\label{L4d-est1}
\int_{Q_r^+(z_1)} |\nabla{\bf d}|^4
&\le Cr^{4\alpha}+C(\frac{r}{r_0})^{4\alpha}\int_{Q^+_{r_0}(z_1)}|\nabla {\bf d}|^4\nonumber\\
&\le C(\alpha, \varepsilon_1, r_0) r^{4\alpha}
\end{align} 
holds for all $z_1\in\Gamma_{\frac{r_0}2}\times [1-\frac{r_0^2}4, 1]$ and
$0<r\le\frac{r_0}4$.

\medskip

Next we want to estimate $\|\partial_{t}\mathbf{d}\|_{L^{2}(Q_{\frac{r}2}^{+}({z}_{1}))}$ for
$z_1\in\Gamma_{\frac{r_0}2}\times [1-\frac{r_0^2}4, 1]$ and
$0<r\le\frac{r_0}4$.  To do this, we first observe that \eqref{d1he4}, \eqref{he-bound}, \eqref{d2-est}, together with
\eqref{L4d-est1}, imply that
\begin{eqnarray}\label{W22-d-est}
\int_{Q_{\frac{r}2}^+(z_1)}|\Delta {\bf d}|^2&\le&
\int_{Q_{\frac{r}2}^+(z_1)}|\Delta{\bf d}^2|^2+|\Delta({\bf d}^1-{\bf h}_E)|^2+|\Delta {\bf h}_E|^2\nonumber\\
&\le& C(\alpha,\varepsilon_1, r_0) r^{2\alpha}.
\end{eqnarray}
Hence it follows from the equation of ${\bf d}$ that 
\begin{eqnarray}\label{l2-dt-est}
\int_{Q_{\frac{r}2}^+(z_1)}|\partial_t{\bf d}|^2
&\le&C\int_{Q_{\frac{r}2}^+(z_1)}(|{\bf u}\cdot\nabla {\bf d}|^2+|\Delta {\bf d}|^2+|\nabla{\bf d}|^4)\nonumber\\
&\le&C(\int_{Q_{\frac{r}2}^+(z_1)}|{\bf u}|^4)^\frac12
(\int_{Q_{\frac{r}2}^+(z_1)}|\nabla {\bf d}|^4)^\frac12\nonumber\\
&&\ + C\int_{Q_{\frac{r}2}^+(z_1)}(|\Delta {\bf d}|^2+|\nabla{\bf d}|^4)\nonumber\\
&\le& C(\alpha,\varepsilon_1, r_0) r^{2\alpha}.
\end{eqnarray}

Putting \eqref{L4d-est1} together with \eqref{l2-dt-est} and applying H\"older's inequality, we conclude
that
\begin{align}\label{morrey-d}
\frac{1}{r^{2}} \int_{Q_{r}^{+}({z}_{1})} (|\nabla\mathbf{d}|^{2}+r^{2}|\partial_{t}\mathbf{d}|^{2})\leq 
Cr^{2\alpha}
\end{align}
holds for any ${z}_{1}\in \Gamma_{\frac{r_0}{2}}^{+}({0})\times [1-\frac{r_0^2}{4},1]$ and $0<r\leq \frac{r_0}{4}$. 

It is clear that \eqref{morrey-d}, combined with the interior regularity Lemma \ref{lem4} and the parabolic Morrey's decay Lemma 
(see, e.g., \cite{CL1993}), yields that $\mathbf{d}\in C^{\alpha,\frac{\alpha}{2}}(Q_{\frac{r_0}{2}}^{+}(z_0),
\mathbb{S}^{2})$ and the estimate \eqref{bdyestimate1} holds.  On the other hand, the boundary H\"older regularity
of ${\bf u}$ and the estimate \eqref{bdyestimate2} can be established exactly as in \cite{LLW} Lemma 2.2. 
Thus  the proof of Lemma \ref{lem5} is complete.
\end{proof}

\medskip
\noindent \textbf{Proof of Theorem \ref{thm3}.} 
Since $\mathbf{u}\in L^{\infty}([0,T], \mathbf{H})\cap L^{2}([0,T], \mathbf{V})$, it follows from 
Ladyzhenskaya's inequality that
$\mathbf{u}\in L^{4}(Q_T,\mathbb{R}^{2}).$
For ${\bf d}$, it follows from $\mathbf{d}\in L^{2}([0,T], H^{2}(\Omega))$ and
$|\mathbf{d}|=1$ that
\begin{align*}
|\nabla \mathbf{d}|^{2}=-\mathbf{d}\cdot\Delta\mathbf{d}\in L^2(Q_T)
\end{align*}
so that $|\nabla \mathbf{d}|\in L^{4}(Q_T)$ and
$\mathbf{u}\cdot\nabla \mathbf{u}+\nabla\cdot(\nabla\mathbf{d}\odot\nabla \mathbf{d})\in L^{\frac{4}{3}}(Q_T)$.
From this and Lemma \ref{lem13} below, we conclude 
that $\nabla P\in L^{\frac{4}{3}}(Q_T)$. By the absolute continuity of $L^{4}$-norm of $(\mathbf{u},\nabla \mathbf{d})$,
we can apply both Lemma \ref{lem4} and Lemma \ref{lem5} to show that for any $\alpha\in (0,1)$,
\begin{align*}
(\mathbf{u},\mathbf{d})\in C^{\alpha,\frac{\alpha}{2}}\big(\overline\Omega\times [\delta, T], 
\mathbb{R}^{2}\times\mathbb{S}^{2}\big) 
\end{align*}
holds for any $0<\delta<T$. 
By employing the standard higher order regularity theory, we can get the interior smoothness
of $({\bf u}, {\bf d})$ in $Q_T$. This completes the proof. 
\qed

\medskip

\subsection{Existence of short time smooth solutions}

In this subsection, we will establish the existence of a unique short time  
smooth solution to \eqref{eq1.1}--\eqref{eq1.3} 
for any smooth initial and boundary data. More precisely, we have

\begin{theorem}\label{thm6}
For any bounded, smooth domain $\Omega\subset\mathbb{R}^{2}$,  $0<T<\infty$ and $\alpha\in (0,1)$, 
let $\mathbf{h}\in C^{2+\alpha,1+\frac{\alpha}{2}}(\Gamma_{T}, \mathbb{S}^{2})$, 
$\mathbf{u}_{0}\in C^{2, \alpha}(\overline{\Omega},\mathbb{R}^{2})$ with $\nabla \cdot\mathbf{u}_{0}=0$, 
$\mathbf{d}_{0}\in C^{2, \alpha}(\overline{\Omega}, \mathbb{S}^{2})$ satisfying 
the compatibility condition \eqref{comp_cond}.
Then there exist $0<T_{*}\leq T$ depending on $\|\mathbf{u}_{0}\|_{C^{2,\alpha}(\Omega)}$, 
$\|\mathbf{d}_{0}\|_{C^{2,\alpha}(\Omega)}$ and $\|\mathbf{h}\|_{C^{2+\alpha,1+\frac{\alpha}{2}}(\Gamma_{T})}$ 
and a unique solution $(\mathbf{u},\mathbf{d})$ to the system \eqref{eq1.1} -\eqref{eq1.3} such that
\begin{align*}
(\mathbf{u},\mathbf{d})\in C^{2+\alpha,1+\frac{\alpha}{2}}(\overline{Q_{T_{*}}},
\mathbb{R}^{2}\times \mathbb{S}^{2}).
\end{align*}
Furthermore, $(\mathbf{u},\mathbf{d})\in C^{\infty}(Q_{T_*}, \mathbb{R}^{2}\times \mathbb{S}^{2})$.
\end{theorem}

\begin{proof}
The proof is based on the contraction mapping principle. Let $\mathbf{h}_{P}:Q_T\mapsto\mathbb R^3$ solve
\begin{align}
 \begin{cases}
 \partial_t\mathbf{h}_P-\Delta \mathbf{h}_P={0},\qquad &\text{ in } Q_T,\\
 \mathbf{h}_P=\mathbf{h},\qquad \qquad  &\text{ on } \Gamma_T,\\
 \mathbf{h}_P= \mathbf{d}_{0}, \qquad \quad\ \  &\text{ in } \Omega\times\{0\},
 \end{cases}\label{LP0}
 \end{align}
 For $0<T_{*}\leq T$ and $K>0$ to be chosen later, define
\begin{align}
\mathfrak{X}(T_*, K)=&\Big\{(\mathbf{v},\mathbf{f})\in C^{2+\alpha,1+\frac{\alpha}{2}}(\overline{Q_{T_{*}}},
\mathbb{R}^{2}\times \mathbb{R}^{3}): \ (\mathbf{v},\mathbf{f})|_{t=0}=(\mathbf{u}_{0}, \mathbf{d}_{0}),
\nonumber\\
&\qquad\ \ \ \nabla\cdot \mathbf{v}=0,\  
\big\|(\mathbf{v}-\mathbf{u}_{0},
\mathbf{f}-\mathbf{h}_{P})\big\|_{C^{2+\alpha,1+\frac{\alpha}{2}}(Q_{T_*})}\leq K\Big\},
\end{align}
which is equipped with the norm
\begin{align*}
\big\|(\mathbf{v},\mathbf{f})\big\|_{\mathfrak{X}(T_*, K)}:
=\big\|(\mathbf{v}, \mathbf{f})\big\|_{C^{2+\alpha,1+\frac{\alpha}{2}}(Q_{T_*})}, 
\quad \forall (\mathbf{v},\mathbf{f})\in \mathfrak{X}(T_*, K).
\end{align*}
\medskip
Note that $\big(\mathfrak{X}(T_*,K), \|\cdot\|_{\mathfrak{X}(T_*, K)}\big)$ is a Banach space. 

We now define the operator $L$ by letting
\begin{align*}
(\mathbf{u},\mathbf{d}):=L(\mathbf{v},\mathbf{f}):\mathfrak{X}(T_*,K)\mapsto C^{2+\alpha,1+\frac{\alpha}{2}}(\overline{Q_{T_{*}}},\mathbb{R}^{2}\times \mathbb{R}^{3})
\end{align*}
be a unique solution to the following non-homogeneous linear system:
\begin{align}\label{NLC2}
\begin{cases}
\partial_{t}\mathbf{u}-\Delta\mathbf{u}+\nabla P=-\mathbf{v}\cdot\nabla \mathbf{v}-\nabla\cdot(\nabla \mathbf{f}\odot\nabla \mathbf{f}) \qquad\quad &\text{in }Q_{T_{*}},\\
\nabla \cdot\mathbf{u}=0\quad &\text{in }Q_{T_{*}},\\
\partial_{t}\mathbf{d}-\Delta\mathbf{d}=|\nabla \mathbf{f}|^{2}\mathbf{f}-\mathbf{v}\cdot\nabla\mathbf{f}\quad &\text{in }Q_{T_{*}},\\
(\mathbf{u}, \mathbf{d})=(0, \mathbf{h}) \quad&\text{on }\Gamma_{T_{*}},\\
(\mathbf{u}, {\bf d})=(\mathbf{u}_{0}, \mathbf{d}_{0})
\quad &\text{in }\Omega\times\{0\}.
\end{cases}
\end{align}
It follows from Lemmas \ref{lem7} and \ref{lem8} below that if we choose
a sufficiently small $T_{*}\in (0, T]$ and a sufficiently large $K>0$, 
then $L:\mathfrak{X}(T_*, K)\mapsto \mathfrak{X}(T_*, K)$ is a contractive map
so that there is a unique fixed point $({\bf u}, {\bf d})\in \mathfrak{X}(T_*, K)$ of $L$,
i.e. $({\bf u}, {\bf d})=L({\bf u}, {\bf d})$. 
Moreover, it follows from Lemma \ref{lem9} that $|\mathbf{d}|=1$ in $Q_{T_*}$. 
Thus the conclusions of Theorem \ref{thm6} hold, if we can prove Lemma \ref{lem7},
Lemma \ref{lem8}, and Lemma \ref{lem9} below.
\end{proof}

\begin{lemma}\label{lem7}
There exist $0<T_{*}\leq T$ and $K>0$ such that $L:\mathfrak{X}(T_*, K)\mapsto \mathfrak{X}(T_*, K)$
is a bounded operator.
\end{lemma}

\begin{proof}
For any $(\mathbf{v},\mathbf{f})\in\mathfrak{X}(T_*, K)$, 
set $(\mathbf{u},\mathbf{d})=L(\mathbf{v},\mathbf{f})$. 
Let $C_{0}>0$ denote a constant depending on $\|\mathbf{u}_{0}\|_{C^{2+\alpha}(\Omega)}$, $\|\mathbf{d}_{0}\|_{C^{2+\alpha}(\Omega)}$ and $\|\mathbf{h}\|_{C^{2+\alpha,1+\frac{\alpha}{2}}(\Gamma_{T})}$ .
\smallbreak

By the Schauder theory to the equation \eqref{LP0},
$\mathbf{h}_{P}$ satisfies
\begin{align} \label{estimatehp}
\|\mathbf{h}_{P}\|_{C^{2+\alpha,1+\frac{\alpha}{2}}(Q_{T_*})}
\leq C\big( \|\mathbf{d}_{0}\|_{C^{2+\alpha}(\Omega)}
+
\|\mathbf{h}\|_{C^{\alpha,\frac{\alpha}{2}}(\Gamma_{T_{*}})}\big),
\end{align}
Set $\widetilde{\mathbf{d}}=\mathbf{d}-\mathbf{h}_{P}$. Then \eqref{NLC2} can be rewritten as
\begin{align}\label{NLC3}
\begin{cases}
\partial_{t}\mathbf{u}-\Delta\mathbf{u}+\nabla P=-\mathbf{v}\cdot\nabla \mathbf{v}-\nabla\cdot(\nabla \mathbf{f}\odot\nabla \mathbf{f}) \qquad\quad &\text{in }Q_{T_{*}},\\
\nabla \cdot\mathbf{u}=0\quad &\text{in }Q_{T_{*}},\\
\partial_{t}\widetilde{\mathbf{d}}-\Delta\widetilde{\mathbf{d}}=|\nabla \mathbf{f}|^{2}\mathbf{f}-\mathbf{v}\cdot\nabla\mathbf{f}\quad &\text{in }Q_{T_{*}},\\
(\mathbf{u},  \widetilde{\mathbf{d}})=(0, 0) \quad&\text{on }\Gamma_{T_{*}},\\
(\mathbf{u}, \widetilde{\mathbf{d}})=({\bf u}_0,0)
\quad &\text{in }\Omega\times\{0\}.
\end{cases}
\end{align}
Assume $K\geq C_{0}$. By the Schauder theory of parabolic systems, we have
\begin{align}\label{widetilded}
\big\|\widetilde{\mathbf{d}}\big\|_{C^{2+\alpha,1+\frac{\alpha}{2}}(Q_{T_*})}
\leq C\big(
\|\mathbf{v}\cdot\nabla \mathbf{f}\|_{C^{\alpha,\frac{\alpha}{2}}(Q_{T_*})}
+\||\nabla\mathbf{f}|^{2} \mathbf{f}\|_{C^{\alpha,\frac{\alpha}{2}}(Q_{T_*})}\big).
\end{align}
To estimate the first term in the right-hand side, let $\widetilde{\mathbf{f}}=\mathbf{f}-\mathbf{h}_{P}$.
Then we have
\begin{align}\label{widetilded1}
\big\|\mathbf{v}\!\cdot\!\nabla \mathbf{f}\big\|_{ C^{\alpha,\frac{\alpha}{2}}(Q_{T_*})}
\leq& C\big[\|(\mathbf{v}\!-\!\mathbf{u}_{0})
\!\cdot\!\nabla\widetilde{\mathbf{f}}\|_{C^{\alpha,\frac{\alpha}{2}}(Q_{T_*})}
+\!\|\mathbf{u}_{0}\!\cdot\!\nabla \widetilde{\mathbf{f}}\|_{C^{\alpha,\frac{\alpha}{2}}(Q_{T_*})}
\nonumber\\
&+\!\|(\mathbf{v}\!-\!\mathbf{u}_{0})\!\cdot\!\nabla \mathbf{h}_{P}\|_{C^{\alpha,\frac{\alpha}{2}}(Q_{T_*})}
\!+\!\|\mathbf{u}_{0}\!\cdot\!\nabla \mathbf{h}_{P}\|_{ C^{\alpha,\frac{\alpha}{2}}(Q_{T_*})}\big].
\end{align}
It follows from \eqref{estimatehp} that
\begin{align*}
\|\mathbf{u}_{0}\!\cdot\!\nabla \mathbf{h}_{P}\|_{\! C^{\alpha,\frac{\alpha}{2}}(Q_{T_*})}\leq C_{0}.
\end{align*}
Since $(\mathbf{v}-\mathbf{u}_{0},\widetilde{\mathbf{f}})=(0, 0)$ in $\Omega\times\{0\}$,
we have 
\begin{align*}
\sum_{k=0}^2\|\nabla^{k}(\mathbf{v}-\mathbf{u}_{0})\|_{C^{0}(Q_{T_*})}\leq KT_{*}, 
\end{align*}
and
\begin{align*}
\sum_{k=0}^2\|\nabla^{k} \widetilde{\mathbf{f}}\|_{C^{0}(Q_{T_*})}
=\|\nabla^{k} (\mathbf{f}-\mathbf{h}_{P})\|_{C^{0}(Q_{T_*})}\leq KT_{*}.
\end{align*}
Employing the interpolation inequalities, we have that for any $0<\delta<1$,
\begin{align*}
\|\mathbf{v}-\mathbf{u}_{0}\|_{C^{\alpha,\frac{\alpha}{2}}(Q_{T_*})}
\leq& C\big(\frac{1}{\delta} \|\mathbf{v}-\mathbf{u}_{0}\|_{C^{0}(Q_{T_*})}
+\delta \|\mathbf{v}-\mathbf{u}_{0}\|_{C^{2+\alpha,1+\frac{\alpha}{2}}(Q_{T_*})}\big)
\nonumber\\
\leq & C\big(\delta+\frac{T_{*}}{\delta}\big)K,\nonumber\\
\big\| \widetilde{\mathbf{f}}\big\|_{C^{\alpha,\frac{\alpha}{2}}(Q_{T_*})}
\leq& C\big( \frac{1}{\delta} \| \widetilde{\mathbf{f}}\|_{C^{0}(Q_{T_*})}
+\delta \|\widetilde{\mathbf{f}}\|_{{C^{2+\alpha,1+\frac{\alpha}{2}}(Q_{T_*})}}\big)\\
\leq & C\big(\delta+\frac{T_{*}}{\delta}\big)K,
\end{align*}
and
\begin{align*}
\big\|\nabla \widetilde{\mathbf{f}}\big\|_{C^{\alpha,\frac{\alpha}{2}}(Q_{T_*})}
\leq &\big( \frac{1}{\delta} \|\nabla \widetilde{\mathbf{f}}\|_{C^{0}(Q_{T_*})}
+\delta \|\widetilde{\mathbf{f}}\|_{{C^{2+\alpha,1+\frac{\alpha}{2}}(Q_{T_*})}}\big)\\
\leq & C\big (\delta+\frac{T_{*}}{\delta}\big)K.
\end{align*}
Putting these estimates together,
we obtain that
\begin{align*}
\big\|(\mathbf{v}-\mathbf{u}_{0})\cdot\nabla \widetilde{\mathbf{f}}\big\|_{C^{\alpha,\frac{\alpha}{2}}(Q_{T_*})}
&\leq  C\big(
\|\mathbf{v}-\mathbf{u}_{0}\|_{C^{0}(Q_{T_*})}
\|\nabla \widetilde{\mathbf{f}}\|_{C^{\alpha,\frac{\alpha}{2}}(Q_{T_*})}\\
&\ + \|\mathbf{v}-\mathbf{u}_{0}\|_{C^{\alpha,\frac{\alpha}{2}}(Q_{T_*})}
\|\nabla \widetilde{\mathbf{f}}\|_{C^{0}(Q_{T_*})}\big)
\nonumber\\
&\leq  CT_{*}\big(\delta+\frac{T_{*}}{\delta}\big)K^{2}.
\end{align*}
Similarly, we have that
\begin{align*}
&\big\|\mathbf{u}_{0}\cdot\nabla\widetilde{\mathbf{f}}\big\|_{C^{\alpha,\frac{\alpha}{2}}(Q_{T_*})}
+\|(\mathbf{v}-\mathbf{u}_{0})\cdot
\nabla\mathbf{h}_{P}\|_{C^{\alpha,\frac{\alpha}{2}}(Q_{T_*})}
\nonumber\\
&\leq 
C_{0}(\|\mathbf{v}-\mathbf{u}_{0}\|_{C^{\alpha,\frac{\alpha}{2}}(Q_{T_*})}
+\|\nabla \widetilde{\mathbf{f}}\|_{C^{\alpha,\frac{\alpha}{2}}(Q_{T_*})}\\
&\ \ +\|\mathbf{v}-\mathbf{u}_{0}\|_{C^{0}(Q_{T_*})}
+\|\nabla \widetilde{\mathbf{f}}\|_{C^{0}(Q_{T_*})})
\nonumber\\
&\leq C_{0}\big(T_{*}+\delta+\frac{T_{*}}{\delta}\big)K.
\end{align*}
Putting all these estimates together, we obtain that
\begin{align*}
&\big\|\mathbf{v}\cdot\nabla\mathbf{f}\big\|_{C^{\alpha,\frac{\alpha}{2}}(Q_{T_*})}\\
&\leq C_{0}\big(T_{*}+\delta+\frac{T_{*}}{\delta}\big)K+CT_{*}\big(\delta+\frac{T_{*}}{\delta}\big)K^{2}
+C_{0}\\
&\leq  \frac{\sqrt{K}}{4}, 
\end{align*}
provided  $K=64 C_{0}^{2}$, $\delta\leq \frac{1}{16(C_{0}+CK)\sqrt{K}}$,
and $T_{*}=\min \{\frac{1}{2}, \delta^{2}\}$.
\medskip

We can estimate the second term in the right-hand side of \eqref{widetilded} by
\begin{align} \label{widetilded2}
&\||\nabla \mathbf{f}|^{2}\mathbf{f}\|_{C^{\alpha,\frac{\alpha}{2}}(Q_{T_*})}\nonumber\\
&\leq \||\nabla \widetilde{\mathbf{f}}|^{2}\widetilde{\mathbf{f}}
\|_{C^{\alpha,\frac{\alpha}{2}}(Q_{T_*})}
+2\||\nabla \widetilde{\mathbf{f}}||\nabla \mathbf{h}_{P}|
|\widetilde{\mathbf{f}}|\|_{C^{\alpha,\frac{\alpha}{2}}(Q_{T_*})}\nonumber\\
&+\||\nabla \widetilde{\mathbf{f}}|^{2}
|\mathbf{h}_{P}|\|_{C^{\alpha,\frac{\alpha}{2}}(Q_{T_*})}
+\||\nabla \mathbf{h}_{P}|^{2}
|\widetilde{\mathbf{f}}|\|_{C^{\alpha,\frac{\alpha}{2}}(Q_{T_*})}\nonumber\\
&+2\||\nabla \widetilde{\mathbf{f}}||\nabla \mathbf{h}_{P}|
|\mathbf{h}_{P}|\|_{C^{\alpha,\frac{\alpha}{2}}(Q_{T_*})}
+\||\nabla \mathbf{h}_{P}|^{2}
|\mathbf{h}_{P}|\|_{C^{\alpha,\frac{\alpha}{2}}(Q_{T_*})}.
\end{align}
It is easy to see that
\begin{align*}
&\||\nabla \mathbf{h}_{P}|^{2}
|\mathbf{h}_{P}|\|_{C^{\alpha,\frac{\alpha}{2}}(Q_{T_*})}\leq C_{0},\\
&\||\nabla \widetilde{\mathbf{f}}|^{2}\widetilde{\mathbf{f}}
\|_{C^{\alpha,\frac{\alpha}{2}}(Q_{T_*})}\\
&\le C\big( \|\nabla \widetilde{\mathbf{f}}\|_{C^{0}(Q_{T_*})}^{2}
\|\widetilde{\mathbf{f}}\|_{C^{\alpha,\frac{\alpha}{2}}(Q_{T_*})}
+
\|\nabla \widetilde{\mathbf{f}}\|_{C^{0}(Q_{T_*})}
\|\nabla\widetilde{\mathbf{f}}\|_{C^{\alpha,\frac{\alpha}{2}}(Q_{T_*})}
\|\widetilde{\mathbf{f}}\|_{C^{0}(Q_{T_*})}\big)
\nonumber\\
& \le C T_{*}^{2}\big(\delta+\frac{T_{*}}{\delta}\big) K^{3},
\end{align*}
and
\begin{align*}
&\||\nabla \widetilde{\mathbf{f}}||\nabla \mathbf{h}_{P}|
|\widetilde{\mathbf{f}}|\|_{C^{\alpha,\frac{\alpha}{2}}(Q_{T_*})}\\
&\leq C\big(
\|\nabla \widetilde{\mathbf{f}}\|_{C^{\alpha,\frac{\alpha}{2}}(Q_{T_*})}
\|\nabla \mathbf{h}_{P}\|_{C^{0}(Q_{T_*})}
\|\widetilde{\mathbf{f}}\|_{C^{0}(Q_{T_*})}
\nonumber\\
&+\|\nabla \widetilde{\mathbf{f}}\|_{C^{0}(Q_{T_*})}(
\|\nabla \mathbf{h}_{P}\|_{C^{\alpha,\frac{\alpha}{2}}(Q_{T_*})}
\|\widetilde{\mathbf{f}}\|_{C^{0}(Q_{T_*})}
+
\|\nabla \mathbf{h}_{P}\|_{C^{0}(Q_{T_*})}
\|\widetilde{\mathbf{f}}\|_{C^{\alpha,\frac{\alpha}{2}}(Q_{T_*})}
)\big)\nonumber\\
&\leq
C_{0} T_{*}\big(\delta+\frac{T_{*}}{\delta}\big) K^{3}.
\end{align*}
Similarly, we have that
\begin{align*}
&\||\nabla \widetilde{\mathbf{f}}|^{2}
|\mathbf{h}_{P}|\|_{C^{\alpha,\frac{\alpha}{2}}(Q_{T_*})}
+\||\nabla \mathbf{h}_{P}|^{2}
|\widetilde{\mathbf{f}}|\|_{C^{\alpha,\frac{\alpha}{2}}(Q_{T_*})}\\
&+2\||\nabla \widetilde{\mathbf{f}}||\nabla \mathbf{h}_{P}|
|\mathbf{h}_{P}|\|_{C^{\alpha,\frac{\alpha}{2}}(Q_{T_*})}\\
&\leq
C_{0} T_{*}(T_{*}+\delta+\frac{T_{*}}{\delta}) K^{2}.
\end{align*}
Substituting these estimates  into \eqref{widetilded2}, we get that
\begin{align*}
\||\nabla\mathbf{f}|^{2}\mathbf{f}\|_{C^{\alpha,\frac{\alpha}{2}}(Q_{T_*})}
\leq& C_{0} T_{*}(T_{*}+\delta+\frac{T_{*}}{\delta}) K^{2}
+CT_{*}^{2} (\delta+\frac{T_{*}}{\delta}) K^{3}+C_{0}
\nonumber\\
\leq& \frac{\sqrt{K}}{4},
\end{align*}
provided  $K=64C_{0}^{2}$, $\delta\leq \frac{1}{32(C_{0}+C)K^{\frac{5}{2}}}$ and
$T_{*}=\min\{\frac{1}{2}, \delta \}$. Hence
\begin{align}\label{widetilded22}
\|\widetilde{\mathbf{d}}\|_{C^{2+\alpha,1+\frac{\alpha}{2}}(Q_{T_*})}
\leq \frac{\sqrt{K}}{2}.
\end{align}
By the Schauder theory for non-homogeneous, non-stationary Stokes equations
$\eqref{NLC3}_{1}$, we have that
\begin{align}\label{strongu}
\|\mathbf{u}-\mathbf{u}_{0}\|_{C^{2+\alpha,1+\frac{\alpha}{2}}(Q_{T_*})}
&\leq C\big[
\|\mathbf{v}\cdot\nabla \mathbf{v}\|_{C^{\alpha,\frac{\alpha}{2}}(Q_{T_*})}\nonumber\\
&\qquad+\|\nabla\cdot(\nabla \mathbf{f}\odot\nabla \mathbf{f})\|_{C^{\alpha,\frac{\alpha}{2}}(Q_{T_*})}\big].
\end{align}
For the first term of the right-hand side of \eqref{strongu}, we have
\begin{align*}
&\|\mathbf{v}\cdot\nabla \mathbf{v}\|_{C^{\alpha,\frac{\alpha}{2}}(Q_{T_*})}\\
&\leq C \|(\mathbf{v}-\mathbf{u}_{0})\cdot\nabla (\mathbf{v}-\mathbf{u}_{0})\|_{C^{\alpha,\frac{\alpha}{2}}(Q_{T_*})}
+
\|(\mathbf{v}-\mathbf{u}_{0})\cdot
\nabla \mathbf{u}_{0}\|_{C^{\alpha,\frac{\alpha}{2}}(Q_{T_*})}\nonumber\\
&
\ \ \ +
\|\mathbf{u}_{0}\cdot
\nabla (\mathbf{v}-\mathbf{u}_{0})\|_{C^{\alpha,\frac{\alpha}{2}}(Q_{T_*})}
+
\|\mathbf{u}_{0}\cdot\nabla \mathbf{u}_{0}\|_{C^{\alpha,\frac{\alpha}{2}}(Q_{T_*})}
\nonumber\\
&\leq C T_{*}(\delta+\frac{T_{*}}{\delta}) K^{2}+C_{0}(T_{*}+\delta+\frac{T_{*}}{\delta})K+C_{0}
\nonumber\\
&\leq  \frac{\sqrt{K}}{4},
\end{align*}
provided $K=64 C_{0}^{2}$, $\delta\leq \frac{1}{16(C_{0}+CK)\sqrt{K}}$ and $T_{*}=\min \{\frac{1}{2}, \delta^{2}\}$.

\noindent For the second term in the right hand side of \eqref{strongu}, we have 
\begin{align*}
&\|\nabla\cdot(\nabla\mathbf{f}\odot\nabla\mathbf{f})
\|_{C^{\alpha,\frac{\alpha}{2}}(Q_{T_*})}\\
&\!\leq C \big[\|\nabla\cdot(\nabla\widetilde{\mathbf{f}}\odot\nabla\widetilde{\mathbf{f}})
\|_{C^{\alpha,\frac{\alpha}{2}}(Q_{T_*})}
+
\|\nabla\cdot(\nabla \mathbf{h}_{P}\odot\nabla\widetilde{\mathbf{f}})
\|_{C^{\alpha,\frac{\alpha}{2}}(Q_{T_*})}
\nonumber\\
&+
\|\nabla\cdot(\nabla\widetilde{\mathbf{f}}\odot\nabla \mathbf{h}_{P})
\|_{C^{\alpha,\frac{\alpha}{2}}(Q_{T_*})}
+
\|\nabla\cdot(\nabla \mathbf{h}_{P}\odot\nabla \mathbf{h}_{P})
\|_{C^{\alpha,\frac{\alpha}{2}}(Q_{T_*})}\big]
\nonumber\\
&\leq C_{0}
\|\widetilde{\mathbf{f}}\|_{C^{2+\alpha,1+\frac{\alpha}{2}}(Q_{T_*})}
+T_{*}K \|\widetilde{\mathbf{f}}\|_{C^{2+\alpha,1+\frac{\alpha}{2}}(Q_{T_*})}
+C_{0}\nonumber\\
&\leq  C_{0} K+C T_{*}K^{2}+C_{0}\leq \frac{\sqrt{K}}{4},
\end{align*}
provided $K=64C_{0}^{2}$ and  $T_{*}\leq \min\{1,T,\frac{1}{8(1+CK^{\frac{1}{2}})K}\}$. 

These two inequalities, together with \eqref{widetilded22} and \eqref{strongu}, imply that
\begin{align*}
\|\mathbf{u}-\mathbf{u}_{0}\|_{C^{2+\alpha,1+\frac{\alpha}{2}}(Q_{T_*})}
+
\|\mathbf{d}-\mathbf{h}_{P}
\|_{C^{2+\alpha,1+\frac{\alpha}{2}}(Q_{T_*})}
\leq K,
\end{align*}
which implies that $L:\mathfrak{X}(T_*, K)\mapsto\mathfrak{X}(T_*, K)$
is bounded. The proof of Lemma \ref{lem7} is completed.
\end{proof}

\begin{lemma} \label{lem8}
There exist a sufficiently large $K>0$ and a sufficiently small $T_{*}>0$ 
such that $L:\mathfrak{X}(T_*,K)\mapsto \mathfrak{X}(T_*, K)$ is a contractive map.
\end{lemma}

\begin{proof}
For $i=1,2$, and any given $(\mathbf{v}_{i},\mathbf{f}_{i})\in\mathfrak{X}(T_*, K)$,  
let $(\mathbf{u}_{i},\mathbf{d}_{i})\in\mathfrak{X}(T_*,K)$ be defined by
\begin{align*}
(\mathbf{u}_{i},\mathbf{d}_{i})=L(\mathbf{v}_{i},\mathbf{f}_{i}).
\end{align*}
Set
\begin{align*}
(\mathbf{u}, {\bf d}, P, {\bf v}, {\bf f})=(\mathbf{u}_{1}-\mathbf{u}_{2},
\mathbf{d}_{1}-\mathbf{d}_{2}, P_{1}-P_{2}, \mathbf{v}_{1}-\mathbf{v}_{2},
\mathbf{f}_{1}-\mathbf{f}_{2}).
\end{align*}
Then $(\mathbf{u},\mathbf{d})$ solves
\begin{align*}
\begin{cases}
\partial_{t}\mathbf{u}-\Delta\mathbf{u}+\nabla P=\mathbf{G},  &\text{in }Q_{T_{*}},\\
\nabla \cdot\mathbf{u}=0,  &\text{in }Q_{T_{*}},\\
\partial_{t}\mathbf{d}-\Delta\mathbf{d}=\mathbf{H},  &\text{in }Q_{T_{*}},\\
(\mathbf{u}, {\bf d})=(0,0), &\text{on }\Gamma_{T_{*}},\\
(\mathbf{u}, {\bf d})=(0, 0), &\text{on }\Omega\times\{0\},
\end{cases}
\end{align*}
where
\begin{align*}
\mathbf{G}=-\mathbf{v}\cdot\nabla \mathbf{v}_{1}-\mathbf{v}_{2}\cdot\nabla \mathbf{v}-\nabla\cdot(\nabla \mathbf{f}\odot\nabla \mathbf{f}_{1}+
\nabla \mathbf{f}_{2}\odot\nabla \mathbf{f}),
\end{align*}
and
\begin{align*}
\mathbf{H}=|\nabla \mathbf{f}_{1}|^{2}\mathbf{f}
+\langle\nabla (\mathbf{f}_{1}+\mathbf{f}_{2}), \nabla\mathbf{f} \rangle\mathbf{f}_{2}
-\mathbf{v}\cdot\nabla\mathbf{f}_{1}+\mathbf{v}_{2}\cdot\nabla\mathbf{f}.
\end{align*}
From Lemma \ref{lem7}, we have that 
\begin{align*}
\sum_{i=1}^2\big(\|\mathbf{u}_{i}-\mathbf{u}_{0}\|_{C^{2+\alpha,1+\frac{\alpha}{2}}(Q_{T_*})}
+\|\mathbf{d}_{i}-\mathbf{h}_{P}
\|_{C^{2+\alpha,1+\frac{\alpha}{2}}(Q_{T_*})}\big)\leq K.
\end{align*}

As in Lemma \ref{lem7},
we can apply the Schauder theory  to get
\begin{align}\label{dW32}
&\|\mathbf{d}\|_{C^{2+\alpha,1+\frac{\alpha}{2}}(Q_{T_*})}\nonumber\\
&\leq C
\|\mathbf{H}\|_{C^{\alpha,\frac{\alpha}{2}}(Q_{T_*})}\nonumber\\
&\leq C\!\!
\|
|\nabla \mathbf{f}_{1}|^{2}\mathbf{f}+\langle\nabla (\mathbf{f}_{1}+\mathbf{f}_{2}), \nabla\mathbf{f}\rangle\mathbf{f}_{2}
-\mathbf{v}\cdot\nabla\mathbf{f}_{1}+\mathbf{v}_{2}\cdot\nabla\mathbf{f}
\|_{C^{\alpha,\frac{\alpha}{2}}(Q_{T_*})}
\nonumber\\
&\leq C
(K^{2}+K) (\|\mathbf{v}\|_{C^{\alpha,\frac{\alpha}{2}}(Q_{T_*})}
+\|\mathbf{f}\|_{C^{\alpha,\frac{\alpha}{2}}(Q_{T_*})}
+\|\nabla\mathbf{f}\|_{C^{\alpha,\frac{\alpha}{2}}(Q_{T_*})}),
\end{align}
and
\begin{align}\label{uW21}
&\|\mathbf{u}\|_{C^{2+\alpha,1+\frac{\alpha}{2}}(Q_{T_*})}\\
&\leq C
 \|\mathbf{G}\|_{C^{\alpha,\frac{\alpha}{2}}(Q_{T_*})}\nonumber\\
&\leq C
\|\mathbf{v}\cdot\nabla \mathbf{v}_{1}+\mathbf{v}_{2}\cdot\nabla \mathbf{v}+\nabla\cdot(\nabla \mathbf{f}\odot\nabla \mathbf{f}_{1}+
\nabla \mathbf{f}_{2}\odot\nabla \mathbf{f})
\|_{C^{\alpha,\frac{\alpha}{2}}(Q_{T_*})}
\nonumber\\
&\leq C
K(\|\mathbf{v}\|_{C^{\alpha,\frac{\alpha}{2}}(Q_{T_*})}
+\|\nabla\mathbf{v}\|_{C^{\alpha,\frac{\alpha}{2}}(Q_{T_*})}\nonumber\\
&\ \ \ +\|\mathbf{f}\|_{C^{2+\alpha,1+\frac{\alpha}{2}}(Q_{T_*})}
+\|\nabla\mathbf{f}\|_{C^{\alpha,\frac{\alpha}{2}}(Q_{T_*})}).\nonumber
\end{align}
Hence, it follows from \eqref{dW32} and \eqref{uW21} that
\begin{align*}
&\|L(\mathbf{v}_{1},\mathbf{f}_{1})-L(\mathbf{v}_{2},\mathbf{f}_{2})\|_{\mathfrak{X}(T_*,K)}\\
&=
\|\mathbf{u}\|_{C^{2+\alpha,1+\frac{\alpha}{2}}(Q_{T_*})}
+
\|\mathbf{d}\|_{C^{2+\alpha,1+\frac{\alpha}{2}}(Q_{T_*})}
\nonumber\\
&\leq C
(K^{2}+K) (\delta(\|\mathbf{v}\|_{C^{2+\alpha,1+\frac{\alpha}{2}}(Q_{T_*})}
+\|\mathbf{f}\|_{C^{2+\alpha,1+\frac{\alpha}{2}}(Q_{T_*})})\\
&\ \ \ +\frac{1}{\delta}
(\|\mathbf{v}\|_{C^{0}(Q_{T_*})}
+\|\mathbf{f}\|_{C^{0}(Q_{T_*})}))
\nonumber\\
&\leq (K^{2}+K) (\delta+\frac{T_{*}}{\delta})
(\|\mathbf{v}\|_{C^{2+\alpha,1+\frac{\alpha}{2}}(Q_{T_*})}+\|\mathbf{f}\|_{C^{2+\alpha,1+\frac{\alpha}{2}}(Q_{T_*})})
\nonumber\\
&\leq \frac{1}{2} (\|\mathbf{v}\|_{C^{2+\alpha,1+\frac{\alpha}{2}}(Q_{T_*})}
+\|\mathbf{f}\|_{C^{2+\alpha,1+\frac{\alpha}{2}}(Q_{T_*})})\nonumber\\
&\le \frac{1}{2} \|(\mathbf{v}_{1},\mathbf{f}_{1})-(\mathbf{v}_{2},\mathbf{f}_{2})\|_{\mathfrak{X}(T_*,K)},
\end{align*}
provided $\delta$ and $T_{*}$ are sufficiently small. Thus we obtain that 
$L:\mathfrak{X}(T_*,K)\mapsto \mathfrak{X}(T_*,K)$ is a contractive map.
This completes the proof of Lemma \ref{lem8}.
\end{proof}

\begin{lemma}\label{lem9}
For a bounded smooth domain $\Omega\subset\mathbb{R}^{2}$ and $0<T<\infty$, 
let $\mathbf{u}\in W^{2,1}_{2}(Q_T,\mathbb{R}^{2})$ 
with $\nabla\cdot\mathbf{u}=0$, $\mathbf{h}\in C^{\alpha,\frac{\alpha}{2}}(\Gamma_{T}, \mathbb{S}^{2}))$,
and  $\mathbf{d}_{0}\in C^{2+\alpha}(\Omega;\mathbb{S}^{2})$. 
If $\mathbf{d}\in C^{2+\alpha,1+\frac{\alpha}{2}}(Q_T, \mathbb{R}^{3})$
is a solution of $\eqref{eq1.1}_{3}$-\eqref{eq1.2}-\eqref{eq1.3},
then  it holds that
 \begin{align*}
 |\mathbf{d}|=1 \quad\text{ in } Q_T.
 \end{align*}
\end{lemma}

\begin{proof}
Multiplying $\eqref{eq1.1}_{3}$ by ${\bf d}$, we get that
\begin{align*}
\partial_{t} (|\mathbf{d}|^{2}-1)+\mathbf{u}\cdot\nabla(|\mathbf{d}|^{2}-1)
=\Delta(|\mathbf{d}|^2-1)+2|\nabla \mathbf{d}|^{2}(|\mathbf{d}|^{2}-1).
\end{align*}
Set $g=|\mathbf{d}|^{2}-1$ and $g^{+}=\max\{g,0\}$.  We have that
\begin{align}\label{g+}
\begin{cases}
\partial_{t}g^{+}-\Delta g^{+} =-\mathbf{u}\cdot\nabla g^{+}+2|\nabla \mathbf{d}|^{2}g^{+},
\quad &\text{ in } Q_T,\\
g^{+}=\max\{|\mathbf{h}|^{2}-1,0\}=0\quad &\text{ on } \Gamma_T,\\
g^{+}=\max\{|\mathbf{d}_{0}|^{2}-1,0\}= 0\quad &\text{ on }\Omega\times\{0\}.
\end{cases}
\end{align}
Multiplying \eqref{g+} with $g^{+}$, integrating the resulting equation over $\Omega$, and using the fact that 
$\nabla \cdot\mathbf{u}=0$, we obtain that
\begin{align*}
\frac{1}{2}\frac{d}{dt}\int_{\Omega}(g^{+})^{2}+\int_{\Omega}|\nabla g^{+}|^{2}
&=-\frac{1}{2}\int_{\Omega} \mathbf{u}\cdot\nabla (g^{+})^{2}+2\int_{\Omega}|\nabla \mathbf{d}|^{2} (g^{+})^{2}
\nonumber\\
&= 2\int_{\Omega}|\nabla \mathbf{d}|^{2}(g^{+})^{2}.
\end{align*}
Integrating over $[0,t]$ for $0<t\leq T$,  and employing the fact
$\|\nabla \mathbf{d}\|_{L^{\infty}(Q_T)}<\infty$ and the Gronwall inequality, we obtain that
\begin{align*}
g^{+}=0 \quad \text{ in } Q_T.
\end{align*}
This implies $|{\bf d}|\le 1$ in $Q_T$. Similarly, we can show that
$|{\bf d}|\ge 1$ in $Q_T$. Hence $|{\bf d}=1$ in $Q_T$.
This completes the proof.
\end{proof}

In order to construct the existence of global strong solutions to
\eqref{eq1.1}--\eqref{eq1.3}, we also need the following Lemma.

\begin{lemma} \label{rem-S+} For $T>0$ and  a bounded smooth domain $\Omega \subset\mathbb{R}^{2}$, 
for a given $\mathbf{u}\in W^{2,1}_{2}(Q_T, \mathbb{R}^{2})$ with $\nabla\cdot\mathbf{u}=0$, 
$\mathbf{h}\in C^{\alpha,\frac{\alpha}{2}}(\Gamma_{T}, \mathbb{S}^{2}))$ and  
$\mathbf{d}_{0}\in C^{2+\alpha}(\Omega, \mathbb{S}^{2})$, 
let $\mathbf{d}\in C^{2+\alpha,1+\frac{\alpha}{2}}(\overline{\Omega}\times [0,T], \mathbb{S}^{2})$ solve $\eqref{eq1.1}_{3}$-\eqref{eq1.2}-\eqref{eq1.3}. 
If ${\bf d}_{0}^3\geq 0$ in $\Omega$ and ${\bf h}_{3}\geq 0$ on $\Gamma_T$, then  
 \begin{align*}
 {\bf d}^{3}({x},t)\geq 0 \quad\text{ in } Q_T.
 \end{align*}
 Here ${\bf d}^3$ denotes the third component of ${\bf d}$.
\end{lemma}
\begin{proof} The proof is similar to that of Lemma \ref{lem9}. For the convenience of readers, we sketch it here.
Set ${\bf d}^{3}_{-}=\min\{{\bf d}^3, 0\}$. Then 
\begin{align*}
\begin{cases}
\partial_t {\bf d}^{3}_{-}-\Delta {\bf d}^{3}_{-}=-{\bf u}\cdot\nabla{\bf d}^3_{-}+|\nabla {\bf d}|^2 {\bf d}^3_{-},
\ & {\rm{in}}\ Q_T,\\ 
{\bf d}^{3}_{-}=0,  \ & {\rm{on}}\ \partial_p Q_T.
\end{cases}
\end{align*}
Multiplying this equation by ${\bf d}^{3}_{-}$, integrating over $\Omega$, and applying $\nabla\cdot{\bf u}=0$,
we obtain
\begin{align*}
\frac{d}{dt}\int_\Omega |{\bf d}^{3}_{-}|^2+\int_\Omega |\nabla {\bf d}^{3}_{-}|^2
=2\int_\Omega |\nabla {\bf d}|^2|{\bf d}^{3}_{-}|^2\le C\int_\Omega |{\bf d}^{3}_{-}|^2.
\end{align*}
Hence by the Gronwall inequality we have
$$\int_\Omega |{\bf d}^{3}_{-}(t)|^2\le e^{Ct}\int_\Omega |({\bf d}_0^{3})_{-}|^2=0, \ \forall t\in [0,T].$$
This implies that ${\bf d}^3\ge 0$ in $Q_T$. 
\end{proof}

In the process to obtain global strong solutions, we also need the following elementary Lemma.

\begin{lemma}\label{vanish}
If $\omega\in C^\infty(\mathbb{S}^2, \mathbb{S}^2_+)$ is a harmonic map, then
$\omega$ must be a constant map.
\end{lemma}

\begin{proof}
Recall that $\omega$ solves the harmonic map equation:
\begin{equation}\label{HM}
\Delta_{\mathbb S^2}\omega +|\nabla_{\mathbb S^2} \omega|^2\omega=0\ {\rm{on}}\ \mathbb S^2.
\end{equation}
Here $\nabla_{\mathbb S^2}$ and $\Delta_{\mathbb S^2}$ denote the gradient and Laplace operator
on $\mathbb S^2$ respectively.
Integrating the equation over $\mathbb S^2$ yields
$$0=\int_{\mathbb S^2}(\Delta_{\mathbb S^2}\omega^3 +|\nabla_{\mathbb S^2} \omega|^2\omega^3)\,d\sigma
=\int_{\mathbb S^2}|\nabla_{\mathbb S^2} \omega|^2\omega^3\,d\sigma.$$
Since $\omega^3\ge 0$ on $\mathbb S^2$, this implies that
$$|\nabla_{\mathbb S^2} \omega|^2\omega^3\equiv 0 \ {\rm{on}}\ \mathbb S^2.$$ 
There are two cases that we need to consider:\\

\noindent(a) If there exists $p_0\in \mathbb S^2$ such that $\omega^3(p_0)>0$, then there exists $\delta_0>0$ such that
$$\nabla_{\mathbb S^2}\omega=0 \ {\rm{in}}\ B_{\delta_0}(p_0)\cap \mathbb S^2,$$
and hence $\omega\equiv p_1\in\mathbb S^2_+$ in $B_{\delta_0}(p_0)\cap\mathbb S^2$. This, combined with the unique continuation property,
yields that $\omega\equiv p_1$ on $\mathbb S^2$.\\

\noindent(b) If $\omega^3\equiv 0$ on $\mathbb S^2$, then we must have $\omega(\mathbb S^2)\subset \partial\mathbb S^2_+\equiv \mathbb S^1\subset\mathbb R^2$. 
In this case, we can write $\omega=e^{i\phi}$ for a smooth function $\phi\in C^\infty(\mathbb S^2)$. Direct calculations imply that $\omega$ is
a harmonic map to $\mathbb S^1$ if and only if $\phi$ is a harmonic function on $\mathbb S^2$. Hence by the maximum principle we conclude that
$\phi$ is a constant. Hence $\omega=e^{i\phi}$ is also a constant on $\mathbb S^2$.
\end{proof}

\subsection{A priori estimates on energy and pressure}
In this subsection, we will provide some basic estimates on both the energy 
and the pressure. First, we have the following generalized global energy inequality.

\begin{lemma}\label{lem10}
For $T>0$, $\mathbf{h}\in L^{2}_tH^{\frac{3}{2}}_x(\Gamma_T,\mathbb{S}^{2}))$, 
$\partial_{t}\mathbf{h}\in L^{2}_tH^{\frac{3}{2}}_x(\Gamma_T)$, and
$(\mathbf{u}_{0},\mathbf{d}_{0})\in \mathbf{H}\times H^{1}(\Omega, \mathbb{S}^{2})$,
suppose $\mathbf{u}\in L^{\infty}([0,T], \mathbf{H})\cap L^{2}([0,T], \mathbf{V})$,
$\mathbf{d}\in L^{\infty}_tH^{1}_x(Q_T,\mathbb{S}^{2})$ and 
$P\in L^{\frac{4}{3}}_tW^{1,\frac43}_x(Q_T)$ is a weak solution to the 
system \eqref{eq1.1}--\eqref{eq1.3}. Then there exists $C>0$ depending
only on $\Omega$ such that for any $t\in (0,T]$, 
\begin{align}\label{energy1}
&\int_{\Omega} (|\mathbf{u}|^{2}+|\nabla \mathbf{d}|^{2})(\cdot,t)
+\int_{Q_t} (|\nabla \mathbf{u}|^{2}+|\Delta\mathbf{d}+|\nabla{\bf d}|^2{\bf d}|^{2})\nonumber\\
&\leq \psi(t)\big[\int_{\Omega} (|\mathbf{u}_{0}|^{2}+|\nabla \mathbf{d}_{0}|^{2})
+C\int_{0}^{t}
\|(\mathbf{h}, \partial_t {\bf h})(\tau)\|_{H^{\frac{3}{2}}(\Gamma)}^{2}\,d\tau\big],
\end{align}
where 
$$\psi(t)=\exp\big(C\int_0^t \|\partial_t {\bf h}(\tau)\|_{H^\frac32(\Gamma)}\,d\tau\big).$$
\end{lemma}

\begin{proof} Let ${\bf h}_E\in L^2_tH^2_x(Q_T,\mathbb R^3)$ 
be the harmonic extension of ${\bf h}$,
i.e., for all $t\in (0,T]$, 
\begin{align*}
\begin{cases}
\Delta{\bf h}_E(\cdot, t)=0 \ {\rm{in}}\ \Omega,\\
 {\bf h}_E(\cdot, t)={\bf h}(\cdot, t)
\ {\rm{on}}\ \Gamma.
\end{cases}
\end{align*}
Then we have that ${\bf h}_E, \partial_t{\bf h}_E\in L^2_tH^2_x(Q_T)$, and
\begin{align}\label{he-est0}
\begin{cases}\displaystyle
\big\|{\bf h}_E\big\|_{L^2_tH^2_x(Q_T)}\le C\big\| {\bf h}_E\big\|_{L^2_tH^\frac32_x(\Gamma_T)},\\
\big\|\partial_t{\bf h}_E\big\|_{L^2_t H^2_x(Q_T)}
\le C\big\|\partial_t{\bf h}_E\big\|_{L^2_tH^\frac32_x(\Gamma_T)}.
\end{cases}
\end{align}
Multiplying $\eqref{eq1.1}_{1}$ by $\mathbf{u}$, integrating over $\Omega$, and using 
\eqref{eq1.2}, we obtain that
\begin{align}\label{energyforu}
\frac{1}{2}\frac{d}{dt}\int_{\Omega} |\mathbf{u}|^{2}+\int_{\Omega} |\nabla \mathbf{u}|^{2}
=-\int_{\Omega} \langle\mathbf{u}\cdot\nabla \mathbf{d}, \Delta\mathbf{d}\rangle,
\end{align}
Multiplying $\eqref{eq1.1}_{3}$ by $\Delta \mathbf{d}+|\nabla{\bf d}|^2{\bf d}$ 
and integrating over $\Omega$ yields that
\begin{align}\label{energyford1}
&\frac{1}{2}\frac{d}{dt}\int_{\Omega} |\nabla \mathbf{d}|^{2}+\int_{\Omega} |\Delta \mathbf{d}+|\nabla{\bf d}|^2{\bf d}|^{2}\nonumber\\
&=\int_{\Omega} \langle\mathbf{u}\cdot\nabla \mathbf{d}, \Delta\mathbf{d}\rangle
+ \int_{\Gamma} \langle\frac{\partial{\bf d}}{\partial {\nu}}, \partial_{t}\mathbf{h}\rangle,
\end{align}
where ${\nu}$ is the outward unit normal vector of $\Gamma$. 

Adding \eqref{energyforu} with \eqref{energyford1}, we have that
\begin{align*}
&\frac{d}{dt}\int_{\Omega} \frac12(|\mathbf{u}|^{2}+|\nabla \mathbf{d}|^{2})
+\int_{\Omega} (|\nabla \mathbf{u}|^{2}+|\Delta\mathbf{d}+|\nabla{\bf d}|^2{\bf d}|^{2})\\
&=\int_{\Gamma} \langle\frac{\partial{\bf d}}{\partial {\nu}}, \partial_{t}\mathbf{h}\rangle.
\end{align*}
Now we estimate the right hand side as follows
\begin{align}\label{bdry-contr}
\int_{\Gamma} \langle\frac{\partial{\bf d}}{\partial {\nu}}, \partial_{t}\mathbf{h}\rangle
&=\int_{\Gamma} \langle\frac{\partial({\bf d}-{\bf h}_E)}{\partial {\nu}}, \partial_{t}\mathbf{h}_E\rangle
+\int_{\Gamma} \langle\frac{\partial{\bf h}_E}{\partial {\nu}}, \partial_{t}\mathbf{h}_E\rangle\nonumber\\
&=I+II.
\end{align}
It is easy to see that
\begin{align*}
|II|&\le C\|\frac{\partial{\bf h}_E}{\partial\nu}\|_{H^\frac12(\Gamma)}\|\partial_t{\bf h}\|_{H^{-\frac12}(\Gamma)}\\
&\le C\|{\bf h}_E\|_{H^2(\Omega)}\|\partial_t{\bf h}\|_{H^{\frac32}(\Gamma)}\\
&\le C\|{\bf h}\|_{H^\frac32(\Gamma)}\|\partial_t{\bf h}\|_{H^{\frac32}(\Gamma)}\\
&\le C\big(\|{\bf h}\|_{H^\frac32(\Gamma)}^2+\|\partial_t{\bf h}\|_{H^{\frac32}(\Gamma)}^2\big).
\end{align*}
While, by the second Green identity, we have
\begin{align*}
I&=\int_{\Gamma} \langle{\bf d}-{\bf h}_E, \frac{\partial}{\partial \nu} (\partial_{t}\mathbf{h}_E)\rangle\\
&\ \ \ +\int_\Omega \langle \Delta({\bf d}-{\bf h}_E), \partial_{t}\mathbf{h}_E\rangle
-\int_\Omega \langle {\bf d}-{\bf h}_E, \Delta(\partial_{t}\mathbf{h}_E)\rangle\\
&=\int_\Omega \langle \Delta{\bf d}, \partial_{t}\mathbf{h}_E\rangle\\
&=\int_\Omega \langle \Delta{\bf d}+|\nabla{\bf d}|^2 {\bf d}, \partial_{t}\mathbf{h}_E\rangle
-\int_\Omega \langle |\nabla{\bf d}|^2{\bf d}, \partial_{t}\mathbf{h}_E\rangle\\
&\le \frac12\int_\Omega |\Delta{\bf d}+|\nabla{\bf d}|^2{\bf d}|^2
+C\big(\int_\Omega |\partial_t{\bf h}_E|^2 +\|\partial_t {\bf h}_E\|_{L^\infty(\Omega)}\int_\Omega|\nabla{\bf d}|^2\big).
\end{align*}
For any $t\in (0,T)$, it follows from Sobolev's embedding theorem that
$$\|\partial_t {\bf h}_E(\cdot, t)\|_{L^\infty(\Omega)}\le C\|\partial_t {\bf h}_E(\cdot, t)\|_{H^2(\Omega)}
\le C\|\partial_t {\bf h}(\cdot, t)\|_{H^\frac32(\Gamma)},
$$
while
$$\int_\Omega |\partial_t{\bf h}_E(\cdot, t)|^2\le \|\partial_t{\bf h}_E(\cdot, t)\|_{H^2(\Omega)}^2
\le C\|\partial_t {\bf h}(\cdot, t)\|_{H^\frac32(\Gamma)}^2.
$$
Substituting these two estimates into I and then adding the resulting inequality with II, we obtain that
\begin{align*}
\int_{\Gamma} \langle\frac{\partial{\bf d}}{\partial {\nu}}, \partial_{t}\mathbf{h}\rangle
&\le \frac12\int_\Omega |\Delta{\bf d}+|\nabla{\bf d}|^2{\bf d}|^2+C\big(\|{\bf h}\|_{H^\frac32(\Gamma)}^2+\|\partial_t{\bf h}\|_{H^{\frac32}(\Gamma)}^2\big)\\
&\ \ \ +C\|\partial_t {\bf h}(\cdot, t)\|_{H^\frac32(\Gamma)}\int_\Omega|\nabla{\bf d}|^2.
\end{align*}
Putting this estimate into \eqref{bdry-contr}, we achieve
\begin{align*}
&\frac{d}{dt}\int_{\Omega} \frac12(|\mathbf{u}|^{2}+|\nabla \mathbf{d}|^{2})
+\frac12\int_{\Omega} (|\nabla \mathbf{u}|^{2}+|\Delta\mathbf{d}+|\nabla{\bf d}|^2{\bf d}|^{2})\\
&\le C\big(\|{\bf h}\|_{H^\frac32(\Gamma)}^2+\|\partial_t{\bf h}\|_{H^{\frac32}(\Gamma)}^2\big)
+C\|\partial_t {\bf h}(\cdot, t)\|_{H^\frac32(\Gamma)}\int_\Omega|\nabla{\bf d}|^2.
\end{align*}
Integrating this inequlaity over $[0,t]$ and applying the Gronwall's inequality yields
\eqref{energy1}. This completes the proof. 
\end{proof}

Next we will establish both interior and boundary generalized local energy inequalities 
for the system \eqref{eq1.1} -\eqref{eq1.3}. More precisely,

\begin{lemma}\label{lem11}
For $T>0$, assume $\mathbf{u}\in L^{\infty}([0,T], \mathbf{H})\cap L^{2}([0,T], \mathbf{V})$,
$\mathbf{d}\in L^{\infty}_tH^{1}_x(Q_T,\mathbb{S}^{2})\cap L^2_tH^2_x(Q_T,\mathbb S^2)$,
and $P\in L^{\frac{4}{3}}_tW^{1,\frac43}_x(Q_T)$ is a weak solution to the system \eqref{eq1.1}---\eqref{eq1.3}. 
Then, for any nonnegative $\phi\in C_{0}^{\infty}(\Omega)$ and  $0<s<t\leq T$, it holds that
\begin{align} \label{localenergy}
&\int_{\Omega}\! \phi(|\mathbf{u}|^{2}\!+\!|\nabla \mathbf{d}|^{2})(\cdot,t)
+2\!\!\int_{s}^{t}\!\int_{\Omega}\! \phi(|\nabla \mathbf{u}|^{2}+|\Delta\mathbf{d}+|\nabla\mathbf{d}|^{2}\mathbf{d}|^{2})\nonumber\\
&\leq \int_{\Omega} \phi(|\mathbf{u}|^{2}+|\nabla \mathbf{d}|^{2})(\cdot,s)
\nonumber\\
&+\int_{s}^{t}\!\int_{\Omega}\!|\nabla\phi|(|\mathbf{u}|^{3}\!+\!2|\nabla \mathbf{u}||\mathbf{u}|\!+\!
2|P\!-\!P_{\Omega}||\mathbf{u}|\!+\!|\nabla\mathbf{d}|^{2}|\mathbf{u}|\!
+\!2|\partial_{t}\mathbf{d}||\nabla \mathbf{d}|),
\end{align}
where $\displaystyle P_{\Omega}=\frac{1}{|\Omega|}\int_\Omega P$ is the average of $P$ over $\Omega$.
\end{lemma}

\begin{proof}
This proof is exactly same  as that of \cite{LLW} Lemma 4.2.  For reader's convenience, we sketch it here.
Multiplying $\eqref{eq1.1}_{1}$ by $\mathbf{u}\phi$ and integrating over $\Omega$ yields
\begin{align}\label{localforu}
&\frac{1}{2}\frac{d}{dt}\!\!\int_{\Omega} \!|\mathbf{u}|^{2}\phi +
\!\!\int_{\Omega}\! |\nabla \mathbf{u}|^{2}\phi\nonumber\\
&=\!\!
\int_{\Omega}\! (\frac{1}{2}|\mathbf{u}|^{2}\mathbf{u}\!-\langle\nabla\mathbf{u}, \mathbf{u}\rangle
+\!(P\!-\!P_{\Omega})\mathbf{u}+\!\frac{1}{2}|\nabla\mathbf{d}|^{2}\mathbf{u})\!\cdot\!\nabla\phi
\!-\langle\mathbf{u}\cdot\nabla\mathbf{d}, \Delta\mathbf{d}\rangle\phi.
\end{align}
On the other hand, multiplying $\eqref{eq1.1}_{3}$ by $-(\Delta \mathbf{d}+|\nabla\mathbf{d}|^{2}\mathbf{d})\phi$ and integrating over $\Omega$ implies
\begin{align} \label{localford1}
&\frac{1}{2}\frac{d}{dt}\int_{\Omega} |\nabla \mathbf{d}|^{2}\phi
+\int_{\Omega} |\Delta\mathbf{d}+|\nabla\mathbf{d}|^{2}\mathbf{d}|^{2}\phi\nonumber\\
&=\int_{\Omega} \langle\mathbf{u}\cdot\nabla \mathbf{d}, \Delta\mathbf{d}\rangle\phi
+\langle\partial_{t}\mathbf{d}, \nabla \mathbf{d}\rangle \cdot\nabla \phi.
\end{align}
Adding \eqref{localforu} with \eqref{localford1},  we obtain that
\begin{align*}
&\frac{d}{dt}\int_{\Omega} (|\mathbf{u}|^{2}+|\nabla \mathbf{d}|^{2})\phi
+2\int_{\Omega} (|\nabla \mathbf{u}|^{2}+|\Delta\mathbf{d}+|\nabla\mathbf{d}|^{2}\mathbf{d}|^{2})\phi\nonumber\\
\leq& \int_{\Omega} |\nabla\phi|(|\mathbf{u}|^{3}+2|\nabla \mathbf{u}||\mathbf{u}|+
2|P-P_{\Omega}||\mathbf{u}|+|\nabla\mathbf{d}|^{2}|\mathbf{u}|+2|\partial_{t}\mathbf{d}||\nabla \mathbf{d}|).
\end{align*}
\eqref{localenergy} follows by integrating this inequality over $[s,t]$.
\end{proof}

Next we will state the local generalized boundary energy inequality, whose proof is 
more delicate than 
\cite{LLW} Lemma 4.3.
\begin{lemma}\label{lem12}
For $T>0$, ${\bf h}\in L^{2}_tH^{\frac{3}{2}}_x(\Gamma_T,\mathbb{S}^{2}))$, 
$\partial_{t}\mathbf{h}\in L^{2}_tH^{\frac{3}{2}}_x(\Gamma_T)$ and
$(\mathbf{u}_{0},\mathbf{d}_{0})\in \mathbf{H}\times H^{1}(\Omega,\mathbb{S}^{2})$, 
assume $\mathbf{u}\in L^{\infty}([0,T], \mathbf{H})\cap L^{2}([0,T], \mathbf{V})$,
$\mathbf{d}\in L^{\infty}_tH^{1}_x(Q_T,\mathbb{S}^{2})\cap L^2_tH^2_x(Q_T,\mathbb S^2)$,
and $P\in L^{\frac{4}{3}}_tW^{1,\frac43}_x(Q_T)$ is a weak solution to the system \eqref{eq1.1}--\eqref{eq1.3}. There exists $r_{0}=r_{0}(\Gamma)>0$ such that for any ${x}_{0}\in\Gamma$, $0<r\leq r_{0}$, $0<s<t\leq T$,
if $0\le \phi\in C_{0}^{\infty}(B_{r}(x_{0}))$ then 
\begin{align}\label{bdyenergy}
&\int_{B_{r}^+({x}_{0})}\!\phi(|\mathbf{u}|^{2}\!+\!|\nabla \widehat{\mathbf{d}}|^{2})(\cdot,t)
+\!\!\int_{s}^{t}\!\int_{B_{r}^+({x}_{0})}\! \phi(|\nabla \mathbf{u}|^{2}+|\Delta\mathbf{d}+|\nabla\mathbf{d}|^{2}\mathbf{d}|^{2})
\nonumber\\
&\leq
\int_{B_{r}^+({x}_{0})}\!\! \phi(|\mathbf{u}|^{2}\!+\!|\nabla \widehat{\mathbf{d}}|^{2})(\cdot,s)
+\!\int_{s}^{t}\!\int_{B_{r}^+({x}_{0})}\!(|\nabla {\bf d}|^2|\partial_{t}\mathbf{h}_{E}|+|\partial_{t}\mathbf{h}_{E}|^{2})\phi
\nonumber\\
&+\!\int_{s}^{t}\!\int_{B_{r}^+({x}_{0})}\!|\nabla\phi|[
|{\bf u}|(|{\bf u}|^2+|\nabla \mathbf{u}|+|P\!-\!P_{\Omega}|+|\nabla\mathbf{d}|^{2})
+|\partial_{t}\widehat{\mathbf{d}}||\nabla \widehat{\mathbf{d}}|],
\end{align}
where  $\mathbf{h}_{E}(\cdot, t)$ is the harmonic  extension of ${\bf h}(\cdot,t)$ for $0<t\le T$, and
$\widehat{\mathbf{d}}=\mathbf{d}-\mathbf{h}_{E}$.
\end{lemma}

\begin{proof}
Multiplying $\eqref{eq1.1}_{1}$ by $\mathbf{u}\phi$, integrating over $B_{r}^+({x}_{0})$,
and using $\mathbf{u}\phi={0}$ on $\partial B_{r}^+({x}_{0})$, we obtain that
\begin{align}\label{bdyforu}
&\frac{1}{2}\frac{d}{dt}\!\int_{B_{r}^+({x}_{0})} |\mathbf{u}|^{2}\phi +
\!\int_{B_{r}^+({x}_{0})} |\nabla \mathbf{u}|^{2}\phi\\
&=\!
\int_{B_{r}^+({x}_{0})} (\frac{1}{2}|\mathbf{u}|^{2}\mathbf{u}-\langle\nabla\mathbf{u}, \mathbf{u}\rangle
+(P-P_{\Omega})\mathbf{u}+\!\frac{1}{2}|\nabla\mathbf{d}|^{2}\mathbf{u})\!\cdot\!\nabla\phi\nonumber\\
&\ \ \ \ \ \ \ \ \ \ \ \ \ \ -\langle\mathbf{u}\cdot\nabla\mathbf{d}, \Delta\mathbf{d}\rangle\phi.\nonumber
\end{align}
Let $\widehat{\mathbf{d}}=\mathbf{d}-\mathbf{h}_{E}$. 
Then
\begin{align*}
\partial_{t}
\widehat{\mathbf{d}}\phi={0}\ \ {\rm{on}}\ \ \partial B_{r}^+({x}_{0}),
\end{align*}
and
\begin{align*}
&-\!\!\int_{B_{r}^+({x}_{0})}\langle\partial_{t}\mathbf{d},
\Delta\mathbf{d}\!+\!|\nabla\mathbf{d}|^{2}\mathbf{d}\rangle\phi=
-\!\!\int_{B_{r}^+({x}_{0})}\langle\partial_{t}\mathbf{d},\Delta\mathbf{d}\rangle\phi\\
&=-\!\!\int_{B_{r}^+({x}_{0})}\langle\partial_{t}\widehat{\mathbf{d}}+ \partial_{t}\mathbf{h}_{E},
\Delta\mathbf{d}\rangle\phi
\nonumber\\
&=-\!\!\int_{B_{r}^+({x}_{0})}\langle \partial_{t}\widehat{\mathbf{d}},
\Delta\widehat{\mathbf{d}}\rangle\phi 
-\int_{B_{r}^+({x}_{0})}\langle\partial_{t}\mathbf{h}_{E},\Delta\mathbf{d}\rangle\phi
\nonumber\\
&=\frac{1}{2}\frac{d}{dt}\int_{B_{r}^+({x}_{0})}
|\nabla\widehat{\mathbf{d}}|^{2}\phi
+\!\int_{B_{r}^+({x}_{0})}\langle\partial_{t}\widehat{\mathbf{d}},\nabla \widehat{\mathbf{d}}\rangle\cdot\!\nabla \phi
-\!\!\int_{B_{r}^+({x}_{0})}\langle \partial_{t}\mathbf{h}_{E}, \Delta\mathbf{d}\rangle\phi.
\end{align*}
Hence, after multiplying $\eqref{eq1.1}_{3}$ by $-(\Delta \mathbf{d}+|\nabla\mathbf{d}|^{2}\mathbf{d})\phi$ 
and integrating over $B_{r}^+({x}_{0})$, we have that
 \begin{align} \label{bdyford1}
&\frac{1}{2}\frac{d}{dt}\int_{B_{r}^+({x}_{0})} |\nabla \widehat{\mathbf{d}}|^{2}\phi
+\int_{B_{r}^+({x}_{0})} |\Delta\mathbf{d}+|\nabla\mathbf{d}|^{2}\mathbf{d}|^{2}\phi\nonumber\\
&=\int_{B_{r}^+({x}_{0})} [\langle\mathbf{u}\cdot\nabla \mathbf{d}, \Delta\mathbf{d}\rangle\phi
-\langle\partial_{t}\widehat{\mathbf{d}}, \nabla \widehat{\mathbf{d}}\rangle\cdot\nabla \phi
+\langle \partial_{t}\mathbf{h}_{E}, \Delta\mathbf{d}\rangle\phi]\nonumber\\
&=\int_{B_{r}^+({x}_{0})} [\langle\mathbf{u}\cdot\nabla \mathbf{d}, \Delta\mathbf{d}\rangle\phi
-\langle\partial_{t}\widehat{\mathbf{d}}, \nabla \widehat{\mathbf{d}}\rangle\cdot\nabla \phi
+\langle \partial_{t}\mathbf{h}_{E}, \Delta\mathbf{d}+|\nabla{\bf d}|^2{\bf d}\rangle\phi]\nonumber\\
&\ \ - \int_{B_r^+(x_0)}\langle \partial_{t}\mathbf{h}_{E}, |\nabla{\bf d}|^2{\bf d}\rangle\phi.
\end{align}
It is readily seen that \eqref{bdyenergy} follows by adding \eqref{bdyforu} with \eqref{bdyford1} 
and applying H\"{o}lder's inequality. The proof of  Lemma  \ref{lem12} is  now complete.
\end{proof}

We also need the following Lemma  on the estimate of pressure function $P$
that is assumed in both Lemma \ref{lem4} and  Lemma \ref{lem5}.

\begin{lemma}\label{lem13}
For $T>0$, assume $\mathbf{u}\!\in\! L^{\infty}([0,T],\mathbf{H})\cap L^{2}([0,T],\mathbf{V})$,
$\mathbf{d}\!\in\! L^{\infty}_tH^{1}_x(Q_T, \mathbb{S}^{2})\cap L^2_tH^2_x(Q_T,\mathbb S^2)$,
and $P\in L^{\frac{4}{3}}_tW^{1,\frac43}_x(Q_T)$ is a weak solution to 
the system \eqref{eq1.1}--\eqref{eq1.3}. Then it holds that for any $0<t\le T$,
\begin{align*}
&\max\!\big\{\|\nabla P\|_{L^\frac43(Q_{t})}, 
\|P\!-\!P_{\Omega}\|_{\!L^\frac43_tL^{4}_x(Q_t)}\big\}\\
&\!\leq C\big( \|\mathbf{u}\|_{L^{4}(Q_{t})} \|\nabla \mathbf{u}\|_{L^{2}(Q_{t})}
+\|\nabla \mathbf{d}\|_{L^{4}(Q_{t})}
\|\nabla^{2}\mathbf{d}\|_{L^{2}(Q_{t})}\big).
\end{align*}

\end{lemma}

\subsection{Proof of Theorem \ref{thm1}}

In this subsection, we will establish the existence of a global weak solution to \eqref{eq1.1}--\eqref{eq1.3}. 
Let us first recall the following version of Ladyzhenskaya's inequality (see Struwe \cite{S1985} Lemma 3.1).

\begin{lemma}\label{lem14}
There exist $M_{0}>0$ and $r_{0}>0$ depending only on $\Omega$ such that for any $T>0$, 
if $f\in L^{\infty}([0,T],L^{2}(\Omega))\cap L^{2}([0,T],H^{1}(\Omega))$ then for $r\in(0,r_{0})$ it holds that for any $0<t\le T$
\begin{align*}
\int_{Q_{t}} |f|^{4}\leq M_{0}\sup_{(x,t)\in Q_{t}} \int_{\Omega\cap B_{r}(x)}|f|^{2} \big(\int_{Q_{t}} |\nabla f|^{2}+\frac{1}{r^{2}}\int_{Q_{t}} |f|^{2}\big).
\end{align*}
\end{lemma}

Next we will show a lower bound estimate of the lift span of the short time smooth solutions in terms of
the local energy profile of the initial and boundary data. More precisely, we have

\begin{lemma}\label{lem15}
Let $\Omega\subset\mathbb R^2 $ be a bounded smooth domain,  $0<T<+\infty$,
${\bf u}_0\in C^{2,\alpha}(\overline\Omega, \mathbb R^2)$,
${\bf d}_0\in C^{2,\alpha}(\overline\Omega, \mathbb S^2)$,
and $\mathbf{h}\in C^{2+\alpha, 1+\frac{\alpha}2}(\Gamma_T,\mathbb S^2)$ satisfy
\eqref{comp_cond}.
Let $\varepsilon_{0}>0$ be the smaller constant
given by Lemma \ref{lem4} and Lemma \ref{lem5}. Then there exist $0<\varepsilon_{1}<\varepsilon_{0}$ 
and $$0<\theta_{0}= \theta_{0}\big(\varepsilon_{1},\|({\bf u}_0, \nabla {\bf d}_0)\|_{L^2(\Omega)},
\|(\mathbf{h},\partial_t{\bf h})\|_{L^2_tH^\frac32_x(\Gamma_{T})}\big)$$
 such that if $0<r_0<\varepsilon_1^4$ satisfies
 \begin{align*}
 \sup_{{x}\in\overline\Omega} \int_{\Omega\cap B_{2r_{0}}({x})}
 (|\mathbf{u}_{0}|^{2}+|\nabla
 \mathbf{d}_{0}|^{2})\leq\varepsilon_{1}^{2},
 \end{align*}
then there exist $T_{0}\geq \theta_{0}r_{0}^{2}$ and a unique solution 
$$(\mathbf{u}, \mathbf{d})\in C^{\infty}(Q_{T_0},
\mathbb{R}^{2}\times\mathbb{S}^{2})\cap C^{2+\alpha,1+\frac{\alpha}{2}}(\overline\Omega
\times [0, T_{0}], \mathbb{R}^{2}\times\mathbb{S}^{2})$$
to the system \eqref{eq1.1}, \eqref{eq1.2} and \eqref{eq1.3}.
Furthermore, it holds that
\begin{align}\label{prioriestimate}
\sup_{({x},t)\in\overline{\Omega}\times [0, T_0]}
 \int_{\Omega\cap B_{r_{0}}({x})}(|\mathbf{u}|^{2}+|\nabla \mathbf{d}|^{2})
 (\cdot,t) \leq 2\varepsilon_{1}^{2}.
\end{align}
\end{lemma}

\begin{proof}
Since  $\mathbf{h}\in C^{2+\alpha,1+\frac{\alpha}{2}}(\Gamma_{T}, \mathbb{S}^{2})$ and
$(\mathbf{u}_{0},\mathbf{d}_{0})\in C^{2,\alpha}(\overline\Omega, \mathbb{R}^{2}\times \mathbb{S}^{2})$, 
Theorem \ref{thm6} implies that there exist $0<T_{0}\le T$ and a unique smooth solution
\begin{align*}
(\mathbf{u}, \mathbf{d})\in C^{\infty}(Q_{T_0}, \mathbb{R}^{2}\times\mathbb{S}^{2})
 \cap C^{2+\alpha,1+\frac{\alpha}{2}}(\overline\Omega\times [0, T_{0}], \mathbb{R}^{2}\times\mathbb{S}^{2})
\end{align*}
to the system \eqref{eq1.1}--\eqref{eq1.3}. Let $0<t_{0}\leq T_{0}$ be the maximal time such that
\begin{align}\label{maximaltime}
\sup_{0\leq t\leq t_{0}}\sup_{{x}\in\overline{\Omega}}
 \int_{\Omega\cap B_{r_{0}}({x})}(|\mathbf{u}|^{2}+|\nabla \mathbf{d}|^{2})(\cdot,t) \leq 2\varepsilon_{1}^{2}.
\end{align}
Then we must have that
\begin{align*}
\sup_{{x}\in\overline{\Omega}}
 \int_{\Omega\cap B_{r_{0}}({x})}(|\mathbf{u}|^{2}+|\nabla \mathbf{d}|^{2})(\cdot,t_{0}) = 2\varepsilon_{1}^{2}.
\end{align*}
In what follows, we will estimate the lower bound of $t_{0}$. Without loss of generality, 
we may assume $t_{0}\leq r_{0}^{2}$.  Denote by
$$\mathcal{E}(t)=\int_{\Omega}(|\mathbf{u}|^{2}+
 |\nabla\mathbf{d}|^{2})(\cdot,t) \ {\rm{for}}\  0<t\le T, \ {\rm{and}}\ \mathcal{E}_{0}
 =\int_\Omega(|\mathbf{u}_{0}|^{2}+|\nabla\mathbf{d}_{0}|^{2}).$$
From Lemma \ref{lem10}, we have that for $0<t\leq t_{0}$
\begin{align}\label{energytot}
\mathcal{E}(t)+\int_{Q_t} (|\nabla\mathbf{u}|^{2}+|\Delta \mathbf{d}+|\nabla{\bf d}|^2{\bf d}|^{2})
&\leq \psi(T)\big(\mathcal{E}_{0}+C\|(\mathbf{h},\partial_t {\bf h})\|_{L^2_tH^\frac32_x(\Gamma_T)}^{2}\big)\nonumber\\
&\leq C(T),
\end{align}
where $C(T)>0$ depends on $T, \mathcal{E}_0$, $\|(\mathbf{h},\partial_t {\bf h})\|_{L^2_tH^\frac32_x(\Gamma_T)}$,
and
$$\psi(T)\equiv\exp\big(C\int_0^T \|\partial_t{\bf h}(\tau)\|_{H^\frac32(\Gamma)}\,d\tau\big).$$
Hence by Lemma \ref{lem14} and \eqref{energytot} we have that for $0<t\le t_0 \le r_0^2$, 
\begin{align*}
&\int_{Q_{t}}|\nabla \mathbf{d}|^{4}\leq M_{0} 
\sup_{({x},\tau)\in Q_{t}}
\int_{\Omega\cap B_{r_{0}}({x})}|\nabla\mathbf{d}|^{2}(\tau)
\Big(\int_{Q_{t}}|\Delta\mathbf{d}|^{2}+ \frac{1}{r_{0}^{2}} \int_{Q_t} |\nabla\mathbf{d}|^{2}\Big)
\nonumber\\
&\leq  2M_{0}\varepsilon_1^2
\Big(\int_{Q_{t}} |\Delta\mathbf{d}+|\nabla{\bf d}|^2{\bf d}|^{2}+\int_{Q_t}|\nabla{\bf d}|^4+
\frac{C(T)t}{ r_{0}^{2}}\Big)\\
&\leq CM_0\varepsilon_1^2\big(C(T)+\int_{Q_t}|\nabla{\bf d}|^4\big),
\end{align*}
which implies that
\begin{align}\label{L4ford}
&\int_{Q_{t}}|\nabla \mathbf{d}|^{4}\leq \frac{C(T)\varepsilon_1^2}
{1-C(T)\varepsilon_{1}^{2}}\leq C(T)\varepsilon_1^2,
\end{align}
provided 
$0<\varepsilon_{1}^{2}<\frac{1}{2C(T)}$.

It follows from \eqref{L4ford} and \eqref{energytot} that 
\begin{align}\label{h2bound}
\int_{Q_t}|\nabla^2{\bf d}|^2\le C(T), \ \forall\ t\in (0,T].
\end{align}

On the other hand, we  can  estimate
\begin{align}\label{L4foru}
&\int_{Q_{t}}| \mathbf{u}|^{4}\leq M_{0} 
\sup_{({x},\tau)\in Q_{t}}
\int_{\Omega\cap B_{r_{0}}({x})}|\mathbf{u}|^{2}(\tau)
\big(\int_{Q_{t}}|\nabla\mathbf{u}|^{2}+ \frac{1}{r_{0}^{2}} \int_{Q_t} |\mathbf{u}|^{2}\big)
\nonumber\\
&\leq  2M_{0}\varepsilon_1^2
(\int_{Q_{t}} |\nabla\mathbf{u}|^{2}+\frac{C(T)t}{ r_{0}^{2}})
\nonumber\\
&\leq C(T)\varepsilon_{1}^{2}.
\end{align}
It follows from \eqref{energytot}, \eqref{L4ford}, \eqref{h2bound}, \eqref{L4foru},
and Lemma \ref{lem13} that 
\begin{align}\label{p-bound}
\big\|P-P_\Omega\big\|_{L^\frac43_tL^4_x(Q_t)}\le C(T)\varepsilon_1^\frac12, 
\ \forall\ t\in (0,T].
\end{align}
From  $\partial_t {\bf d}=-{\bf u}\cdot\nabla{\bf d}+(\Delta{\bf d}+|\nabla{\bf d}|^2{\bf d})$
and \eqref{energytot}, \eqref{L4ford}, \eqref{L4foru}, we have that
\begin{align}\label{l2fordt}
\big\|\partial_t{\bf d}\big\|_{L^2(Q_t)}&\leq C\big(\|{\bf u}\|_{L^4(Q_t)}\|\nabla{\bf d}\|_{L^4(Q_t)}
+\|\Delta{\bf d}+|\nabla{\bf d}|^2{\bf d}\|_{L^2(Q_t)}\big)\nonumber\\
&\le C(T), \ \forall\ t\in (0,T].
\end{align}

Now we are ready to refine the estimate of the quantity 
$$\displaystyle\max_{x\in\overline\Omega}\int_{\Omega\cap B_{r_0}(x)}
(|{\bf u}|^2+|\nabla{\bf d}|^2)(t), \ 0\le t\le t_0.$$ 
To do it, for any ${x}\in\overline{\Omega}$, let $\phi\in C_{0}^{\infty}(B_{2r_{0}}({x}))$ 
be a cut-off function of $B_{r_{0}}({x})$ such that
\begin{align*}
0\leq \phi\leq 1;\ \phi=1 \text{ in } B_{r_{0}}({x});\  \phi=0 \text{ outside } B_{2r_{0}}({x}); 
\ {\rm{and}}\ |\nabla\phi|\leq \frac{4}{r_{0}}.
\end{align*}
Applying  Lemma \ref{lem11}, we see that for any $B_{2r_0}(x)\subset\Omega$, it holds that
\begin{align} \label{Eng}
&\sup_{0\leq t\leq t_0} \int_{B_{r_0}(x)} (|\mathbf{u}|^{2}+|\nabla \mathbf{d}|^{2})
-\int_{B_{2r_0}(x)}(|{\bf u}_0|^2+|\nabla{\bf d}_0|^2)\nonumber\\
&\leq \sup_{0\leq t\leq t_0} \int_{B_{2r_0}(x)} \phi(|\mathbf{u}|^{2}+|\nabla \mathbf{d}|^{2})
-\int_{B_{2r_0}(x)}\phi(|{\bf u}_0|^2+|\nabla{\bf d}_0|^2)\nonumber\\
&\leq C\int_{0}^{t_{0}}\int_{B_{2r_{0}}({x})}|\nabla\phi| (|\mathbf{u}|^{3}+|\nabla\mathbf{u}||\mathbf{u}|+|P-P_{\Omega}||\mathbf{u}|\nonumber\\
&\qquad\qquad\qquad\qquad\ \ \ \ \ \ \ \ +|\nabla\mathbf{d}|^{2}|\mathbf{u}|+|\partial_{t}\mathbf{d}||\nabla\mathbf{d}|)
\nonumber\\
&\leq C(\frac{t_{0}}{r_{0}^{2}})^{\!\frac{1}{4}}\!
\Big[\!\|\mathbf{u}\|_{\!L^{4}(Q_{t_{0}})}^{3}\!+\!\|\nabla \mathbf{u}\|_{\!L^{2}(Q_{t_{0}})} \|\mathbf{u}\|_{\!L^{4}(\Omega_{t_{0}})}\nonumber\\
&\qquad\qquad+\|\nabla\mathbf{d}\|_{\!L^{4}(Q_{t_{0}})}^{2} \|\mathbf{u}\|_{\!L^{4}(Q_{t_{0}})}
+\|\partial_{t}\mathbf{d}\|_{\!L^{2}(Q_{t_{0}})}
\|\nabla\mathbf{d}\|_{\!L^{4}(Q_{t_{0}})}
\nonumber\\
&\ \ \ \qquad\quad+\big\|P-P_{\Omega}\big\|_{L^{\frac43}_tL^{4}_x(Q_{t_0})}
\|\mathbf{u}\|_{L^{\infty}_tL^2_x(B_{2r_0}(x)\times [0, t_0])}\Big]\nonumber\\
&\leq C(\frac{t_{0}}{r_{0}^{2}})^{\!\frac{1}{4}}\varepsilon_1^\frac12,
\end{align}
where we have used \eqref{maximaltime}, \eqref{energytot}, \eqref{L4ford},
\eqref{L4foru}, \eqref{p-bound}, and \eqref{l2fordt} in the last step.

For $B_{2r_0}(x_0)\cap\Gamma\not=\emptyset$, we can apply 
\eqref{bdyenergy} of Lemma \ref{lem12} to get that
\begin{align}\label{Eng0}
&\sup_{0\le t\le t_0}\int_{\Omega\cap B_{r_0}(x_0)} (|{\bf u}|^2+|\nabla\widehat{{\bf d}}|^2)
-\int_{\Omega\cap B_{2r_0}(x_0)} (|{\bf u}_0|^2+|\nabla\widehat{{\bf d}}_0|^2)\nonumber\\
&\leq \sup_{0\le t\le t_0}\int_{\Omega\cap B_{2r_0}(x_0)} \phi(|{\bf u}|^2+|\nabla\widehat{{\bf d}}|^2)
-\int_{\Omega\cap B_{2r_0}(x_0)} \phi(|{\bf u}_0|^2+|\nabla\widehat{{\bf d}}_0|^2)\nonumber\\
&\leq \int_0^{t_0}\int_{\Omega\cap B_{2r_0}({x}_{0})}
\phi (|\nabla {\bf d}|^2|\partial_{t}\mathbf{h}_{E}|+|\partial_{t}\mathbf{h}_{E}|^{2})
+|\nabla\phi||\partial_{t}\widehat{\mathbf{d}}||\nabla \widehat{\mathbf{d}}|
\nonumber\\
&\ \ +\!\int_{0}^{t_0}\int_{\Omega\cap B_{2r_0}({x}_{0})}\!|\nabla\phi||{\bf u}|(|\mathbf{u}|^2
+|\nabla \mathbf{u}|+|P\!-\!P_{\Omega}|+|\nabla\mathbf{d}|^{2})\nonumber\\
&=I+II+III.
\end{align}
As in \eqref{Eng}, we can estimate $III$ by
$$|III|\le C(\frac{t_0}{r_0^2})^\frac14 \varepsilon_1^\frac12.$$ 
From $\partial_t {\bf h}_E\in L^2_tH^2_x(Q_T)$ and 
the Sobolev embedding theorem, we have that $\partial_t {\bf h}_E\in L^2_tL^\infty_x(Q_T)$,
and
$$\big\|\partial_t {\bf h}_E\big\|_{L^2_tL^\infty_x(Q_T)}
\leq C\big\|\partial_t {\bf h}_E\big\|_{L^2_tH^2_x(Q_T)}
\leq C\big\|\partial_t {\bf h}\big\|_{L^2_tH^{\frac32}_x(\Gamma_T)}.$$ 
Since ${\bf h}\in L^2_tH^\frac32_x(\Gamma_T)$ and $\partial_t{\bf h}\in L^2_tH^\frac32_x(\Gamma_T)$,
 ${\bf h}\in C([0,T], H^\frac32(\Gamma))$ and
$$\big\|{\bf h}\big\|_{L^\infty_tH^\frac32_x(\Gamma_T)}
\le C\big(T, \|{\bf h}\|_{L^2_tH^\frac32_x(\Gamma_T)}, \|\partial_t{\bf h}\|_{L^2_tH^\frac32_x(\Gamma_T)}\big).$$
This, combined with the fact that ${\bf h}_E(\cdot, t)$ is a harmonic extension of $h(\cdot,t)$
for $t\in [0,T]$, implies that ${\bf h}_E\in L^\infty_tH^2_x(Q_T)$ and hence by Sobolev embedding theorem we 
obtain that
\begin{align}\label{he-est}
\big\|\nabla{\bf h}_E\big\|_{L^\infty_tL^4_x(Q_T)}&\leq C
\big\|{\bf h}_E\big\|_{L^\infty_tH^2_x(Q_T)}\le C\big\|{\bf h}\big\|_{L^\infty_tH^\frac32_x(\Gamma_T)}\nonumber\\
&\le C\big(T, \|{\bf h}\|_{L^2_tH^\frac32_x(\Gamma_T)}, \|\partial_t{\bf h}\|_{L^2_tH^\frac32_x(\Gamma_T)}\big).
\end{align}
We also have that
$$\|\nabla {\bf h}_E(t)\|_{L^2(\Omega)}\le \|\nabla {\bf d}(t)\|_{L^2(\Omega)},
\ 0\le t\le T.$$
Hence 
\begin{align*}
|I|&\leq C\|\partial_t{\bf h}_E\|_{L^2_tL^\infty_x(Q_{t_0})}
\big(\sup_{0\le t\le t_0}\int_{\Omega\cap B_{2r_0}(x_0)}|\nabla{\bf d}|^2\big)t_0^\frac12\\
&\ \ +C\|\partial_t{\bf h}_E\|_{L^2_tL^\infty_x(Q_{t_0})}^2 r_0^2\\
&\leq C\big(\|\partial_t{\bf h}\big\|_{L^2_tH^\frac32_x(\Gamma_T)}\varepsilon_1^2 t_0^\frac12
+\|\partial_t{\bf h}\big\|_{L^2_tH^\frac32_x(\Gamma_T)}^2r_0^2\big)\\
&\le C(\varepsilon_1^2r_0+r_0^2).
\end{align*}
While $II$ can be estimated as follows.
\begin{align*}
|II|&\le C\int_{0}^{t_0}\int_{\Omega\cap B_{2r_0}(x_0)}
|\nabla\phi|(|\partial_t {\bf d}||\nabla{\bf d}|+|\partial_t {\bf d}||\nabla{\bf h}_E|\\
&\ \ \ \qquad\qquad\qquad\qquad\qquad+|\partial_t {\bf h}_E||\nabla{\bf d}|+|\partial_t {\bf h}_E||\nabla{\bf h}_E|)\\
&\le C(\frac{t_0}{r_0^2})^\frac14\|\partial_t{\bf d}\|_{L^2(Q_{t_0})}\|\nabla{\bf d}\|_{L^4(Q_{t_0})}\\
&\ \ +C\frac{t_0^\frac12}{r_0}\|\partial_t{\bf d}\|_{L^2(Q_{t_0})}\|\nabla{\bf h}_E\|_{L^\infty_tL^2_x(Q_{t_0})}\\
&\ \ +C {t_0}^\frac12\|\partial_t{\bf h}_E\|_{L^2_tL^\infty_x(Q_{t_0})}\|\nabla{\bf d}\|_{L^\infty_tL^2_x(Q_{t_0})}\\
&\leq C\big[\varepsilon_1^\frac12(\frac{t_0}{r_0^2})^\frac14+\frac{t_0^\frac12}{r_0}+{t_0}^\frac12\big]
\le C\big[\varepsilon_1^\frac12(\frac{t_0}{r_0^2})^\frac14+t_0^\frac12\big].
\end{align*}
Putting these estimates of $I, II,$ and $III$ into \eqref{Eng0} yields that
\begin{align}\label{Eng01}
&\sup_{0\le t\le t_0}\int_{\Omega\cap B_{r_0}(x_0)} (|{\bf u}|^2+|\nabla\widehat{{\bf d}}|^2)
-\int_{\Omega\cap B_{2r_0}(x_0)} (|{\bf u}_0|^2+|\nabla\widehat{{\bf d}}_0|^2)\nonumber\\
&\leq C\big[t_0^\frac12+\varepsilon_1^\frac12(\frac{t_0}{r_0^2})^\frac14\big].
\end{align}
Applying \eqref{he-est}, we can estimate
\begin{align*}&\int_{\Omega\cap B_{r_0}(x_0)} (|{\bf u}|^2+|\nabla\widehat{{\bf d}}|^2)\\
&\ge 
\frac45\int_{\Omega\cap B_{r_0}(x_0)} (|{\bf u}|^2+|\nabla{\bf d}|^2)
-C\int_{\Omega\cap B_{r_0}(x_0)} |\nabla{\bf h}_E|^2\\
&\ge \frac45\int_{\Omega\cap B_{r_0}(x_0)} (|{\bf u}|^2+|\nabla{\bf d}|^2)
-C\|{\bf h}_E\|_{L^\infty_tH^2_x(Q_T)}^2r_0\\
&\ge \frac45\int_{\Omega\cap B_{r_0}(x_0)} (|{\bf u}|^2+|\nabla{\bf d}|^2)
-Cr_0, \ \forall \ t\in [0,T],
\end{align*}
and
\begin{align*}
&\int_{\Omega\cap B_{2r_0}(x_0)} (|{\bf u}_0|^2+|\nabla\widehat{{\bf d}}_0|^2)\\
&\leq \frac54\int_{\Omega\cap B_{2r_0}(x_0)} (|{\bf u}_0|^2+|\nabla {\bf d}_0|^2)
+C\int_{\Omega\cap B_{2r_0}(x_0)} |\nabla{\bf h}_E|^2\\
&\leq \frac54\int_{\Omega\cap B_{2r_0}(x_0)} (|{\bf u}_0|^2+|\nabla {\bf d}_0|^2)
+Cr_0.
\end{align*}
Therefore we obtain
\begin{align}\label{Eng02}
&\sup_{0\le t\le t_0}\int_{\Omega\cap B_{r_0}(x_0)} (|{\bf u}|^2+|\nabla{\bf d}|^2)\nonumber\\
&\le (\frac54)^2\int_{\Omega\cap B_{2r_0}(x_0)} (|{\bf u}_0|^2+|\nabla{\bf d}_0|^2)
+C\big[r_0+t_0^\frac12+\varepsilon_1^\frac12(\frac{t_0}{r_0^2})^\frac14\big]\nonumber\\
&\le (\frac54)^2 \varepsilon_1^2+C\big[r_0+t_0^\frac12+\varepsilon_1^\frac12(\frac{t_0}{r_0^2})^\frac14\big].
\end{align}
Combining \eqref{Eng} with \eqref{Eng02}, we obtain that
\begin{align}\label{Eng03}
2\epsilon_1^2&=\sup_{0\le t\le t_0}\max_{x_0\in\overline\Omega}\int_{\Omega\cap B_{r_0}(x_0)} (|{\bf u}|^2+|\nabla{\bf d}|^2)\nonumber\\
&\le (\frac54)^2\varepsilon_1^2
+C\big[r_0+\varepsilon_1^\frac12(\frac{t_0}{r_0^2})^\frac14\big]\nonumber\\
&\le (\frac{25}{16}+C\varepsilon_1^2)\epsilon_1^2+C\varepsilon_1^\frac12(\frac{t_0}{r_0^2})^\frac14.
\end{align}
Therefore if we choose $\varepsilon_0\le \frac{5}{16C}$, then $t_0\ge \theta_0 r_0^2$ with
$\theta_0=\big(\frac{3\varepsilon_1^\frac32}{8C}\big)^4$. This gives the desired estimates
of $T_0$ and \eqref{prioriestimate}. The proof is now complete. 
\end{proof}

\medskip

\noindent\textbf{Proof of Theorem \ref{thm1}.}  From $\mathbf{u}_{0}\in \mathbf{H}$, there exists $\{\mathbf{u}_{0}^{k}\}
\subset C^{2, \alpha}(\overline\Omega,\mathbb{R}^{2})$ with $\nabla\cdot\mathbf{u}_{0}^{k}=0$
in $\Omega$ such that
\begin{align*}
\lim_{k\uparrow \infty} \|\mathbf{u}_{0}^{k}-\mathbf{u}_{0}\|_{L^{2}(\Omega)} =0.
\end{align*}
Since dimension of $\partial_pQ_T=\Omega\cup \Gamma_T$ is $2$, 
$\mathbf{h}\in L^{2}_tH^{\frac{3}{2}}_x(\Gamma_T,\mathbb{S}^{2})$, 
$\partial_{t}\mathbf{h}\in L^{2}_tH^{\frac{3}{2}}_x(\Gamma_T)$,
$\mathbf{d}_{0}\in H^{1}(\Omega,\mathbb{S}^{2})$, and
$\mathbf{d}_{0}|_{\Gamma}=\mathbf{h}|_{\Gamma\times\{0\}}$, there exist maps $(\mathbf{h}^{k}, \mathbf{d}_{0}^{k})$ such that $\mathbf{h}^{k}\in C^{2+\alpha,1+\frac{\alpha}{2}}(\Gamma_{T}, \mathbb{S}^{2})$ and $\mathbf{d}_{0}^{k}\in C^{2,\alpha}(\overline\Omega, \mathbb{S}^{2})$ with $\mathbf{d}_{0}^{k}|_{\Gamma}=\mathbf{h}^{k}|_{\Gamma\times\{0\}}$, and
\begin{align}\label{strong-app}
&\lim_{k\uparrow\infty} \big\|(\mathbf{h}^{k}-\mathbf{h}, \partial_{t}(\mathbf{h}^{k}-\mathbf{h}))\|_{L^{2}_tH^{\frac{3}{2}}_x(\Gamma_T)} =\lim_{k\uparrow \infty} \big\|\mathbf{d}_{0}^{k}-\mathbf{d}_{0}\big\|_{H^{1}(\Omega)} =0.
\end{align}
From the absolute continuity of $\displaystyle\int(|\mathbf{u}_{0}|^{2}+|\nabla\mathbf{d}_{0}|^{2})$, 
there exists $r_{0}\in (0,\varepsilon_1^2)$ such that
\begin{align*}
\sup_{{x}\in\overline{\Omega}} \int_{\Omega\cap B_{2r_{0}}({x})}
(|\mathbf{u}_{0}|^{2}+|\nabla\mathbf{d}_{0}|^{2})\leq \frac{\varepsilon_{1}^{2}}{2},
\end{align*}
where $\varepsilon_{1}>0$ is the constant given by Lemma \ref{lem15}. 
By the strong convergence of $(\mathbf{u}_{0}^{k}, \nabla \mathbf{d}_{0}^{k})$ to $(\mathbf{u}_{0}, \nabla\mathbf{d}_{0})$ in $L^{2}(\Omega)$, we may assume that
\begin{align}\label{initialk}
\sup_{{x}\in\overline{\Omega}} \int_{\Omega\cap B_{2r_{0}}({x})}
(|\mathbf{u}_{0}^{k}|^{2}+|\nabla\mathbf{d}_{0}^{k}|^{2})\leq \varepsilon_{1}^{2}
\quad\text{ for } k\ge 1.
\end{align}
We may also assume that 
\begin{align}\label{boundaryk}
\big\|(\mathbf{h}^{k}, \partial_t{\bf h}^k)\big\|_{L^2_tH^{\frac{3}{2}}_x(\Gamma_T)}\le C\ \ \ \text{ for } k\ge1.
\end{align}
By Lemma \ref{lem15}, there is $\theta_{0}>0$ depending on
$T, \varepsilon_{1},\mathcal{E}_{0}, \displaystyle\|(\mathbf{h}, \partial_t{\bf h})\|_{L^2_tH^{\frac{3}{2}}_x(\Gamma_T)}$ 
and smooth solutions $(\mathbf{u}^{k},\mathbf{d}^{k})\in  C^{\infty}(\overline\Omega\times [0, T^k], \mathbb{R}^{2}\times\mathbb{S}^{2})$,
with $T^{k}\geq \theta_{0}r_{0}^{2}$, to the system \eqref{eq1.1} 
under the initial and boundary condition
\begin{align*}
(\mathbf{u}^{k}, \mathbf{d}^{k})=(\mathbf{u}_{0}^{k},\mathbf{d}_{0}^{k}) &\quad \text{ in }\Omega\times\{0\},\\
(\mathbf{u}^{k}, \mathbf{d}^{k})=(0, \mathbf{h}^{k}) &\quad \text{ on } \Gamma_{T^k}.
\end{align*}
Moreover, it holds that
\begin{align}\label{prioriestimatek}
\sup_{({x},t)\in \overline{\Omega}\times [0, T^k]} \int_{\Omega\cap B_{r_{0}}({x})} 
(|\mathbf{u}^{k}|^{2}+|\nabla\mathbf{d}^{k}|^{2})\leq \varepsilon_{1}^{2},
\end{align}
and for any $0<t\le T^k$, 
\begin{align}\label{energyk}
&\sup_{0\le \tau\le t}\int_{\Omega} (|\mathbf{u}^{k}|^{2}+|\nabla \mathbf{d}^{k}|^{2})(\tau)
+\int_{Q_{t}} (|\nabla \mathbf{u}^{k}|^{2}+|\Delta\mathbf{d}^{k}+|\nabla{\bf d}^k|^2{\bf d}^k|^{2})\nonumber\\
&\leq \psi_k(t)\big[\int_{\Omega} (|\mathbf{u}_{0}^{k}|^{2}+|\nabla \mathbf{d}_{0}^{k}|^{2})
+C\big\|(\mathbf{h}^{k}, \partial_t{\bf h}^k)\big\|_{L^2_tH^{\frac{3}{2}}_x(\Gamma_t)}^2\big]\\
&\le C\big(T, \mathcal{E}_0, \|(\mathbf{h}, \partial_t{\bf h})\big\|_{L^2_tH^{\frac{3}{2}}_x(\Gamma_T)}\big),\nonumber
\end{align}
where
$$\psi_k(t)=\exp\big(C\int_0^t \|\partial_t{\bf h}^k(\tau)\|_{H^\frac32(\Gamma)}\,d\tau\big)\le C<\infty,
\ \forall\ 0\le t\le T.
$$
Combining \eqref{prioriestimatek}, \eqref{energyk} together with Lemma \ref{lem15}, we conclude that
\begin{align}\label{L4forudk}
\int_{Q_{T_{0}^{k}}} (|\mathbf{u}^{k}|^{4}+|\nabla\mathbf{d}^{k}|^{4}) \leq C\varepsilon_{1}^{2}, 
\ \forall\ k\ge 1,
\end{align}
and
\begin{align}\label{Engk1}
&\|\partial_t{\bf d}^k\|_{L^2(Q_{T^k})}^2+\|\nabla \mathbf{u}^{k}\|_{L^{2}(Q_{T^k})}^{2}\!+\!
\|\nabla^2\mathbf{d}^{k}\|_{L^{2}(Q_{T^k})}^{2}\le C, \ \forall\ k\ge 1.
\end{align}
It follows from Lemma \ref{lem13}, \eqref{prioriestimatek}, \eqref{energyk}, \eqref{L4forudk}, and \eqref{Engk1}
that
\begin{align}\label{p-bound1}
\|\nabla P^{k}\|_{L^{\frac{4}{3}}(Q_{T^k})}
\leq  C\varepsilon_{1}^{\frac{1}{2}}, \ \forall\ k\ge 1.
\end{align}
Furthermore, $\eqref{eq1.1}_{1}$,  \eqref{L4forudk}, \eqref{Engk1},
and \eqref{p-bound1} imply that 
\begin{align}\label{dkt-bound}
   \|\partial_{t}\mathbf{u}^{k}\|_{L^\frac43_tH^{-1}_x(Q_T)}\le C, \ \forall\  k\ge 1.
\end{align}
By Theorem \ref{thm6}, we conclude that for any $\alpha\in (0,1)$
such that for any $\delta>0$,
\begin{align*}
\big\|(\mathbf{u}^{k},\mathbf{d}^{k})\big\|_{C^{\alpha,\frac{\alpha}{2}}
(\overline\Omega\times [\delta, T^k])}
\leq  C\big(\alpha, \delta,\mathcal{E}_{0},\varepsilon_{1},
\|(\mathbf{h}, \partial_t{\bf h})\|_{L^2_tH^\frac32_x(\Gamma_{T^k})}\big),
\end{align*}
for any compact subdomain $\omega\subset\subset\Omega$,
\begin{align*}
\big\|(\mathbf{u}^{k},\mathbf{d}^{k})\big\|_{C^{\ell}(\omega\times[\delta,T^{k}])}
\leq C( \operatorname{dist}(\omega,\partial\Omega), \delta, \ell,\mathcal{E}_{0})\quad \text{ for all }\ell\geq 1.
\end{align*}
There exist $T_{0}\geq \theta_{0}r_{0}^{2}$, 
 $\mathbf{u}\in L^\infty_tL^2_x\cap L^2_tH^1_x(Q_{T_{0}},\mathbb{R}^{2})$,
 $\mathbf{d}\in L^2_tH^2_x(Q_{T_{0}},\mathbb{S}^{2})$ such that after passing to a possible subsequence,
 $T^k\rightarrow T_0$, 
 \begin{align*}
 &\mathbf{u}^{k}\rightarrow \mathbf{u} \text{ weakly in } W^{1,0}_{2}(Q_{T_{0}},\mathbb{R}^{2})
 \ {\rm{and\ strongly \ in}}\ L^2(Q_{T_0}),\nonumber\\
 &\mathbf{d}^{k}\rightarrow \mathbf{d} \text{ weakly in } W^{2,1}_{2}(Q_{T_{0}},\mathbb{R}^{3})
 \ {\rm{and\ strongly \ in}}\ L^2_tH^1_x(Q_{T_0}),\nonumber\\
 &\lim_{k\uparrow\infty}
 \big(\|\mathbf{u}^{k}-\mathbf{u}\|_{L^{4}(Q_{T_{0}})}
 +\|\nabla\mathbf{d}^{k}-\nabla\mathbf{d}\|_{L^{4}(Q_{T_{0}})}\big)
 =0,
 \end{align*}
and for any $\ell\geq 2$, $\delta>0$, and compact $\omega\subset\subset\Omega$,
\begin{align*}
 &\lim_{k\uparrow\infty}
 \|(\mathbf{u}^{k},\mathbf{d}^{k})
 -(\mathbf{u},\mathbf{d})\|_{C^{\ell}(\omega\times[\delta,T_{0}])}
 =0,\nonumber\\
 &\lim_{k\uparrow\infty}
 \|(\mathbf{u}^{k},\mathbf{d}^{k})
 -(\mathbf{u},\mathbf{d})\|_{C^{\alpha,\frac{\alpha}{2}}
 (\overline{\Omega}\times[\delta,T_{0}])}
 =0.
\end{align*}
Thus $(\mathbf{u},\mathbf{d})\in C^{\infty}(\Omega\times(0,T_{0}], \mathbb{R}^{2}\times\mathbb{S}^{2})\cap
C^{\alpha,\frac{\alpha}{2}}(\overline{\Omega}\times(0,T_{0}],
\mathbb{R}^{2}\times\mathbb{S}^{2})$
solves the system \eqref{eq1.1}--\eqref{eq1.3} in $\Omega\times (0,T_{0}]$. 
From \eqref{energyk}, we can show that
\begin{align*}
(\mathbf{u},\nabla \mathbf{d})(\cdot,t)\rightarrow (\mathbf{u}_{0},\nabla\mathbf{d}_{0})
\text{ in } L^{2}(\Omega)\ {\rm{as}}\ t\downarrow 0.
\end{align*}
Hence $(\mathbf{u},\mathbf{d})$ satisfies the initial and boundary condition
\eqref{eq1.2} and \eqref{eq1.3}.
Let $T_{1}\in (0,T)$ be the first singular time of $(\mathbf{u},\mathbf{d})$, that is
\begin{align*}
(\mathbf{u},\mathbf{d})\in C^{\infty}(\Omega\times(0,T_{1}), \mathbb{R}^{2}\times\mathbb{S}^{2})\bigcap
C^{\alpha, \frac{\alpha}{2}}(\overline{\Omega}\times(0,T_{1}),
\mathbb{R}^{2}\times\mathbb{S}^{2}),
\end{align*}
but
\begin{align*}
(\mathbf{u},\mathbf{d})\notin C^{\infty}(\Omega\times(0,T_{1}],\mathbb{R}^{2}\times\mathbb{S}^{2})\bigcap
C^{\alpha,\frac{\alpha}{2}}(\overline{\Omega}\times(0,T_{1}],
\mathbb{R}^{2}\times\mathbb{S}^{2}).
\end{align*}
Thus we must have
\begin{align} \label{singulartime}
\limsup_{t \uparrow T_{1}} \max_{{x}\in\overline{\Omega}}
\int_{\Omega\cap B_{r}({x})} (|\mathbf{u}|^{2}+|\nabla\mathbf{d}|^{2})(\cdot,t)\geq \varepsilon_{1}^{2}\
\  \text{ for all } r>0.
\end{align}
In what follows, we will look for an eternal extension of this weak solution beyond $T_{1}$. 
To do it, we need to define $({\bf u}, {\bf d})$ at time $T_{1}$, which follows from 
the claim that
\begin{align} \label{claim}
(\mathbf{u},\mathbf{d})\in C([0,T_{1}], L^{2}(\Omega, \mathbb{R}^{2}\times\mathbb{S}^{2})).
\end{align}
In fact, for any $\phi\in H_{0}^{2}(\Omega, \mathbb{R}^{3})$, we can derive from $\eqref{eq1.1}_{3}$ that
\begin{align*}
&|\langle\partial_{t}\mathbf{d},\phi\rangle|=
\big|\int_{\Omega} (\langle\nabla\mathbf{d}, \nabla\phi
\rangle+\langle\mathbf{u}\cdot\nabla \mathbf{d},\phi\rangle-|\nabla\mathbf{d}|^{2}\langle\mathbf{d}, \phi\rangle)\big|
\nonumber\\
&\leq C
\|\nabla\mathbf{d}\|_{L^{2}(\Omega)} \|\nabla\phi\|_{L^{2}(\Omega)}
+(\|\mathbf{u}\|_{L^{2}(\Omega)}+ \|\nabla\mathbf{d}\|_{L^{2}(\Omega)})
\|\nabla\mathbf{d}\|_{L^{2}(\Omega)}\|\phi\|_{L^{\infty}(\Omega)}
\nonumber\\
&\leq C
\big[\|\nabla\mathbf{d}\|_{L^{2}(\Omega)}
+(\|\mathbf{u}\|_{L^{2}(\Omega)}+ \|\nabla\mathbf{d}\|_{L^{2}(\Omega)})
\|\nabla\mathbf{d}\|_{L^{2}(\Omega)}\big]\|\phi\|_{H^{2}(\Omega)},
\end{align*}
so that
$\partial_{t}\mathbf{d}\in L^{2}([0,T_{1}], H^{-2}(\Omega, \mathbb{R}^{3}))$. This and
$\mathbf{d}\in L^{2}_t H^{1}_x(Q_T)$ imply that
$\mathbf{d}\in C([0,T_{1}],  L^{2}(\Omega,\mathbb{S}^{2}))$.

For any $\phi\in H_{0}^{3}(\Omega,\mathbb{R}^{2})$, with $\nabla\cdot\phi=0$, 
$\eqref{eq1.1}_{1}$ implies that
\begin{align*}
&|\langle\partial_{t}\mathbf{u},\phi\rangle|=
\left|\int_{\Omega} (\nabla\mathbf{u}\cdot\nabla\phi+\mathbf{u}\cdot\nabla \mathbf{u}\cdot\phi-\nabla\mathbf{d}\odot\nabla\mathbf{d}:\nabla\phi)\right|
\nonumber\\
&\leq C
\|\nabla\mathbf{u}\|_{L^{2}(\Omega)} \|\nabla\phi\|_{L^{2}(\Omega)}\\
&\ \ +C(\|\mathbf{u}\|_{L^{2}(\Omega)}\|\nabla\mathbf{u}\|_{L^{2}(\Omega)}+ \|\nabla\mathbf{d}\|_{L^{2}(\Omega)}^{2})
\|\nabla\phi\|_{L^{\infty}(\Omega)}
\nonumber\\
&\leq C
(\|\nabla\mathbf{u}\|_{L^{2}(\Omega)}
+\|\mathbf{u}\|_{L^{2}(\Omega)}\|\nabla\mathbf{u}\|_{L^{2}(\Omega)}+ \|\nabla\mathbf{d}\|_{L^{2}(\Omega)}^{2})\|\phi\|_{H^{3}(\Omega)},
\end{align*}
so that $\partial_{t}\mathbf{u}\in L^{2}([0,t_{1}], H^{-3}(\Omega,\mathbb{R}^{2}))$.
This and $\mathbf{u}\in L^{2}_tH^{1}_x(Q_T)$ imply
$\mathbf{u}\in C([0,T_{1}], L^{2}(\Omega))$. Thus \eqref{claim} follows.

It follows from \eqref{claim} that
\begin{align*}
(\mathbf{u},\mathbf{d})(\cdot,T_{1})=\lim_{t\uparrow T_{1}} (\mathbf{u},\mathbf{d})(\cdot,t) \text{ in } L^{2}(\Omega).
\end{align*}
This and \eqref{energyk} imply that
\begin{align*}
\nabla\mathbf{d}(\cdot,t)\rightarrow \nabla \mathbf{d}(\cdot,T_{1}) \text{ weakly in } L^{2}(\Omega) \ {\rm{as}}\ t\uparrow T_1.
\end{align*}
Thus $\mathbf{u}(\cdot,T_{1})\in \mathbf{H} \text{ and } \mathbf{d}(\cdot,T_{1})\in H^{1}(\Omega, \mathbb{S}^{2}).$
Since $H^1(\Omega)\subset L^2(\Gamma)$ is compact, we also have that
$${\bf d}(\cdot, t)(={\bf h}(\cdot, t))\rightarrow {\bf d}(\cdot, T_1) \ {\rm{ in }}\ L^2(\Gamma) \ {\rm{as}}\ t\uparrow T_1.$$
This and ${\bf h}\in C([0,T], H^\frac32(\Gamma))$ imply that
$\mathbf{d}(\cdot,T_{1})=\mathbf{h}(\cdot, T_{1})$ on $\Gamma$.

Now, we can use $(\mathbf{u},\mathbf{d})(\cdot,T_{1})$ and $({0},\mathbf{h})$ as the initial and  boundary value
to extend the weak solution of \eqref{eq1.1}--\eqref{eq1.3} to the time interval $[0, T_2]$ for some
$T_2>T_1$. Repeating this procedure, we eventually obtain the existence of global weak solution in the time interval
$[0, T)$. Next we want to show 
\medskip

\noindent{\it Claim 1.  There are at most finitely singular times}.
To show it, first observe that at any singular time $T_{\sharp}\in (0, T)$, there is at least 
a loss of energy of $\frac12\varepsilon_{1}^{2}$. It follows from \eqref{singulartime}
that for any $r>0$, there exist $t_{i}\uparrow T_{\sharp}$ and $x_i\in \overline{\Omega}$ such that
$x_i\rightarrow x_0\in\overline\Omega$, and
\begin{align*}
\int_{\Omega\cap B_{r}(x_i)} (|\mathbf{u}|^{2}+|\nabla\mathbf{d}|^{2})(t_{i})\geq 
\frac12\varepsilon_{1}^{2},
\end{align*}
and hence
\begin{align} \label{W1}
&\int_{\Omega}(|\mathbf{u}|^{2}+|\nabla\mathbf{d}|^{2})(T_{\sharp})\nonumber\\
&= \lim_{r\downarrow 0}\int_{\Omega\setminus B_{2r}({x}_0)}
(|\mathbf{u}|^{2}+|\nabla\mathbf{d}|^{2})(T_{\sharp})
\nonumber\\
&\leq \lim_{r\downarrow 0}\liminf_{t_{i}\uparrow T_{\sharp}}\int_{\Omega\setminus B_{2r}({x}_{0})}
(|\mathbf{u}|^{2}+|\nabla\mathbf{d}|^{2})(t_{i})
\nonumber\\
&\leq \lim_{r\downarrow 0}\big[\liminf_{t_{i}\uparrow T_{\sharp}}\int_{\Omega}
(|\mathbf{u}|^{2}+|\nabla\mathbf{d}|^{2})(t_{i})\nonumber\\
&\ \ -\limsup_{t_{i}\uparrow T_{\sharp}} \int_{\Omega\cap B_{2r}({x_{0}})}
(|\mathbf{u}|^{2}+|\nabla\mathbf{d}|^{2})(t_{i})\big]
\nonumber\\
&\leq \liminf_{t_{i}\uparrow T_{\sharp}}\int_{\Omega}
(|\mathbf{u}|^{2}+|\nabla\mathbf{d}|^{2})(t_{i})-\frac12\varepsilon_{1}^{2},
\end{align}
\smallbreak

We will prove \textit{Claim} 1 by contradiction. Suppose that there were infinitely many singular times 
$\{T_{j}\}_{j=1}^\infty\subset (0,T]$, with  $0<T_{1}<T_{2}<\cdots<T_{j}<\cdots$, and 
$\displaystyle\lim_{j\uparrow +\infty}T_{j}=T_{*}\leq T$. Hence for any $\delta>0$,
there exists a sufficiently large
$j_0=j_0(\delta)\ge 1$ such that for $j\ge j_0$, we have
$$\exp\big(C\int_{T_j}^{T_{j+1}} \|\partial_t{\bf h}(\tau)\|_{H^\frac32(\Gamma)}\,d\tau\big)
\le 1+\delta,$$
and
$$
\big\|({\bf h}, \partial_t{\bf h})\big\|_{L^2_tH^\frac32_x(\Gamma\times[T_j, T_{j+1}])}^2\le\delta.$$
Then by \eqref{energy1} we have that for any $t\in [T_{j},T_{j+1})$
\begin{align}\label{W2}
&\mathcal{E}(t)=\int_{\Omega}(|\mathbf{u}|^{2}+|\nabla\mathbf{d}|^{2})(t)\nonumber\\
&\le\exp\big(C\int_{T_j}^{T_{j+1}}\|\partial_t{\bf h}(\tau)\|_{H^\frac32(\Gamma)}\,d\tau\big)
\big[\int_{\Omega}(|\mathbf{u}|^{2}+|\nabla{\mathbf{d}}|^{2})(T_j)\nonumber\\
&\ \ \ \qquad\qquad\qquad\qquad
+C\big\|({\bf h}, \partial_t{\bf h})\big\|_{L^2_tH^\frac32_x(\Gamma\times[T_j, T_{j+1}])}^2\big]\nonumber\\
&\leq (1+\delta)[\int_{\Omega}(|\mathbf{u}|^{2}+|\nabla{\mathbf{d}}|^{2})(T_j)+\delta]\nonumber\\
&\le \int_{\Omega}(|\mathbf{u}|^{2}+|\nabla{\mathbf{d}}|^{2})(T_j)+C\delta.
\end{align}
Putting \eqref{W1} and \eqref{W2} together, we obtain
\begin{align}\label{W3}
&\int_{\Omega}(|\mathbf{u}|^{2}+|\nabla{\mathbf{d}}|^{2})(T_{j+1})\nonumber\\
&\le \int_{\Omega}(|\mathbf{u}|^{2}+|\nabla{\mathbf{d}}|^{2})(T_j)+C\delta-\frac12\varepsilon_1^2\nonumber\\
&\le \int_{\Omega}(|\mathbf{u}|^{2}+|\nabla{\mathbf{d}}|^{2})(T_j)-\frac14\varepsilon_1^2,
\end{align}
provided $\displaystyle\delta\le\frac{\varepsilon_1^2}{4C}$.

Iterating the above inequality $m$ times, we obtain that
\begin{align*}
0\leq \mathcal{E}(T_{j_{0}+m})\leq \mathcal{E}(T_{j_{0}})-\frac{m \varepsilon_{1}^{2}}{4}.
\end{align*}
This yields that
$$m\leq \big[\frac{4K_{0}}{\varepsilon_{1}^{2}}\big] +1,$$
where $K_0=\mathcal{E}(T_{j_{0}})$. This proves {\it Claim} 1.

If $T_{L}<T$ is the last singular time, then we can use $(\mathbf{u}(T_{L}), \mathbf{d}(T_{L}))$ and $({0},\mathbf{h})|_{\Gamma\times (T_{L},T]}$ as the initial and boundary data to construct a weak solution $(\mathbf{u},\mathbf{d})$ to system \eqref{eq1.1}--\eqref{eq1.3} on $[T_{L},T]$ as before so that we obtain a global weak solution
$(\mathbf{u}, \mathbf{d})$ to \eqref{eq1.1}--\eqref{eq1.3} in the time interval $[0, T)$. 
This completes the proof of Theorem \ref{thm1}.
 $\hfill\Box$


\subsection{Proof of Theorem \ref{globalW}}The proof of
Theorem \ref{globalW} is similar to \cite{LLW} Theorem 1.3. For the convenience
of reader, we sketch it here. Let $(\mathbf{u}_{0},\mathbf{d}_{0})$ and $\mathbf{h}$ satisfy
the assumptions of Theorem \ref{globalW}. By Lemma \ref{rem-S+},  the weak solution 
$(\mathbf{u},\mathbf{d})$ to \eqref{eq1.1}---\eqref{eq1.3}, obtained by Theorem \ref{globalW},  satisfies
\begin{align*}
\mathbf{d}(x,t)\in \mathbb{S}_{+}^{2}, \ {\rm{for\ a.e.}}\ (x,t)\in Q_T.
\end{align*}
Assume that  $(\mathbf{u},\mathbf{d})$  has a singular time $T_1\in (0,T)$. Then, it follows from
 \eqref{blowup} that for $\mathcal{C}>1$, to be determined later, there exist $t_{m}\uparrow T_{1}^-$ and 
 $r_{m}\downarrow 0^+$   such that
 \begin{align} \label{Ass1}
 \frac{\varepsilon_{1}^{2}}{\mathcal{C}}=\sup_{{x}\in\overline{\Omega},0\leq t\leq t_{m}}
 \int_{\Omega\cap B_{r_{m}}({x})} (|\mathbf{u}|^{2}+|\nabla \mathbf{d}|^{2}).
 \end{align}
 It follows from the proof of Lemma \ref{lem15},
 there exist $\theta_{0}$, depending only on $\varepsilon_{1}$, $\mathcal{E}_{0}$, 
 and $\displaystyle\|(\mathbf{h},\partial_t{\bf h})\|_{L^2_tH^\frac32_x(\Gamma_{T})}$, and
 $x_m\in\overline\Omega$, such that
 \begin{align}
 &\int_{\Omega\cap B_{2r_{m}}(x_m)}(|\mathbf{u}|^{2}+|\nabla \mathbf{d}|^{2})(t_{m}-\theta_{0}r_{m}^{2})\nonumber\\
 &\ge \frac12\sup_{x\in\overline\Omega}\int_{\Omega\cap B_{2r_{m}}(x)}(|\mathbf{u}|^{2}+|\nabla \mathbf{d}|^{2})
 (t_{m}-\theta_{0}r_{m}^{2})\nonumber\\
 &\geq \frac{\varepsilon_{1}^{2}}{4\mathcal{C}}.
 \end{align}
 By \eqref{energy1} in Lemma \ref{lem10}, \eqref{Ass1} and the Ladyzhenskaya inequality, 
 we have
 \begin{align}
 \begin{cases}\displaystyle
 \int_{Q_{t_m}} (|\nabla\mathbf{u}|^{2}+|\nabla^2\mathbf{d}|^{2})\leq
 C\big(\varepsilon_{1}, \mathcal{E}_{0},\|(\mathbf{h}, \partial_t{\bf h})\|_{L^2_tH^\frac32_x(\Gamma_{T})}\big),\\
 \displaystyle\int_{Q_{t_m}} (|\mathbf{u}|^{4}+|\nabla\mathbf{d}|^{4})\leq \frac{C\varepsilon_1^2}{\mathcal{C}}.
 \end{cases}
 \end{align}
 Set $\Omega_{m}=r_{m}^{-1}(\Omega\backslash\{{x}_{m}\})$ and define
$(\mathbf{u}_{m},\mathbf{d}_{m}):\Omega_m\times[-\frac{t_{m}}{r_{m}^{2}}, 0]
 \mapsto \mathbb{R}^{2}\times \mathbb{S}_{+}^{2}$ by
 \begin{align*} 
 (\mathbf{u}_{m}, {\bf d}_m)({x},t)=(r_{m} \mathbf{u}({x}_{m}+r_{m}{x}, t_{m}+r_{m}^{2} t), 
  \mathbf{d}({x}_{m}+r_{m}{x}, t_{m}+r_{m}^{2} t)).
 \end{align*}
 Then $(\mathbf{u}_{m},\mathbf{d}_{m})$ solves \eqref{eq1.1}--\eqref{eq1.3} 
 in $\Omega_{m}\times[-\frac{t_{m}}{r_{m}^{2}},0]$, along with
 $$(\mathbf{u}_{m},\mathbf{d}_{m})(x, -\frac{t_{m}}{r_{m}^{2}})
 =(r_{m} \mathbf{u} ({x}_{m}+r_{m}{x},0),
 \mathbf{d}({x}_{m}+r_{m}{x},{0}))$$ 
 and
 $$(\mathbf{u}_{m},\mathbf{d}_{m})(x,t)
 =({0}, \mathbf{h}({x}_{m}+r_{m}{x},t_{m}+r_{m}^{2}t))
 \ {\rm{on}}\ \partial\Omega_m\times [-\frac{t_m}{r_m^2}, 0].$$
 Moreover,
 \begin{align} \label{umdm1}
 &\int_{\Omega_{m}\cap B_{2}(0)} (|\mathbf{u}_{m}|^{2}+|\nabla \mathbf{d}_{m}|^{2})
 (-\theta_{0})\geq \frac{\varepsilon_{1}^{2}}{4\mathcal{C}},\nonumber\\
 & \int_{\Omega_{m}\cap B_{1}({x})}(|\mathbf{u}_{m}|^{2}+|\nabla \mathbf{d}_{m}|^{2})
 (t)\leq \frac{\varepsilon_{1}^{2}}{\mathcal{C}}, \forall {x}\in\Omega_{m},
 -\frac{t_{m}}{r_{m}^{2}}\leq t\leq 0, \nonumber\\
 &\int_{\Omega_{m}\times [-\frac{t_{m}}{r_{m}^{2}},0]}
 (|\mathbf{u}_{m}|^{4}+|\nabla \mathbf{d}_{m}|^{4})\leq \frac{C\varepsilon_1^2}{\mathcal{C}},\nonumber\\
& \int_{\Omega_{m}\times [-\frac{t_{m}}{r_{m}^{2}},0]} (|\nabla\mathbf{u}_m|^{2}+|\nabla^2\mathbf{d}_m|^{2})\leq
 C\big(\varepsilon_{1}, \mathcal{E}_{0},\|(\mathbf{h}, \partial_t{\bf h})\|_{L^2_tH^\frac32_x(\Gamma_{T})}\big).
 \end{align}
Assume ${x}_{m}\rightarrow {x}_{0}\in\overline{\Omega}$ and $\mathcal{C}>0$ is chosen sufficiently large.
We divide the discussion into two cases:

\medskip
\noindent {\textit{Case}} 1: ${x}_{0}\in \Omega$. Then $r_{m}<\operatorname{dist}({x}_{0}, \Gamma)$ and $\Omega_{m}\rightarrow \mathbb{R}^{2}$, 
$-\frac{t_{m}}{r_{m}^{2}}\rightarrow -\infty$.
By Theorem \ref{thm3},  there exists a smooth solution $(\mathbf{u}_{\infty},\mathbf{d}_{\infty}):\mathbb R^2\times(-\infty,0]\mapsto \mathbb{R}^{2}\times \mathbb{S}_{+}^{2}$ to the system \eqref{eq1.1}--\eqref{eq1.3} such that
\begin{align*}
(\mathbf{u}_{m},\mathbf{d}_{m})\rightarrow (\mathbf{u}_{\infty},\mathbf{d}_{\infty}) \quad
\text{ in }C_{loc}^{2}(\mathbb{R}^{2}\times(-\infty,0],\mathbb{R}^{2}\times\mathbb{S}_{+}^{2}).
\end{align*}
For any set $P_R=B_{R}\times [-R^2,0]\subset \mathbb{R}^{2}\times(-\infty,0]$,  it is easy to see that 
\begin{align*}
\int_{P_{R}} |\mathbf{u}_{\infty}|^{4}=\lim_{m\uparrow\infty} \int_{P_{R}}|\mathbf{u}_{m}|^{4}
=\lim_{m\uparrow\infty} 
\int_{B_{R r_{m}}({x}_{m})\times [t_m-R^2r_m^2, t_m]} |\mathbf{u}|^{4} =0.
\end{align*}
Thus $\mathbf{u}_{\infty}\equiv {0}$. 

It is also easy to see that  for any compact $\omega\subset \mathbb{R}^{2}$,
\begin{align*}
&\int_{-1}^{0}\int_{\omega} |\Delta \mathbf{d}_{\infty}+|\nabla \mathbf{d}_{\infty}|^{2}\mathbf{d}_{\infty}|^{2}\\
&\leq
\liminf_{m\uparrow\infty} \int_{-1}^{0}\int_{\Omega_{m}}
|\Delta\mathbf{d}_{m}+|\nabla \mathbf{d}_{m}|^{2}\mathbf{d}_{m}|^{2}
\nonumber\\
&\leq \lim_{m\uparrow \infty} \int_{t_{m}-r_{m}^{2}}^{t_{m}}
\int_{\Omega} |\Delta \mathbf{d}+|\nabla\mathbf{d}|^{2}\mathbf{d}|^{2}
=0,
\end{align*}
which,  together with $\eqref{eq1.1}_{3}$,  implies that
\begin{align*}
\partial_{t}\mathbf{d}_{\infty}+\mathbf{u}_{\infty}\cdot\nabla\mathbf{d}_{\infty}=\mathbf{0}
\text{ in } \mathbb{R}^{2}\times[-1,0].
\end{align*}
Hence $\partial_t{\bf d}_\infty\equiv 0$ and ${\bf d}_\infty:\mathbb R^2\mapsto\mathbb S^2_+$ 
is a nontrivial smooth harmonic map with finite energy according to \eqref{umdm1}, 
which contradicts to Lemma \ref{vanish}.
 
\medskip

\noindent{\textit{Case}} 2.  ${x}_{0}\in \Gamma$. Then  we have either\\
(a) $\displaystyle\lim_{m\uparrow \infty} \frac{|{x}_{m}-{x}_{0}|}{r_{m}} =\infty$. Then, as in {\it Case} 1, 
$\Omega_m\rightarrow\mathbb R^2$ and $({\bf u}_m, {\bf d}_m)$ converges to $(0,{\bf d}_\infty)$
in $C^2_{\rm{loc}}(\mathbb R^2\times [-1,0])$, where ${\bf d}_\infty:\mathbb R^2\mapsto\mathbb S^2_+$
is a nontrivial smooth harmonic map with finite energy, which contradicts to Lemma \ref{vanish}.\\
or \\
(b) $\displaystyle\lim_{m\uparrow \infty} \frac{|{x}_{m}-{x}_{0}|}{r_{m}} =a\in [0, \infty)$.
Then we would have
$$(\Omega_m,\ \partial\Omega_m)
\rightarrow \big(\mathbb R^2_{-a}=\big\{x=(x_1, x_2)\in\mathbb R^2: \ x_2\ge -a\big\},
\ \partial\mathbb R^2_{-a}\big).$$ 
Observe that ${\bf h}_m(x,t)={\bf h}(x_m+r_mx, t_m+r^2_m t)$
is uniformly bounded in $C^{\alpha,\frac{\alpha}2}(\partial\Omega_m\times [-1,0])$
for any $\alpha\in (0,1)$. Hence we may assume that 
${\bf h}_m\rightarrow {\bf p}$, in $C^0_{\rm{loc}}(\mathbb R^2_{-a}\times [-1,0])$,
for some point ${\bf p}\in \mathbb S^2$. Thus, similar to 
\cite{LLW} Theorem 1.3, we would obtain a nontrivial harmonic map 
${\bf d}_\infty: \mathbb R^2_{-a}\to\mathbb S^2_+$ with ${\bf d}_\infty={\bf p}$ on
$\partial\mathbb R^2_{-a}$, that has finite energy. This is again impossible. 

From {\it Case} 1 and {\it Case} 2, we conclude
that \eqref{Ass1} never occurs in $[0,T]$. This completes the proof of Theorem \ref{globalW}.
$\hfill\Box$



\section{Global strong solution}
In this section, we will show the existence of unique, global strong solutions to
the system \eqref{eq1.1}--\eqref{eq1.3}. For this purpose, we will assume
that the initial data 
$$(\mathbf{u}_{0},\mathbf{d}_{0})\in\mathbf{V}\times {H}^{2}(\Omega,\mathbb{S}_{+}^{2}),$$ 
and the boundary data
\begin{align}\label{assumptionforh2}
\mathbf{h}\in H^{\frac{5}{2}, \frac{5}{4}}(\Gamma_{T},\mathbb{S}_{+}^{2})
 \text{ and  }
\partial_{t}\mathbf{h}\in L^{2}_t H^{\frac{3}{2}}_x(\Gamma_T).
\end{align}
More precisely, we will prove 
\begin{theorem}\label{gsolu}
 Let   $(\mathbf{u}_0, \mathbf{d}_0)\in \mathbf{V}\times
 H^2(\Omega,\mathbb{S}^{2}_+)$, and $\mathbf{h}:\Gamma_T\mapsto\mathbb S^2_+$ 
 satisfy \eqref{assumptionforh2} and the compatibility condition \eqref{comp_cond}.
 Let $(\mathbf{u},\mathbf{d}): \Omega\times[0,T]\mapsto \mathbb{R}^{2}\times\mathbb{S}_{+}^{2}$ be a weak solution to the system \eqref{eq1.1}--\eqref{eq1.3}, with initial value $(\mathbf{u}_{0},\mathbf{d}_{0})$ and boundary value
 $({0},\mathbf{h})$, obtained by  Theorem \ref{globalW}. Then $(\mathbf{u},\mathbf{d})$ is a  unique global strong solution to the system \eqref{eq1.1}--\eqref{eq1.3},  that satisfies
 \begin{align*}
&\mathbf{u}\in L^{\infty}([0, T], \mathbf{V})\cap L^{2}_tH^{2}_x(Q_T,\mathbb{R}^{2}),\nonumber\\
&\mathbf{d}\in L^{\infty}_tH^{2}_x(Q_T,\mathbb{S}_{+}^{2})
\cap L^{2}_tH^{3}_x(Q_T,\mathbb{S}_{+}^{2}).
\end{align*}
Moreover, the following estimate holds
 \begin{align}   \label{globstrong}
  &\|\mathbf{u}(t)\|_{L^\infty_tH^{1}_x(Q_T)}^2
   + \|\mathbf{d}(t)\|_{L^\infty_tH^{2}_x(Q_T)}^2\nonumber\\
  &+\int_{0}^{T}\big(\|\mathbf{u}(\tau)\|_{H^{2}(\Omega)}^{2}
 +\|\mathbf{d}(\tau)\|_{H^{3}(\Omega)}^{2}\big)
 \leq C_T,
 \end{align}
where $C_T>0$ depends on $\|(\mathbf{u}_0,\nabla{\bf d}_0)\|_{H^{1}(\Omega)}$,
$\|\mathbf{h}\|_{H^{\frac{5}{2},\frac{5}{4}}(\Gamma_{T})}$, $\|\partial_t{\bf h}\|_{L^2_tH^\frac32_x(\Gamma_T)}$,
$\varepsilon_1$, $T$, and $\Omega$.
 \end{theorem}

\begin{remark}\label{gsoluR} {\rm
Employing $\eqref{eq1.1}_{1}$ and $\eqref{eq1.1}_{3}$ and the estimate \eqref{globstrong}, 
we can verify that the global strong solution $(\mathbf{u},\mathbf{d})$ obtained by Theorem \ref{gsolu} 
satisfies
\begin{align*}
\partial_{t}\mathbf{u} \in L^{2}([0,T],\mathbf{H})
\ \ {\rm{and}} \ \partial_t\mathbf{d} \in L^2_tH^{1}_x(Q_T),
\end{align*}
which,  combined with the Aubin-Lions Lemma,  implies that
\begin{align*}
\mathbf{u} \in C([0,T], \mathbf{V})
  \ \text{ and } \ \mathbf{d} \in C([0,T], H^{2}(\Omega)).
\end{align*}}
\end{remark}

\noindent\textbf{Proof of Theorem \ref{gsolu}. }
Since the uniqueness part of strong solutions follows immediately  from the continuous dependence Theorem \ref{gsoluC} below, 
we will focus on the proof of the existence of a global strong solution
$({\bf u}, {\bf d})$ that satisfies the estimate \eqref{globstrong}.

For $(\mathbf{u}_{0},\mathbf{d}_{0})\in \mathbf{V}\times H^{2}(\Omega,\mathbb{S}_{+}^{2})$ and  $\mathbf{h}$ satisfies \eqref{assumptionforh2}, 
recall that Theorem \ref{globalW} implies that there exists a global weak solution $(\mathbf{u},\mathbf{d})$ to the system \eqref{eq1.1}--\eqref{eq1.3}, with initial condition $(\mathbf{u}_{0},\mathbf{d}_{0})$ and boundary
condition $({0},\mathbf{h})$, which satisfies
 \begin{align*}
\mathbf{u}\in L^{\infty}([0, T], \mathbf{H})\cap L^{2}([0, T], \mathbf{V}),
\ \ \ \mathbf{d}\in L^{\infty}_tH^{1}_x(Q_T)\cap L^2_tH^2_x(Q_T).
\end{align*}

In order to prove that this global weak solution $({\bf u}, {\bf d})$ is the desired strong solution satisfying \eqref{globstrong},
we need to show that the sequence of smooth solutions $({\bf u}^k, {\bf d}^k):Q_T\mapsto\mathbb R^2\times\mathbb S^2_+$ of the system \eqref{eq1.1},
under the initial and boundary conditions $({\bf u}_0^k, {\bf d}^k_0)$ and $(0, {\bf h}^k)$, from Theorem \ref{globalW} 
satisfy \eqref{globstrong} with a constant $C_{T}$ that is independent of $k$.

In fact,  it follows from the proof of Theorem \ref{globalW}  that \\
i) 
\begin{equation}\label{strong-conv}
\begin{cases} (\nabla {\bf u}^k, \nabla^2 {\bf d}^k)\rightarrow  (\nabla {\bf u}, \nabla^2 {\bf d}) \ {\rm{weakly\ in}}\ L^2(Q_T),\\
({\bf u}^k, {\bf d}^k) \rightarrow  ({\bf u}, {\bf d}) \ {\rm{in}}\ C([0, T], L^2(\Omega)\times H^1(\Omega)).
\end{cases}
\end{equation}

\smallskip\noindent
ii) there exists $\varepsilon_{1}$ and  $r_{0}>0$  such that 
\begin{align}\label{small-bdd}
\sup_{({x},t)\in \overline{Q}_{T}}
\int_{\Omega\cap B_{r_{0}}({x})} (|\mathbf{u}^k|^{2}+|\nabla \mathbf{d}^k|^{2})(\cdot,t)
\leq \varepsilon_{1}^2, \forall k\ge 1, 
\end{align}
and
\begin{align}\label{weakbdd}
\|(\mathbf{u}^k,\nabla\mathbf{d}^k)\|_{L^{\infty}_tL^2_x(Q_T)}
+\|(\nabla \mathbf{u}^k,\nabla^2 \mathbf{d}^k)\|_{L^{2}(Q_{T})}\leq K_{T,\varepsilon_{1}}, \forall k\ge 1,
\end{align}
where $K_{T,\varepsilon_{1}}>0$  depends on $\|(\mathbf{u}_{0},\mathbf{d}_{0})\|_{\mathbf{H}\times H^{1}(\Omega)},
\|(\mathbf{ h},\partial_t{\bf h})\|_{L^{2}_tH^{\frac{3}{2}}_x(\Gamma_T)}$,
$\varepsilon_{1}$, $T$, and $\Omega$.

Now we claim that
\begin{align}\label{norm-conv}
&\lim_{k\rightarrow\infty}\int_{Q_T}(|\nabla \mathbf{u}^{k}|^{2}+|\Delta\mathbf{d}^{k}+|\nabla{\bf d}^k|^2{\bf d}^k|^{2})\nonumber\\
&=\int_{Q_T} (|\nabla \mathbf{u}|^{2}+|\Delta\mathbf{d}+|\nabla{\bf d}|^2{\bf d}|^{2}).
\end{align}
To show \eqref{norm-conv}, first observe from the proof of Lemma \ref{lem10} 
that $({\bf u}^k,{\bf d}^k)$ satisfies the following energy equality:
\begin{align}\label{energyk1}
&\int_{\Omega} (|\mathbf{u}^{k}|^{2}+|\nabla \mathbf{d}^{k}|^{2})(T)
+2\int_{Q_T} (|\nabla \mathbf{u}^{k}|^{2}+|\Delta\mathbf{d}^{k}+|\nabla{\bf d}^k|^2{\bf d}^k|^{2})\nonumber\\
&=\int_{\Omega} (|\mathbf{u}^{k}_0|^{2}+|\nabla \mathbf{d}^{k}_0|^{2})
+2\int_{Q_T}\langle\Delta {\bf d}^k, \partial_t{\bf h}_E^k\rangle+2\int_{\Gamma_T}\langle\frac{\partial{\bf h}_E^k}{\partial\nu}, \ \partial_t{\bf h}^k\rangle,
\end{align}
where ${\bf h}^k_E(\cdot, t)$ is the harmonic extension of ${\bf h}^k(\cdot, t)$ .

Since ${\bf d}\in L^2_t H^2_x(Q_T)$, an argument similar to Lemma \ref{lem10} also yields
that $({\bf u},{\bf d})$ satisfies the same energy equality as \eqref{energyk1}:
\begin{align}\label{energy01}
&\int_{\Omega} (|\mathbf{u}|^{2}+|\nabla \mathbf{d}|^{2})(T)
+2\int_{Q_T} (|\nabla \mathbf{u}|^{2}+|\Delta\mathbf{d}+|\nabla{\bf d}|^2{\bf d}|^{2})\nonumber\\
&=\int_{\Omega} (|\mathbf{u}_0|^{2}+|\nabla \mathbf{d}_0|^{2})+
2\int_{Q_T}\langle\Delta {\bf d}, \partial_t{\bf h}_E\rangle
+2\int_{\Gamma_T}\langle\frac{\partial{\bf h}_E}{\partial\nu}, \ \partial_t{\bf h}\rangle,
\end{align}
Since $({\bf h}^k, \partial_t{\bf h}^k)\rightarrow ({\bf h}, \partial_t{\bf h})$ in $L^2_tH^\frac32_x(\Gamma_T)$, it follows the standard estimate on
harmonic functions that
$$\partial_t{\bf h}^k_E\rightarrow \partial_t{\bf h}_E \ {\rm{in}}\ L^2_tH^2_x(Q_T),$$
and
$$\frac{\partial {\bf h}^k_E}{\partial\nu}\rightarrow \frac{\partial {\bf h}_E}{\partial\nu} \ {\rm{in}}\  
L^2_tH^\frac12_x(\Gamma_T).$$
Therefore, after sending $k\rightarrow\infty$ in \eqref{energyk1} and comparing the resulting equality with \eqref{energy01}, we see
that \eqref{norm-conv} holds true.

On the other hand,  \eqref{strong-conv} and the Ladyzhenskaya's inequality imply that
\begin{equation}\label{strong-conv1}
({\bf u}^k, \nabla{\bf d}^k)\rightarrow ({\bf u}, \nabla{\bf d}) \ {\rm{in}}\  L^4(Q_T).
\end{equation}
Now it is easy to see from \eqref{norm-conv} and \eqref{strong-conv1} that
\begin{align}\label{norm-conv1}
\lim_{k\rightarrow\infty}\int_{Q_T}(|\nabla \mathbf{u}^{k}|^{2}+|\Delta\mathbf{d}^{k}|^{2})
=\int_{Q_T} (|\nabla \mathbf{u}|^{2}+|\Delta\mathbf{d}|^{2}).
\end{align}
In particular, we conclude that
\begin{equation}\label{norm-conv2}
({\bf u}^k, {\bf d}^k)\rightarrow ({\bf u}, {\bf d}) \ {\rm{in}}\ L^2_tH^1_x(Q_T)\times L^2_tH^2_x(Q_T).
\end{equation}
It is clear that \eqref{norm-conv2} yields the following uniform absolute continuity:
for any $\epsilon>0$, there exists $\delta=\delta(\epsilon)>0$ such that 
\begin{equation}\label{uni-abs-cont}
\|{\bf u}^k\|_{L^2([s_1,s_2], H^1(\Omega))}^2+\|{\bf d}^k\|_{L^2([s_1,s_2], H^2(\Omega))}^2\le\epsilon^2,
\end{equation}
provided $0\le s_1<s_2\le T$ satisfies $|s_2-s_1|\le\delta$.

Now we can show that $({\bf u}^k, {\bf d}^k)$ satisfies the estimate \eqref{globstrong} with a constant $C_{T}$,
that is independent of $k$, as follows. For simplicity, we drop the upper index $k$ and write 
$({\bf u}, {\bf d}, {P})$ for $({\bf u}^k, {\bf d}^k, {P}^k)$.

By employing the $W^{2,1}_{2}$-regularity theory on the non-stationary Stokes system, we have
that $\partial_{t}\mathbf{u},\nabla^{2}\mathbf{u},\nabla P\in L^{2}(Q_T)$, and for all $0<t\leq T$
\begin{align*}
&\|\partial_{t} \mathbf{u}\|_{L^{2}(Q_{t})}
+\|\nabla^{2} \mathbf{u}\|_{L^{2}(Q_{t})}\\
&\leq C\big[
\|\nabla \mathbf{u}_{0}\|_{L^{2}(\Omega)}
+\|\mathbf{u}\cdot\nabla\mathbf{u}\|_{L^{2}(Q_{t})}
+\|\nabla(\nabla\mathbf{d}\odot\nabla\mathbf{d})\|_{L^{2}(Q_{t})}\big]
\nonumber\\
&\leq C\big[ \|\nabla \mathbf{u}_{0}\|_{L^{2}(\Omega)}
 +\|\mathbf{u}\|_{L^{4}(Q_{t})}
  \|\nabla \mathbf{u}\|_{L^{4}(Q_{t})}
 +\|\nabla \mathbf{d}\|_{L^{4}(Q_{t})}
 \|\nabla^{2} \mathbf{d}\|_{L^{4}(Q_{t})}\big].
\end{align*}
On the other hand, it follows from the trace theorem 
$W^{2,1}_{2}(Q_T)\hookrightarrow H^{1}(\Omega\times\{t\})$, $\forall t\in[0,T]$,
that 
\begin{align*}
\|\nabla \mathbf{u}\|_{L^{\infty}_tL^{2}_x(Q_t)}\leq C
\|\mathbf{u}\|_{W^{2,1}_{2}(Q_{t})}\leq C\big[ \|\partial_{t} \mathbf{u}\|_{L^{2}(Q_t)}
+\|(\nabla{\bf u}, \nabla^{2} \mathbf{u})\|_{L^{2}(Q_{t})}\big].
\end{align*}
Putting these two estimates together, we obtain that
\begin{align}\label{u-est10}
&\|\nabla \mathbf{u}\|_{L^{\infty}_tL^{2}_x(Q_t)}+
\|\partial_{t} \mathbf{u}\|_{L^{2}(Q_{t})}
+\|\nabla^{2} \mathbf{u}\|_{L^{2}(Q_{t})}
\nonumber\\
&\leq C\big[\|\nabla \mathbf{u}_{0}\|_{L^{2}(\Omega)}
 +\|\mathbf{u}\|_{L^{4}(Q_t)}
  \|\nabla \mathbf{u}\|_{L^{4}(Q_t)}+\|\nabla \mathbf{d}\|_{L^{4}(Q_t)}
  \|\nabla^{2} \mathbf{d}\|_{L^{4}(Q_t)}\big].
\end{align}
To estimate $\mathbf{d}$,  let $\mathbf{h}_{P}$  be the parabolic lifting function of ${\bf h}$, i.e.,
\begin{align*}
\begin{cases}
\partial_t {\bf h}_P-\Delta {\bf h}_P=0  & \ {\rm{in}}\ Q_T,\\
{\bf h}_P={\bf h} & \ {\rm{on}}\ \Gamma_T,\\
{\bf h}_P={\bf d}_0  & \ {\rm{in}}\ \Omega\times\{0\}.
\end{cases}
\end{align*}
It follows from  \eqref{assumptionforh2} and the regularity theory of parabolic equations (cf. \cite{LSU}) that
\begin{align}\label{hp-est}
&\|{\bf h}_P\|_{L^\infty_tH^2_x(Q_T)}+\|\mathbf{h}_{P}\|_{H^{3,\frac{3}{2}}(Q_{T})}\nonumber\\
&\leq C\big[
\|\mathbf{h}\|_{H^{\frac{5}{2},\frac{5}{4}}(\Gamma_{T})}
 +\|\partial_t{\bf h}\|_{L^2_tH^\frac32_x(\Gamma_T)}+\|\mathbf{d}_{0}\|_{H^{2}(\Omega)}\big].
\end{align}
Set $\widetilde{\mathbf{d}}=\mathbf{d}-\mathbf{h}_{P}$. Then we have
\begin{align}\label{dtilde-eqn}
\begin{cases}
\partial_{t}\widetilde{\mathbf{d}}-\Delta\widetilde{\mathbf{d}}
=-\mathbf{u}\cdot\nabla\mathbf{d}+|\nabla \mathbf{d}|^{2}\mathbf{d} \quad & \text{ in }Q_T,\\
\widetilde{\mathbf{d}}={0}\quad & \text{ on } \partial_p Q_T.
\end{cases}
\end{align}
It follows from the regularity theory of parabolic equations, the trace theorem, and the estimate of $\mathbf{h}_{P}$ that 
for $0<t\leq T$, it holds 
\begin{align}\label{d-est10}
&\|\mathbf{d}\|_{L^{\infty}_tH^{2}_x(Q_t)} +\|\mathbf{d}\|_{H^{3,\frac{3}{2}}(Q_{t})}\nonumber\\
&\leq C\big[\|{\bf d}_0\|_{H^2(\Omega)}+\|\partial_t{\bf h}\|_{L^2_tH^\frac32_x(\Gamma_T)}+\|\mathbf{h}\|_{H^{\frac52,\frac{5}{4}}(\Gamma_{T})} \nonumber\\
&\ \ \ +
\|\mathbf{u}\cdot\nabla \mathbf{d}\|_{H^{1,\frac{1}{2}}(Q_{t})}+
\||\nabla \mathbf{d}|^{2} \mathbf{d}\|_{H^{1,\frac{1}{2}}(Q_{t})}\big].
\end{align}
From the interpolation inequality 
$W^{0,0}_{\frac{8}{3}}(Q_{t})\cap {W}^{0,1}_{\frac{8}{5}}(Q_{t}) \hookrightarrow W^{0,\frac{1}{2}}_{2}(Q_{t})$ (cf. \cite{Simon} Lemma 7), 
we can estimate the last two terms 
in the right hand side of \eqref{d-est10} by
\begin{align*}
&\|\mathbf{u}\cdot\nabla \mathbf{d}\|_{H^{1,\frac{1}{2}}(Q_{t})}\\
&\leq C\big[ \||\mathbf{u}||\nabla \mathbf{d}|\|_{L^{2}(Q_{t})}+\||\nabla \mathbf{u}||\nabla \mathbf{d}|\|_{L^{2}(Q_{t})}
+\|| \mathbf{u}||\nabla^{2} \mathbf{d}|\|_{L^{2}(Q_{t})}\nonumber\\
&\ \ \ +\|{\bf u}\cdot\nabla{\bf d}\|_{L^\frac83(Q_t)}^\frac12 \|{\bf u}\cdot\nabla{\bf d}\|_{W^{0,1}_{\frac85}(Q_t)}^\frac12\big]
\nonumber\\
&\leq C[(\|\mathbf{u}\|_{L^{4}(Q_{t})}+\|\nabla \mathbf{u}\|_{L^{4}(Q_{t})})\|\nabla \mathbf{d}\|_{L^{4}(Q_{t})}
+\|\mathbf{ u}\|_{L^{4}(Q_{t})}\|\nabla^{2} \mathbf{d}\|_{L^{4}(Q_t)}\nonumber\\
&\qquad+\|{\bf u}\|_{L^4(Q_t)}^\frac12\|\nabla {\bf d}\|_{L^8(Q_t)}^\frac12\nonumber\\
&\ \ \ \ \cdot(\|\partial_{t}\mathbf{u}\|_{L^{2}(Q_T)}^\frac12
\|\nabla \mathbf{d}\|_{L^{8}(Q_{t})}^\frac12
+\|\partial_{t}\nabla \mathbf{d}\|_{L^{2}(Q_{t})}^\frac12\|\mathbf{u}\|_{L^{8}(Q_{t})}^\frac12)],
\end{align*}
and
\begin{align*}
&\||\nabla \mathbf{d}|^{2} \mathbf{d}\|_{H^{1,\frac{1}{2}}(Q_{t})}\\
&\leq C\big[
\|\nabla \mathbf{d}\|_{L^{4}(Q_{t})}^{2}
+\||\nabla \mathbf{d}|^{3}\|_{L^2(Q_t)}+\||\nabla \mathbf{d}||\nabla^{2}\mathbf{d}|\|_{L^{2}(Q_{t})}\\
&\ \ \ \ \ \ +\| |\nabla{\bf d}|^2\mathbf{d}\|_{L^\frac83(Q_t)}^\frac12 \|\partial_t(|\nabla \mathbf{d}|^2{\bf d})\|_{L^\frac85(Q_t)}^\frac12\big]
\nonumber\\
&\leq C\big[
\|\nabla \mathbf{d}\|_{L^{4}(Q_{t})}^{2} +\|\nabla \mathbf{d}\|_{L^{6}(Q_{t})}^{3}
+\|\nabla \mathbf{d}\|_{L^{4}(Q_{t})}
\|\nabla^{2}\mathbf{d}\|_{L^{4}(Q_{t})}
\nonumber\\
&\qquad +\|\nabla{\bf d}\|_{L^4(Q_t)}^\frac12\|\nabla{\bf d}\|_{L^8(Q_t)}^\frac12 \nonumber\\
&\ \ \ \ \cdot(\|\partial_{t}\mathbf{d}\|_{L^{2}(Q_{t})}^\frac12
\|\nabla \mathbf{d}\|_{L^{16}(Q_{t})}
+\|\partial_{t}\nabla \mathbf{d}\|_{L^{2}(Q_{t})}^\frac12
\|\nabla \mathbf{d}\|_{L^{8}(Q_{t})}^\frac12)\big].
\end{align*}
Substituting these two estimates into \eqref{u-est10} and \eqref{d-est10}, we obtain that
\begin{align} \label{Energy2}
&\big(\|\mathbf{u} \|_{L^{\infty}_t H^{1}_x(Q_t)}^{2} +
\| \mathbf{d} \|_{L^{\infty}_t H^{2}_x(Q_t)}^{2}\big)
+\int_{0}^{t} (\| \mathbf{u}\|_{H^{2}(\Omega)}^{2} +\| \mathbf{d}\|_{H^{3}(\Omega)}^{2})\,dt
\nonumber\\
&\leq C\big[ \|\nabla \mathbf{u}_{0}\|_{L^{2}(\Omega)}^{2}
+ \|\mathbf{d}_{0}\|_{H^{2}(\Omega)}^{2}+\|{\bf h}\|_{H^{\frac52,\frac54}(\Gamma_T)}^2
+\|\partial_t{\bf h}\|_{L^2_tH^\frac32_x(\Gamma_T)}^2\big]\nonumber\\
&+C\big[\|\nabla \mathbf{d}\|_{L^{4}(Q_{t})}^{4} +\|\nabla \mathbf{d}\|_{L^{6}(Q_{t})}^{6}
+\|\partial_{t}\mathbf{u}\|_{L^{2}(Q_{t})}\|{\bf u}\|_{L^4(Q_t)}\|\nabla \mathbf{d}\|_{L^{8}(Q_{t})}^2\nonumber\\
&+(\|\mathbf{u}\|_{L^{4}(Q_{t})}^{2}+\|\nabla \mathbf{d}\|_{L^{4}(Q_{t})}^{2})
\cdot(\|\nabla \mathbf{u}\|_{L^{4}(Q_{t})}^{2}+\|\nabla^{2} \mathbf{d}\|_{L^{4}(Q_{t})}^{2})
\nonumber\\
&+\|{\bf u}\|_{L^4(Q_t)}^2\|\nabla{\bf d}\|_{L^4(Q_t)}^2+
\|\partial_{t}\nabla \mathbf{d}\|_{L^{2}(Q_{t})}\|\nabla{\bf d}\|_{L^8(Q_t)}\nonumber\\
&\cdot(
\|\mathbf{u}\|_{L^{4}(Q_{t})}\|\mathbf{u}\|_{L^{8}(Q_{t})}+\|\nabla \mathbf{d}\|_{L^{4}(Q_{t})}\|\nabla \mathbf{d}\|_{L^{8}(Q_{t})})\nonumber\\
&+ \|\partial_{t}\mathbf{d}\|_{L^{2}(Q_{t})}\|\nabla \mathbf{d}\|_{L^{4}(Q_{t})}\|\nabla \mathbf{d}\|_{L^{8}(Q_{t})}
\|\nabla \mathbf{d}\|_{L^{16}(Q_{t})}^2\big]
\end{align}
for any $0<t\leq T$, where $Q_{t} =\Omega\times [0,t]$.

To simplify the presentation, we set two auxiliary functions
$$\Phi(t)=\|\mathbf{u} \|_{L^{\infty}_t H^{1}_x(Q_t)}^{2} +
\| \mathbf{d} \|_{L^{\infty}_t H^{2}_x(Q_t)}^{2}+\int_{0}^{t} (\| \mathbf{u}\|_{H^{2}(\Omega)}^{2} +\| \mathbf{d}\|_{H^{3}(\Omega)}^{2})\,dt,$$
and
$$\eta(t)=\int_{0}^{t} (\| \mathbf{u}\|_{H^{1}(\Omega)}^{2} +\| \mathbf{d}\|_{H^{2}(\Omega)}^{2})\,dt$$
for $t\in [0,T]$.

From \eqref{weakbdd} and Sobolev's inequality, it is readily seen that 
\begin{equation}\label{ud-est1}
\begin{cases}
\|\nabla \mathbf{d}\|_{L^{4}(Q_{t})}^{4}\leq C\|\nabla \mathbf{d}\|_{L^{\infty}_{t}L^{2}_{x}(Q_{t})}^{2}
\|\mathbf{d}\|_{L_{t}^{2}H_{x}^{2}(Q_{t})}^{2}\leq C\eta(t),\\
\|\nabla^{2} \mathbf{d}\|_{L^{4}(Q_{t})}^{4}\leq C\| \mathbf{d}\|_{L^{\infty}_{t}H^{2}_{x}(Q_{t})}^{2}
\|\mathbf{d}\|_{L_{t}^{2}H_{x}^{3}(Q_{t})}^{2}\leq C\Phi^2(t),\\
\|\nabla \mathbf{d}\|_{L^{6}(Q_{t})}^{6}\leq  C\|\nabla \mathbf{d}\|_{L^{\infty}_{t}L^{2}_{x}(Q_{t})}^{3}
\| \mathbf{d}\|_{L^{2}_{t}H^{3}_{x}(Q_{t})}
\| \mathbf{d}\|_{L^{\infty}_{t}H^{2}_{x}(Q_{t})}
\| \mathbf{d}\|_{L^{2}_{t}H^{2}_{x}(Q_{t})}\\
\qquad\qquad\ \ \ \ \leq C\eta^{\frac{1}{2}}(t) \Phi(t),\\
\|\nabla \mathbf{d}\|_{L^{8}(Q_{t})}^{8}\leq  C\|\nabla \mathbf{d}\|_{L^{\infty}_{t}L^{2}_{x}(Q_{t})}^{3}
\| \mathbf{d}\|_{L^{2}_{t}H^{3}_{x}(Q_{t})}
\| \mathbf{d}\|_{L^{\infty}_{t}H^{2}_{x}(Q_{t})}^{3}
\| \mathbf{d}\|_{L^{2}_{t}H^{2}_{x}(Q_{t})}\\
\qquad\qquad\ \ \ \ \leq C\eta^{\frac{1}{2}}(t) \Phi^2(t),\\
\|\nabla \mathbf{d}\|_{L^{16}(Q_{t})}^{16}\leq  C\|\nabla \mathbf{d}\|_{L^{\infty}_{t}L^{2}_{x}(Q_{t})}^{3}
\| \mathbf{d}\|_{L^{2}_{t}H^{3}_{x}(Q_{t})}
\| \mathbf{d}\|_{L^{\infty}_{t}H^{2}_{x}(Q_{t})}^{11}
\| \mathbf{d}\|_{L^{2}_{t}H^{2}_{x}(Q_{t})}\\
\qquad\qquad\ \ \ \ \leq C\eta^{\frac{1}{2}}(t) \Phi^{6}(t),
\end{cases}
\end{equation}
where $C$ is the positive constant depending only on $\Omega$ and $K_{T,\varepsilon_{1}}$.
Similarly, we also have that
\begin{equation}\label{ud-est2}
\begin{cases}
\|\mathbf{u}\|_{L^{4}(Q_{t})}^{4}\leq C\| \mathbf{u}\|_{L^{\infty}_{t}L^{2}_{x}(Q_{t})}^{2}
\|\mathbf{u}\|_{L_{t}^{2}H_{x}^{1}(Q_{t})}^{2}\leq C\eta(t),\\
\|\nabla \mathbf{u}\|_{L^{4}(Q_{t})}^{4}\leq C\| \mathbf{u}\|_{L^{\infty}_{t}H_{x}^{1}(Q_{t})}^{2}
\|\mathbf{u}\|_{L_{t}^{2}H_{x}^{2}(Q_{t})}^{2}\leq C\Phi^2(t),\\
\| \mathbf{u}\|_{L^{8}(Q_{t})}^{8}\leq  C\| \mathbf{u}\|_{L^{\infty}_{t}L^{2}_{x}(Q_{t})}^{3}
\| \mathbf{u}\|_{L^{2}_{t}H^{2}_{x}(Q_{t})}
\| \mathbf{u}\|_{L^{\infty}_{t}H^{1}_{x}(Q_{t})}^{3}
\| \mathbf{u}\|_{L^{2}_{t}H^{1}_{x}(Q_{t})}\\
\qquad\qquad\  \leq C\eta^{\frac{1}{2}}(t) \Phi^{2}(t).
\end{cases}
\end{equation}
It follows from \eqref{u-est10}, \eqref{ud-est1}, and \eqref{ud-est2} that
\begin{align}\label{u-est15}
\|\partial_t{\bf u}\|_{L^2(Q_t)}\le C\big[\|\nabla{\bf u}_0\|_{L^2(\Omega)}+\eta^\frac14\Phi^\frac12(t)\big],
\end{align}
and it follows from \eqref{dtilde-eqn} and \eqref{hp-est} that
\begin{align}\label{d-est15}
\|\partial_t{\bf d}\|_{L^2(Q_t)}
&\le C\big[\|\partial_t{\bf h}\|_{L^2_tH^\frac32_x(\Gamma_T)}
+\|{\bf d}_0\|_{H^2(\Omega)}\nonumber\\
&\qquad+\|{\bf u}\|_{L^4(Q_t)}\|\nabla{\bf d}\|_{L^4(Q_t)}
+\|\nabla{\bf d}\|_{L^4(Q_t)}^2\big]\nonumber\\
&\le C\big[\|\partial_t{\bf h}\|_{L^2_tH^\frac32_x(\Gamma_T)}
+\|{\bf d}_0\|_{H^2(\Omega)}+\eta^\frac12(t)\big].
\end{align}
From \eqref{ud-est1}, \eqref{ud-est2}, and the equation
$$\partial_t\nabla{\bf d}=\nabla\Delta{\bf d}-\nabla({\bf u}\cdot\nabla{\bf d})+\nabla(|\nabla{\bf d}|^2{\bf d}),$$
we have that
\begin{align}\label{Dtd-est}
&\|\partial_{t}\nabla \mathbf{d}\|_{L^{2}(Q_T)}\nonumber\\
&\leq (\|\nabla\Delta \mathbf{d}\|_{L^{2}(Q_T)}
+\|\nabla(\mathbf{u}\cdot\nabla \mathbf{d})\|_{L^{2}(Q_T)}
+\|\nabla(|\nabla \mathbf{d}|^{2}\mathbf{d})\|_{L^{2}(Q_T)})
\nonumber\\
&\leq C\big[\Phi^\frac12(t)+\|\nabla{\bf u}\|_{L^4(Q_t)} \|\nabla{\bf d}\|_{L^4(Q_t)}
+\|{\bf u}\|_{L^4(Q_t)}\|\nabla^2{\bf d}\|_{L^4(Q_t)}\nonumber\\
&\qquad+\|\nabla{\bf d}\|_{L^6(Q_t)}^3+\|\nabla{\bf d}\|_{L^4(Q_t)}\|\nabla^2{\bf d}\|_{L^4(Q_t)}\big]\nonumber\\
&\leq C\big[\Phi^\frac12(t)+\eta^\frac14(t)\Phi^\frac12(t)\big].
\end{align}
Putting all these estimates into \eqref{Energy2}, we obtain that 
\begin{align}\label{estimate1}
&\Phi(t)\leq C\big[ \|\nabla \mathbf{u}_{0}\|_{L^{2}(\Omega)}^{2}+
\|{\bf d}_0\|_{H^2(\Omega)}^2+\|{\bf h}\|_{H^{\frac52,\frac54}(\Gamma_T)}^2
+\|\partial_t{\bf h}\|_{L^2_tH^\frac32_x(\Gamma_T)}^2\big]\nonumber\\
&\ \ \ \ \ \ \ +C\big[\eta(t)+(\eta^\frac38(t)+\eta^\frac48(t)+\eta^\frac58(t)+\eta^\frac78(t))\Phi(t)\big].
\end{align}
From the uniform absolute continuity property \eqref{uni-abs-cont}, there exists $0<t_0\le T$ such that
$$C\big[\eta^\frac38(t_0)+\eta^\frac48(t_0)+\eta^\frac58(t_0)+\eta^\frac78(t_0)\big]\le \frac12,$$
and hence we arrive at
\begin{align*}
&\Phi(t_0)\leq C\big[1+\|\nabla \mathbf{u}_{0}\|_{L^{2}(\Omega)}^{2}+
\|{\bf d}_0\|_{H^2(\Omega)}^2+\|{\bf h}\|_{H^{\frac52,\frac54}(\Gamma_T)}^2
+\|\partial_t{\bf h}\|_{L^2_tH^\frac32_x(\Gamma_T)}^2\big]\\
&\qquad\ \ \ +\frac12\Phi(t_0).
\end{align*}
This further yields
\begin{align*}
&\Phi(t_0)\leq C\big[ 1+\|\nabla \mathbf{u}_{0}\|_{L^{2}(\Omega)}^{2}+
\|{\bf d}_0\|_{H^2(\Omega)}^2+\|{\bf h}\|_{H^{\frac52,\frac54}(\Gamma_T)}^2
+\|\partial_t{\bf h}\|_{L^2_tH^\frac32_x(\Gamma_T)}^2\big].
\end{align*}
Hence \eqref{globstrong} holds with $T$ replaced by $t_0$.
Next we can repeat the same argument as above to show that \eqref{globstrong}
also holds with $T$ replaced by $2t_0$. After iterating finitely many times,
we can see that \eqref{globstrong} holds with $T$. This completes the proof of Theorem \ref{gsolu}.
$\hfill\Box$

\medskip

Next we will establish the continuous dependence  of the global strong solution to 
the system \eqref{eq1.1}--\eqref{eq1.3} for initial  data
in  $\mathbf{V}\times {H}^2(\Omega,\mathbb{S}_{+}^{2})$ and boundary data in $H^{\frac52, \frac54}(\Gamma_T)$ , which is crucial to
the Fr\'echet differentiability of the control to state operator $\mathcal{S}$.

\begin{theorem} \label{gsoluC} Under the same assumptions of Theorem \ref{gsolu},
let $(\mathbf{u}^{(i)}, \mathbf{d}^{(i)})$, $i=1,2$, be the global strong solution
corresponding to the initial data $(\mathbf{u}_0^{(i)}, \mathbf{d}_0^{(i)})$ and 
the boundary data $({0},\mathbf{h}^{(i)})$.
Define $\overline{\mathbf{u}}=\mathbf{u}^{(1)}-\mathbf{u}^{(2)}$, $\overline{\mathbf{d}}=\mathbf{d}^{(1)}-\mathbf{d}^{(2)}$, $\overline{\mathbf{u}}_{0}=\mathbf{u}^{(1)}_{0}-\mathbf{u}^{(2)}_{0}$, $\overline{\mathbf{d}}_{0}=\mathbf{d}^{(1)}_{0}-\mathbf{d}^{(2)}_{0}$
and $\overline{\mathbf{h}}= \mathbf{h}^{(1)}-\mathbf{h}^{(2)}$. Then it holds that
\begin{align}
&\|\overline{\mathbf{u}}(t)\|_{H^{1}(\Omega)}^{2}+\|\overline{\mathbf{d}}(t)\|_{H^{2}(\Omega)}^{2}
+\int_{0}^{t}(\|\overline{\mathbf{u}}(\tau)\|_{H^{2}(\Omega)}^{2}
+\|\overline{\mathbf{d}}(\tau)\|_{H^{3}(\Omega)}^{2})\,d\tau \nonumber\\
&\leq {C}_T\big(\|\overline{\mathbf{u}}_{0}\|_{H^{1}(\Omega)}^{2}
+\| \overline{\mathbf{d}}_{0}\|_{H^{2}(\Omega)}^{2}
+
\|\overline{\mathbf{h}}\|_{H^{\frac{5}{2},\frac{5}{4}}(\Gamma_{t})}^{2}\big)\quad \forall\, t\in [0,T],\label{conti}
\end{align}
where ${C}_T>0$, depending on $\|{\bf u}^{(i)}_0\|_{H^{1}(\Omega)}$,
 $\|{\bf d}^{(i)}_0\|_{H^{2}(\Omega)}$, 
 $\|{\bf h}^{(i)}\|_{H^{\frac{5}{2},\frac{5}{4}}(\Gamma_{T})}$,
 and $\|\partial_t{\bf h}^{(i)}\|_{L^2_tH^\frac32_x(\Gamma_T)}$ for $i=1,2$, 
 $\varepsilon_{1}$, $\Omega$, and $T$.
\end{theorem}

\begin{proof}
For $0<t\leq T$, define
$$\Phi(t)=\sum_{i=1}^2 (\|\mathbf{u}^{(i)}\|_{L^{\infty}_{t}H^{1}_{x}(Q_{t})}^{2}
+\|\mathbf{d}^{(i)}\|_{L^{\infty}_{t}H^{2}_{x}(Q_{t})}^{2}),$$
 and
$$\Psi(t)=\int_{0}^{t}\sum_{i=1}^2\big(\|\mathbf{u}^{(i)}(\tau)\|_{H^{2}(\Omega)}^{2}
 +\|\mathbf{d}^{(i)}(\tau)\|_{H^{3}(\Omega)}^{2}\big)\,d\tau.$$
By Theorem \ref{gsolu}, there exists $C_T>0$,
depending on $\|\mathbf{u}_{0}^{(i)}\|_{H^{1}(\Omega)}$,
 $\|\mathbf{d}_{0}^{(i)}\|_{H^{2}(\Omega)}$,
    $\|\mathbf{h}^{(i)}\|_{H^{\frac{5}{2},\frac{5}{4}}(\Gamma_{T})}$, and $\|\partial_t{\bf h}^{(i)}\|_{L^2_tH^\frac32_x(\Gamma_T)}$
    for $i=1,2$, $\varepsilon_{1}$, $\Omega$ and $T$, such that
 \begin{align}   \label{globstrong1}
  &\Phi(T)+\Psi(T)\leq C_T^2.
 \end{align}
Moreover, for any $0<\delta<1$, there exists $t_\delta\in (0,\delta^{6}]$ such that
\begin{align} \label{delta1}
\Psi(t)\leq \delta^2 \quad \text{ for all } 0<t\leq t_{\delta}.
\end{align}

Observe that $(\overline{\mathbf{u}},\overline{\mathbf{d}})$ satisfies the system
\begin{align}\label{NLC5}
\begin{cases}
\partial_{t}\overline{\mathbf{u}}-\!\Delta\overline{\mathbf{u}} 
+\mathbf{u}^{(1)}\cdot\nabla \overline{\mathbf{u}}+\overline{\mathbf{u}}\cdot\nabla \mathbf{u}^{(2)}+\nabla \overline{P}
\\
=-\nabla\!\cdot\!(\nabla \overline{\mathbf{d}}\odot\nabla \mathbf{d}^{(1)}
+\nabla \mathbf{d}^{(2)}\odot\nabla \overline{\mathbf{d}})
&\text{in }Q_{T},\\
\nabla \cdot\overline{\mathbf{u}}=0\quad &\text{in }Q_{T},\\
\partial_{t}\overline{\mathbf{d}}-\!\Delta\overline{\mathbf{d}}
+\mathbf{u}^{(1)}\cdot\nabla \overline{\mathbf{d}}+\overline{\mathbf{u}}\cdot\nabla \mathbf{d}^{(2)}\\
=|\nabla \mathbf{d}^{(1)}|^{2}\overline{\mathbf{d}}
+2\langle\nabla (\mathbf{d}^{(1)}+\mathbf{d}^{(2)}), \nabla \overline{\mathbf{d}}\rangle \mathbf{d}^{(2)}
 &\text{in }Q_{T},\\
\big(\overline{\mathbf{u}},\overline{\mathbf{d}}\big)
=\big(0, \overline{\mathbf{h}}\big) \quad&\text{on }\Gamma_{T},\\
\big(\overline{\mathbf{u}}, \overline{\bf d}\big)\big|_{t=0}=\big(\overline{\mathbf{u}}_{0},
\overline{\mathbf{d}}_{0}\big)
\quad &\text{in }\Omega.
\end{cases}
\end{align}
Let $\overline{\mathbf{h}}_{P}$ be the parabolic lifting function of $\overline{\bf h}$: 
 \begin{align}
 \begin{cases}
 \partial_t\overline{\mathbf{h}}_P-\Delta \overline{\mathbf{h}}_P={0},\quad &\text{ in } Q_{T},\\
 \overline{\mathbf{h}}_P=\overline{\mathbf{h}},\quad   &\text{ on } \Gamma_{T},\\
 \overline{\mathbf{h}}_P|_{t=0}= \overline{\mathbf{d}}_{0}, \quad   &\text{ in } \Omega.
 \end{cases}\label{LP1}
 \end{align}
By the regularity theory of parabolic equations, we have that
\begin{align} \label{estimateofdp}
&\|\overline{\mathbf{h}}_{P}\|_{L^{\infty}_tH^{2}_x(Q_T)}+
\|\overline{\mathbf{h}}_{P}\|_{H^{3,\frac{3}{2}}(Q_{T})}\nonumber\\
&\leq C\big[\|\partial_t\overline{\bf h}\|_{L^2_tH^\frac32_x(\Gamma_T)}
\|\overline{\mathbf{h}}\|_{H^{\frac{5}{2},\frac{5}{4}}(\Gamma_{T})} 
+\|\overline{\mathbf{d}}_{0}\|_{H^{2}(\Omega)}\big].
\end{align}
Set $\widehat{{\mathbf{d}}}
=\overline{\mathbf{d}}-\overline{\mathbf{h}}_{P}$.
Then $\widehat{{\mathbf{d}}}$ solves
\begin{align*}
\begin{cases}
\partial_{t}\widehat{{\mathbf{d}}}-\Delta\widehat{{\mathbf{d}}}
=-\mathbf{u}^{(1)}\cdot\nabla\overline{\mathbf{d}}
-\overline{\mathbf{u}}\cdot\nabla \mathbf{d}^{(2)}+|\nabla \mathbf{d}^{(1)}|^{2}\overline{\mathbf{d}}\\
 \qquad\qquad\qquad+2\langle\nabla(\mathbf{d}^{(1)}+\mathbf{d}^{(2)}),\nabla\overline{\mathbf{d}} 
 \rangle\mathbf{d}^{(2)}\quad & \text{ in }Q_{T},\\
\widehat{{\mathbf{d}}}=0\quad & \text{ on } \partial_p Q_{T}.
\end{cases}
\end{align*}
It follows from the regularity theory of parabolic equations and \eqref{estimateofdp} that for $0<t\leq T$,
\begin{align} \label{cstrongford}
&\|\overline{\mathbf{d}}\|_{L^{\infty}_{t}H^{2}_{x}(Q_{t})} +\|\overline{\mathbf{d}}\|_{H^{3,\frac{3}{2}}(Q_T)}\\
&\leq C\big[ \|\overline{\mathbf{h}}\|_{H^{\frac52, \frac54}(\Gamma_T)}+\|\partial_t\overline{\bf h}\|_{L^2_tH^\frac32_x(\Gamma_T)}
+\|\overline{\mathbf{d}}_{0}\|_{H^{2}(\Omega)} +
\|\mathbf{u}^{(1)}\cdot\nabla \bar{\mathbf{d}}\|_{H^{1,\frac{1}{2}}(Q_{t})}
\nonumber\\
&+\|\bar{\mathbf{u}}\cdot\nabla \mathbf{d}^{(2)}\|_{H^{1,\frac{1}{2}}(Q_{t})}
+
\||\nabla \mathbf{d}^{(1)}|^{2} \bar{\mathbf{d}}\|_{H^{1,\frac{1}{2}}(Q_{t})}
\nonumber\\
&+\|\langle\nabla(\mathbf{d}^{(1)}+\mathbf{d}^{(2)}),\nabla\bar{\mathbf{d}} \rangle\mathbf{d}^{(2)}\|_{H^{1,\frac{1}{2}}(Q_{t})}
\big]\nonumber.
\end{align}
Since  the last four terms of the right hand side of \eqref{cstrongford}
can be estimated in a way similar to that of the proof of Theorem \ref{gsolu},
we will only sketch below the estimate of 
$\||\nabla \mathbf{d}^{(1)}|^{2} \overline{\mathbf{d}}\|_{H^{1,\frac{1}{2}}(Q_{t})}$.
As in Theorem \ref{gsolu} ,  we first have
\begin{align} \label{d12bard}
\||\nabla \mathbf{d}^{(1)}|^{2} \overline{\mathbf{d}}\|_{H^{1,\frac{1}{2}}(Q_{t})}
&\leq C\big[\||\nabla \mathbf{d}^{(1)}|^{2} \overline{\mathbf{d}}\|_{L^{2}(Q_{t})}
+\|\nabla(|\nabla \mathbf{d}^{(1)}|^{2} \overline{\mathbf{d}})\|_{L^{2}(Q_{t})}
\nonumber\\
&\ \ \ \ \ \ \ +\|\langle\partial_{t}(|\nabla \mathbf{d}^{(1)}|^{2} \overline{\mathbf{d}})\|
_{L^\frac32(Q_t)}\big].
\end{align}
Here we have used Sobolev's embedding:
$W^{1,1}_{\frac32}(Q_T)\hookrightarrow H^{\frac12,\frac12}(Q_T)$.

Applying H\"older's inequality and the Sobolev inequality, we 
obtain that for $0<t\le t_\delta$, the following estimates hold:
\begin{align*}
\||\nabla \mathbf{d}^{(1)}|^{2} \overline{\mathbf{d}}\|_{L^{2}(Q_{t})}
\leq & \||\nabla \mathbf{d}^{(1)}|^{2}\|_{L^2(Q_t)} \|\overline{\mathbf{d}}\|_{L^{\infty}(Q_{t})}\nonumber\\
\leq &\|\nabla \mathbf{d}^{(1)}\|_{L^{\infty}_{t}L^{2}_{x}(Q_{t})}
\|\mathbf{d}^{(1)}\|_{L^{2}_{t}H^{2}_{x}(Q_{t})}
\|\overline{\mathbf{d}}\|_{L^{\infty}_{t}H^{2}_{x}(Q_{t})}
\nonumber\\
\leq &C_T\delta \|\overline{\mathbf{d}}\|_{L^{\infty}_{t}H^{2}_{x}(Q_{t})},
\nonumber
\end{align*}
and
\begin{align*}
\|\nabla(|\nabla \mathbf{d}^{(1)}|^{2} \bar{\mathbf{d}})\|_{L^{2}(Q_{t})}
\leq &C\big[\|\nabla \mathbf{d}^{(1)}\|_{L^{4}_{t} L^{8}_{x}(Q_{t})}^{2}
\|\nabla\overline{\mathbf{d}}\|_{L^{\infty}_{t}L^{4}_{x}(Q_{t})}
\nonumber\\
&\ \ \ \ +\|\nabla\mathbf{d}^{(1)}\|_{L^{4}(Q_{t})}
\|\nabla^{2}\mathbf{d}^{(1)}\|_{L^{4}(Q_{t})}
\|\overline{\mathbf{d}}\|_{L^{\infty}(Q_{t})}\big]
\nonumber\\
\leq &
C\|\mathbf{d}^{(1)}\|_{L^{4}_{t}H^{2}_{x}(Q_{t})}^{2}\|\overline{\mathbf{d}}\|_{L^{\infty}_{t}H^{2}_{x}(Q_{t})}\\
\leq & C\|{\bf d}^{(1)}\|_{L^\infty_tH^2_x(Q_t)}\|{\bf d}^{(1)}\|_{L^2_tH^2_x(Q_t)}
\|\overline{\mathbf{d}}\|_{L^{\infty}_{t}H^{2}_{x}(Q_{t})}\\
\leq & C_{T}\delta\|\overline{\mathbf{d}}\|_{L^{\infty}_{t}H^{2}_{x}(Q_{t})}.
\end{align*}
Applying the equation of ${\bf d}^{(1)}$,  H\"{o}lder's inequality, and the interpolation inequality, 
we can estimate
\begin{align*}
&\|\partial_{t}\nabla\mathbf{d}^{(1)}\|_{L^{2}(Q_{t})}\\
&\leq 
\|\nabla\Delta \mathbf{d}^{(1)}\|_{L^{2}(Q_{t})}
+
\|\nabla(\mathbf{u}^{(1)}\cdot\nabla \mathbf{d}^{(1)})\|_{L^{2}(Q_{t})}
+\|\nabla(|\nabla\mathbf{d}^{(1)}|^{2}\mathbf{d}^{(1)})\|_{L^{2}(Q_{t})}
\nonumber\\
&\leq\|\mathbf{d}^{(1)}\|_{L^{2}_{t}H^{3}_{x}(Q_{t})}
+
\big(\|\mathbf{u}^{(1)}\|_{L^{4}(Q_{t})}+\|\nabla\mathbf{d}^{(1)}\|_{L^{4}(Q_{t})}\big)
\|\nabla^{2} \mathbf{d}^{(1)}\|_{L^{4}(Q_{t})}\\
&\ \ \ \ +
\|\nabla\mathbf{u}^{(1)}\|_{L^{4}(Q_{t})}\|\nabla\mathbf{d}^{(1)}\|_{L^{4}(Q_{t})}
+
\|\nabla\mathbf{d}^{(1)} \|_{L^{6}(Q_{t})}^{3}
\nonumber\\
&\leq  C_T.
\end{align*}
Applying $\eqref{NLC5}_{3}$ and H\"{o}lder's inequality, we have that
\begin{align*}
&\|\partial_{t}\overline{\mathbf{d}}\|_{L^{2}(Q_{t})}\\
&\leq \|\Delta\overline{\mathbf{d}}\|_{L^{2}(Q_{t})}
+\|\mathbf{u}^{(1)}\cdot\nabla\overline{\mathbf{d}}\|_{L^{2}(Q_{t})}
+\|\overline{\mathbf{u}}\cdot\nabla\mathbf{d}^{(2)}\|_{L^{2}(Q_{t})}
\nonumber\\
&\ \ \ \ +\||\nabla \mathbf{d}^{(2)}|^{2}\overline{\mathbf{d}}\|_{L^{2}(Q_{t})}
+2\|\langle \nabla(\mathbf{d}^{(1)}+\mathbf{d}^{(2)}),\nabla\overline{\mathbf{d}}
\rangle\mathbf{d}^{(2)}\|_{L^{2}(Q_{t})}
\nonumber\\
&\leq  Ct^\frac12\|\overline{\mathbf{d}}\|_{L^{\infty}_{t}H^{2}_{x}(Q_{t})}
+\|\mathbf{u}^{(1)}\|_{L^{4}(Q_{t})}
\|\nabla \overline{\mathbf{d}}\|_{L^{4}(Q_{t})}\\
&\ \ \ +\|\overline{\mathbf{u}}\|_{L^{4}(Q_{t})}
\|\nabla \overline{\mathbf{d}}^{(2)}\|_{L^{4}(Q_{t})}
+\|\nabla\mathbf{d}^{(2)}\|_{L^{4}(Q_{t})}^{2}
\|\overline{\mathbf{d}}\|_{L^{\infty}_{t}H^{2}_{x}(Q_{t})}\\
&\ \ \ +2(\|\nabla\mathbf{d}^{(1)}\|_{L^{4}(Q_{t})}
+\|\nabla\mathbf{d}^{(2)}\|_{L^{4}(Q_{t})})
\|\nabla\overline{\mathbf{d}}\|_{L^{4}(Q_{t})}
\nonumber\\
&\leq t^\frac12 \|\overline{\bf d}\|_{L^\infty_tH^2_x(Q_t)}
+C\Psi^\frac12(t)
\big( \|\overline{\mathbf{u}}\|_{L^{\infty}_{t}H^{1}_{x}(Q_{t})}
+ \|\overline{\mathbf{d}}\|_{L^{\infty}_{t}H^{2}_{x}(Q_{t})}\big)\\
&\leq C\delta \big( \|\overline{\mathbf{u}}\|_{L^{\infty}_{t}H^{1}_{x}(Q_{t})}
+ \|\overline{\mathbf{d}}\|_{L^{\infty}_{t}H^{2}_{x}(Q_{t})}\big).
\end{align*}
Hence
\begin{align*}
&\|\partial_{t}(|\nabla \mathbf{d}^{(1)}|^{2} \overline{\mathbf{d}})\|_{L^\frac32(Q_t)}\\
&\leq \|\nabla{\bf d}^{(1)}\|_{L^{12}(Q_t)}^2\|\partial_t\overline{\bf d}\|_{L^2(Q_t)} 
+
\|\partial_{t}\nabla\mathbf{d}^{(1)}\|_{L^{2}(Q_{t})}
\|\nabla\mathbf{d}^{(1)}\|_{L^{6}(Q_{t})}
\|\overline{\mathbf{d}}\|_{L^{\infty}_{t}H^{2}_{x}(Q_{t})}
\nonumber\\
&\leq Ct^\frac16\|{\bf d}^{(1)}\|_{L^\infty_t H^2_x(Q_t)}^2 
\|\partial_t\overline{\bf d}\|_{L^2(Q_t)} \\
&\ \ \ +Ct^\frac16 \|\partial_{t}\nabla\mathbf{d}^{(1)}\|_{L^{2}(Q_{t})}\|{\bf d}^{(1)}\|_{L^\infty_t H^2_x(Q_t)}
\|\overline{\mathbf{d}}\|_{L^{\infty}_{t}H^{2}_{x}(Q_{t})}\\
&\leq (1+C_T^2)t^\frac16\big(\|\partial_t\overline{\bf d}\|_{L^2(Q_t)}
+\|\overline{\mathbf{d}}\|_{L^{\infty}_{t}H^{2}_{x}(Q_{t})}\big)\\
&\leq (1+C_T^2)t^\frac16\big(\|\overline{\bf u}\|_{L^\infty_tH^1_x(Q_t)}
+\|\overline{\mathbf{d}}\|_{L^{\infty}_{t}H^{2}_{x}(Q_{t})}\big),
\end{align*}
where we have used the Sobolev inequalities
\begin{align*}
\|\nabla\mathbf{d}^{(1)}\|_{L^{6}(Q_{t})}
\leq C t^{\frac{1}{6}}\|\mathbf{d}^{(1)}\|_{L^{\infty}_{t}H^{2}_{x}(Q_{t})}, 
\|\nabla\mathbf{d}^{(1)}\|_{L^{12}(Q_{t})}
\leq C t^{\frac{1}{12}}\|\mathbf{d}^{(1)}\|_{L^{\infty}_{t}H^{2}_{x}(Q_{t})}.
\end{align*}
Putting all these estimates into \eqref{d12bard}, we obtain that 
for $0\le t\le t_\delta$, 
\begin{align*}
&\||\nabla \mathbf{d}^{(1)}|^{2} \bar{\mathbf{d}}\|_{H^{1,\frac12}(Q_{t})}
\leq C_T\delta(\| \bar{\mathbf{u}}\|_{L^{\infty}_{t}H^{1}_{x}(Q_{t})}
+\|\bar{\mathbf{d}}\|_{L^{\infty}_{t}H^{2}_{x}(Q_{t})}).
\end{align*}
Similarly, we can estimate
\begin{align*}
\big\|{\bf u}^{(1)}\cdot\nabla\overline{\bf d}\big\|_{H^{1,\frac12}(Q_t)}
\le C_T\delta \big(\|\overline{\bf u}\|_{L^\infty_tH^1_x(Q_t)}
+\|\overline{\bf d}\|_{L^\infty_tH^2_x(Q_t)}+\|\overline{\bf d}\|_{L^2_tH^3_x(Q_t)}\big),
\end{align*}
\begin{align*}
\big\|\overline{\bf u}\cdot\nabla {\bf d}^{(2)}\big\|_{H^{1,\frac12}(Q_t)}
\le C_T\delta \big(\|\overline{\bf u}\|_{L^\infty_tH^1_x(Q_t)}
+\|\overline{\bf d}\|_{L^\infty_tH^2_x(Q_t)}+\|\overline{\bf u}\|_{L^2_tH^2_x(Q_t)}\big),
\end{align*}
and
\begin{align*}
&\big\|\langle\nabla ({\bf d}^{(1)}+{\bf d}^{(2)}), \nabla \overline{\bf d}\rangle {\bf d}^{(2)}\big\|_{H^{1,\frac12}(Q_t)}\\
&\le C_T\delta \big(\|\overline{\bf u}\|_{L^\infty_tH^1_x(Q_t)}
+\|\overline{\bf d}\|_{L^\infty_tH^2_x(Q_t)}+\|\overline{\bf d}\|_{L^2_tH^3_x(Q_t)}\big).
\end{align*}
Therefore we obtain that for $0\le t\le t_\delta$,
\begin{align}\label{cstrongford2}
&\|\overline{\mathbf{d}}\|_{L^{\infty}_tH^{2}_x(Q_t)}+
\|\overline{\mathbf{d}}\|_{L^2_tH^{3}_x(Q_{t})}\nonumber\\
&\le C\big[\|\overline{\mathbf{h}}\|_{H^{\frac52, \frac54}(\Gamma_T)}
\|\partial_t\overline{\bf h}\|_{L^2_tH^\frac32_x(\Gamma_T)}+
\|\overline{\mathbf{d}}_0\|_{H^{2}(\Omega)}\nonumber\\
&\ + \delta(\|\overline{\bf u}\|_{L^\infty_tH^1_x(Q_t)}
+\|\overline{\bf d}\|_{L^\infty_tH^2_x(Q_t)}+\|\overline{\bf u}\|_{L^2_tH^2_x(Q_t)}+\|\overline{\bf d}\|_{L^2_tH^3_x(Q_t)})\big].
\end{align}

We can apply the $W^{2,1}_{2}$-regularity theory of $\eqref{NLC5}_{1}$
to estimate $\bar{\mathbf{u}}$ as follows. For $0<t\leq t_\delta$, it holds that
\begin{align} \label{cstrongforu}
&\|\overline{\mathbf{u}}\|_{L^{\infty}_{t}H^{1}_{x}(Q_{t})} +\|\overline{\mathbf{u}}\|_{L^2_t H^2_x(Q_t)}\nonumber\\
&\leq C\big[ \|\overline{\mathbf{u}}_{0}\|_{H^{1}(\Omega)}
+\|\mathbf{u}^{(1)}\cdot\nabla \overline{\mathbf{u}}\|_{L^{2}(Q_{t})}
+\|\overline{\mathbf{u}}\cdot\nabla \mathbf{u}^{(2)}\|_{L^{2}(Q_{t})}
\nonumber\\
&+\|\nabla\cdot(\nabla\overline{\mathbf{d}}\odot\nabla \mathbf{d}^{(1)}
+\nabla\mathbf{d}^{(2)}\odot\nabla\overline{\mathbf{d}})\|_{L^{2}(Q_{t})}\big]\nonumber\\
&\le C\big[\|\overline{\mathbf{u}}_{0}\|_{H^{1}(\Omega)}+C_T\delta (\|\overline{\bf u}\|_{L^\infty_tH^1_x(Q_t)}
+\|\overline{\bf d}\|_{L^\infty_tH^2_x(Q_t)})\big].
\end{align}
For $0\le t\le T$, set  
\begin{align*}
&\overline{\Phi}(t)
=(\|\overline{\mathbf{u}}\|_{L^{\infty}_{t}H^{1}_{x}(Q_{t})}^{2}
+\|\overline{\mathbf{d}}\|_{L^{\infty}_{t}H^{2}_{x}(Q_{t})}^{2})
+(\|\overline{\mathbf{u}}\|_{L^{2}_{t}H^{2}_{x}(Q_{t})}^{2}
+\| \overline{\mathbf{d}}\|_{L^{2}_{t}H^{3}_{x}(Q_{t})}^{2}).
\end{align*}
Then \eqref{cstrongford2} and \eqref{cstrongforu} imply
that for $0\le t\le t_\delta$,
\begin{align*}
\overline{\Phi}(t)
&\leq C_T\delta\overline{\Phi}(t) \\
&\ \ \ +C\big[\|\overline{\mathbf{u}}_{0}\|_{H^{1}(\Omega)}^{2}
+\| \overline{\mathbf{d}}_{0}\|_{H^{2}(\Omega)}^{2}
+
\|\overline{\mathbf{h}}\|_{H^{\frac{5}{2},\frac{5}{4}}(\Gamma_{t})}^{2}+\|\partial_t\overline{\bf h}\|_{L^2_tH^\frac32_x(\Gamma_T)}^2\big].
\end{align*}
If we choose $\delta>0$ such that $C_{T}\delta\leq\frac{1}{2}$, then
we conclude that \eqref{conti} holds for all $ t\in [0,t_{\delta}]$.  
By repeating the same argument for $t\in [it_\delta, (i+1)t_\delta]$ for
$i=1, \cdots, \big[\frac{T}{t_\delta}]+1$, we see that  \eqref{conti} holds for all $ t\in [0,T]$.
This completes the proof of Theorem \ref{gsoluC}.
\end{proof}

\section{Optimal boundary control}

The second main purpose of this paper is to consider the optimal boundary control problem \eqref{CF} for
the nematic liquid crystal flow \eqref{eq1.1}--\eqref{eq1.3}.
For a given $0<T<\infty$, we make the following assumptions:
\begin{itemize}
\item[(A1)] $\beta_i\geq 0$ ($i=1,2,3,4,5$) are constants that do not vanish simultaneously.

\item[(A2)] The vector-valued functions
$$
  \mathbf{u}_{Q_{T}}\in L^{2}([0,T], \mathbf{H}),\
 \mathbf{d}_{Q_{T}} \in {L}^2(Q_{T},\mathbb{S}^{2}),\ \mathbf{u}_{\Omega}\in \mathbf{H},
\ \mathbf{d}_{\Omega} \in {L}^2(\Omega, \mathbb{S}^{2})
$$
are given target maps.
\end{itemize}
The optimal boundary control problem (\ref{CF}) seeks a boundary data ${\bf h}$ in a suitable function space
that minimizes the cost functional:
\begin{align} \label{CF1}
2\mathcal{C}((\mathbf{u}, \mathbf{d}), \mathbf{h}) &:=\beta_1
\|\mathbf{u}-\mathbf{u}_{Q_T}\|^2_{{L}^2(Q_{T})}
+\beta_2\|\mathbf{d}-\mathbf{d}_{Q_T}\|^2_{{L}^2(Q_{T})}\nonumber\\
&\ \ \ +\beta_3\|\mathbf{u}(T)-\mathbf{u}_\Omega\|^2_{{L}^2(\Omega)}
+\beta_4\|\mathbf{d}(T)-\mathbf{d}_\Omega\|^2_{{L}^2(\Omega)}\nonumber\\
&\ \ \ +{\beta}_5\|\mathbf{h}-\mathbf{e}_3\|^2_{{L}^2(\Gamma_{T})},
\end{align}
where $({\bf u}, {\bf d})$ is the unique strong solution of \eqref{eq1.1}--\eqref{eq1.3}
under the boundary condition $(0, \mathbf{h})$ and the initial condition
$({\bf u}_0, {\bf d}_0)$.

\subsection{Fr\'echet differentiability of the control to state map}

In this subsection, we will study the control to state map $\mathcal{S}$ and establish
its Fr\'echet differentiability over suitable function spaces.

\subsubsection{Function space of admissible boundary control data}

The natural function space for the  boundary control data $\mathbf{h}$,
that guarantees the existence of unique strong solutions to the system \eqref{eq1.1}--\eqref{eq1.3}, is
\begin{align}\label{U1}
\mathcal{U}&\equiv\Big\{\mathbf{h}\ |\  \mathbf{h}\in {H}^{\frac{5}{2},\frac{5}{4}}(\Gamma_T, 
\mathbb{S}_{+}^{2}) \text{ and } \partial_{t}
\mathbf{h} \in L^{2}_tH^{\frac{3}{2}}_x(\Gamma_T)\Big\},
\end{align}
which is equipped with the norm
$$
\big\|\mathbf{h}\big\|_{\mathcal{U}}:
=\big\|\mathbf{h}\big\|_{{H}^{\frac{5}{2},\frac{5}{4}}(\Gamma_{T}))}
+\big\|\partial_{t}\mathbf{h}\big\|_{L^{2}_tH^{\frac{3}{2}}_x(\Gamma_T)}, \ \mathbf{h}\in \mathcal{U}.
$$
Given an initial data $(\mathbf{u}_0, \mathbf{d}_0)\in \mathbf{V}\times
{H}^{2}(\Omega,\mathbb{S}_{+}^{2})$, the \emph{function space for boundary control functions}
$\widetilde{\mathcal{U}}$, associated with $(\mathbf{u}_0, \mathbf{d}_0)$,  is defined by
\begin{align}
\widetilde{\mathcal{U}}& :=\Big\{\mathbf{h}\ |\ \mathbf{h}\in\mathcal{U}, 
\  \text{with}\ \mathbf{h}(x,0)=\mathbf{d}_{0}(x)\ {\rm{on}}\ \Gamma\Big\}.
\label{U}
\end{align}

\begin{remark}{\rm 
By the Aubin-Lions Lemma and the Sobolev embedding Theorem, we have that
\begin{align*}
&\mathcal{U}
\hookrightarrow C([0,T],  H^\frac{3}{2}(\Gamma, \mathbb{S}_{+}^{2}))
\hookrightarrow C(\Gamma_T, \mathbb{S}_{+}^{2}),
\end{align*}
and hence any $\mathbf{h}\in \widetilde{\mathcal{U}}$ is continuous  on $\Gamma_T$
and the compatibility condition $\mathbf{h}(x,0)=\mathbf{d}_0(x)$ holds for $x\in\Gamma$ 
in the classical sense.}
\end{remark}

The minimization problem is taken place in a bounded, {\it intrinsically} convex closed set in $\widetilde{\mathcal{U}}$
that will be specified below.

Let $\Pi: \mathbb {S}^2\setminus\{-{\bf e}_3\}\to \mathbb {R}^2$ be the stereographic projection
from the south pole $-{\bf e}_3$, and $\Pi^{-1}$ be its inverse map. Then $\Pi:\mathbb {S}^2_+\mapsto
B_1^2=\big\{y\in\mathbb R^2: \ |y|\le 1\big\}$ is a smooth differeomorphism. It is clear that
any map ${\bf h}:\Gamma_T\mapsto\mathbb{S}^2_+$ belongs to $\widetilde{\mathcal{U}}$ if and only if
$\Pi({\bf h}):\Gamma_T\mapsto B_1^2$ satisfies 
$$\Pi({\bf h})\in H^{\frac52,\frac54}(\Gamma_T, B_1^2)
\ \ {\rm{and}}\ \ \partial_t({\Pi({\bf h})})\in L^2_tH^\frac32_x(\Gamma_T),$$
and $\Pi({\bf h})(x,0)=\Pi({\bf d}_0)(x)$ for $x\in\Gamma$.

We also equip $\Pi({\bf h})$ with the norm 
$$\big\|\Pi({\bf h})\big\|_{\mathcal{U}}
:=\big\|\Pi({\bf h})\big\|_{H^{\frac52,\frac54}(\Gamma_T)}+
\big\|\partial_t({\Pi({\bf h})})\big\|_{L^2_tH^\frac32_x(\Gamma_T)}.
$$
\begin{definition}\label{UadM} For $M>0$, we define the intrinsic ball in $\widetilde{\mathcal{U}}$
with center $0$ and radius $M$, denoted as
$\widetilde{\mathcal{U}}_{M}$, by
\begin{align} \label{Uad}
\widetilde{\mathcal{U}}_{M}=
&\Big\{\mathbf{h}\in \widetilde{\mathcal{U}} \ | \ \|\Pi(\mathbf{h})\|_{\mathcal{U}}\leq M \Big\}.
\end{align}
\end{definition}

It is not hard to see that for sufficiently large $M>0$, $\widetilde{\mathcal{U}}_M\not=\emptyset$.
In fact, there exists $C>1$ such that if $M\ge C\|{\bf d}_0\|_{H^3(\Omega)}$
then we can construct ${\bf h}\in\widetilde{\mathcal{U}}$ such that $\|\Pi({\bf h})\|_{\mathcal{U}}\le M$
and ${\bf h}(\cdot, 0)={\bf d}_0(\cdot)$ on $\Gamma$. For example, let ${\bf h}:\Gamma_T\mapsto\mathbb S^2$
be the solution to the heat flow of harmonic map from $\Gamma$ to $\mathbb S^2$:
\begin{align*}
\begin{cases}
\partial_t{\bf h}-\Delta_{\Gamma}{\bf h}=|\nabla_{\Gamma}{\bf h}|^2{\bf h} & \ \ {\rm{in}}\ \ \ \Gamma_T,\\
{\bf h}(\cdot,0)={\bf d}_0(\cdot)& \  \ {\rm{on}}\ \ \ \Gamma.
\end{cases}
\end{align*}
Here $\nabla_{\Gamma}$ and $\Delta_{\Gamma}$ denote the gradient and Laplace operator on $\Gamma$.
Since $\Gamma$ is a $1$-dimensional smooth closed curve and ${\bf d}_0\in H^\frac52(\Gamma,\mathbb {S}^2_+)$,
it follows from the standard  theory of heat flow of harmonic maps in dimensions one that
there exists a unique solution ${\bf h}\in H^{\frac52, \frac54}(\Gamma_T,\mathbb {S}^2)$, with $\partial_t{\bf h}\in
L^2_tH^\frac32_x(\Gamma_T)$, such that 
$$\|\Pi({\bf h})\|_{\mathcal{U}}\le C\|{\bf d_0}\|_{H^3(\Omega)}\le M.$$
Moreover, it follows from ${\bf d}_0^3\ge 0$ on $\Gamma$ that ${\bf h}^3\ge 0$. Therefore,
$\widetilde{\mathcal{U}}_M$ is non-empty.
\begin{remark}\label{convex}
{\rm It is clear that $\widetilde{\mathcal{U}}_M$ is convex in the following sense:
if ${\bf h}_1, {\bf h}_2\in \widetilde{\mathcal{U}}_M$, then 
$\Pi^{-1}(s\Pi({\bf h}_1)+(1-s)\Pi({\bf h}_2))\in \widetilde{\mathcal{U}}_M$ for all $s\in [0,1]$.}
\end{remark}

In fact, it follows from the definition of $\widetilde{\mathcal{U}}_M$ that
$\Pi({\bf h}_i):\Gamma_T\mapsto B_1^2$ for $i=1,2$, this implies that
$s\Pi({\bf h}_1)+(1-s)\Pi({\bf h}_2):\Gamma_T\mapsto B_1^2$ for $t\in [0,1]$. 
Also note that
\begin{align*}
\big\|s\Pi({\bf h}_1)+(1-s)\Pi({\bf h}_2)\big\|_{\mathcal{U}}
&\le s\big\|\Pi({\bf h}_1)\big\|_{\mathcal{U}}+(1-s)\big\|\Pi({\bf h}_2)\big\|_{\mathcal{U}}\\
&\le sM+(1-s)M=M.
\end{align*}
Thus ${\bf h}(s)=\Pi^{-1}(s\Pi({\bf h}_1)+(1-s)\Pi({\bf h}_2))\in C^1([0,1], \widetilde{\mathcal{U}}_M)$
is a path joining ${\bf h}(0)={\bf h}_1$ and ${\bf h}(1)={\bf h}_2$.

\subsubsection{The control-to-state operator $\mathcal{S}$} To define $\mathcal{S}$, we first need to introduce
the function space for global strong solutions to the system \eqref{eq1.1}--\eqref{eq1.3}:
\begin{align}\label{H}
\mathcal{H}= C([0, T], \mathbf{V})\cap L^{2}_tH^2_x(Q_T)
\times C([0, T], H^{2}(\Omega,\mathbb{S}_{+}^{2})) \cap
L^{2}_tH^3_x(Q_T),
\end{align}
which is equipped with the norm
$$\big\|({\bf u}, {\bf d})\big\|_{\mathcal{H}}
=\big\|{\bf u}\big\|_{L^\infty_tH^1_x(Q_T)}+\big\|{\bf u}\big\|_{L^2_tH^2_x(Q_T)}
+\big\|{\bf d}\big\|_{L^\infty_tH^2_x(Q_T)}+\big\|{\bf d}\big\|_{L^2_tH^3_x(Q_T)}.
$$
We also introduce the function space for the Fr\'echet derivative of $\mathcal{S}$:
\begin{align}\label{W}
\mathcal{W}= C([0, T], \mathbf{H})\cap L^{2}_tH^1_x(Q_T)
\times C([0, T], H^{1}(\Omega,\mathbb{R}^{3})) \cap
L^{2}_tH^2_x(Q_T),
\end{align}
which is equipped with the norm
\begin{align*}
&\big\|({\bm\omega}, {\bm\phi})\big\|_{\mathcal{W}}\\
&\ =\big\|{\bm \omega}\big\|_{L^\infty_tL^2_x(Q_T)}+\big\|{\bm \omega}\big\|_{L^2_tH^1_x(Q_T)}
+\big\|{\bm\phi}\big\|_{L^\infty_tH^1_x(Q_T)}+\big\| {\bm \phi}\big\|_{L^2_tH^2_x(Q_T)}.
\end{align*}
Note that ${\mathcal{H}}$ is a subset of ${\mathcal{W}}$.
Now we define the \textit{control-to-state}  map $\mathcal{S}$ as follows.

\begin{definition}\label{Def-Scal}
Given an initial data $({\bf u}_0, {\bf d}_0)\in \mathbf{V}\times H^2(\Omega,\mathbb{S}^2_+)$,
the control-to-state mapping $\mathcal{S}:\widetilde{\mathcal{U}}\mapsto \mathcal{H}$, associated with $(\mathbf{u}_0, \mathbf{d}_0)$,  is defined by letting
\begin{align} \label{Scal}
\mathbf{h}\in \widetilde{\mathcal{U}}\mapsto \mathcal{S}(\mathbf{h})=(\mathbf{u},\mathbf{d})\in \mathcal{H}
\end{align}
to be the unique global strong solution to the system
\eqref{eq1.1}--\eqref{eq1.3} on $[0,T]$, with the initial condition $(\mathbf{u}_0, \mathbf{d}_0)$
and the boundary condition $(0, {\bf h})$.
\end{definition}

It follows directly from Theorem \ref{gsolu}, Remark \ref{gsoluR}, and Theorem \ref{gsoluC}
that  the map $\mathcal{S}$ is Lipschitz continuous. More precisely, we have

\begin{proposition}
For $n=2$, $T\in(0,+\infty)$, and $M>0$, under the same assumptions of Theorem \ref{gsolu}, if
$\widetilde{\mathcal{U}}_M\not=\emptyset$, then the control-to-state map $\mathcal{S}$
is Lipschitz continuous from $\widetilde{\mathcal{U}}$ to ${\mathcal{H}}$, i.e.,
$$\big\|\mathcal{S}({\bf h}_1)-\mathcal{S}({\bf h}_2)\big\|_{\mathcal{H}}
\le C_M\big\|{\bf h}_1-{\bf h}_2\big\|_{\mathcal{U}}, \ \forall\ {\bf h}_1, {\bf h}_2\in\widetilde{\mathcal{U}}_M,
$$
where $C_M>0$ depends only on $M$, $\Omega$, $\big\|{\bf u}_0\|_{\mathbf{H}}$,
and $\big\|{\bf d}_0\big\|_{H^1(\Omega)}$.
\end{proposition}

\subsubsection{Differentiability of the control-to-state operator $\mathcal{S}$}
We will establish the differentiability of the control-to-state operator 
$\mathcal{S}:\widetilde{\mathcal{U}}\mapsto\mathcal{H}$. 

First,  we will define the Fr\'echet differentialiblity of $\mathcal{S}$. To do it,
we need to introduce tangential spaces of $\widetilde{\mathcal{U}}$ and
$\mathcal{H}$. Given an element ${\bf h}\in \widetilde{\mathcal{U}}$, the pullback bundle
of the tangent bundle $T\widetilde{\mathcal{U}}$ by ${\bf h}$, ${\bf h}^{*}T\widetilde{\mathcal{U}}$, is defined
by
\begin{align*}
{\bf h}^{*}T\widetilde{\mathcal{U}}&=\Big\{{\bm{\xi}}\in H^{\frac52,\frac54}(\Gamma_T,\mathbb R^3)\ \big|
\ \partial_t {\bm\xi}\in L^2_tH^\frac32_x(\Gamma_T),\ {\bm\xi}(x,0)=0\ {\rm{for}}\ x\in\Gamma,\\
&\ \ \ \ \ \ \langle{\bm \xi}, {\bf h}\rangle(x,t)=0\ {\rm{for}}\ (x,t)\in\Gamma_T
\Big\},
\end{align*}
which is equipped  with the same norm $\|\cdot\|_{\mathcal{U}}$ as that on $\widetilde{\mathcal{U}}$.

For a fixed $({\bf u}_0, {\bf d}_0)\in \mathbf{V}\times H^2(\Omega, \mathbb S^2_+)$
and an element $({\bf u}, {\bf d})\in \mathcal{H}$, with $({\bf u}, {\bf d})=({\bf u}_0, {\bf d}_0)$ at $t=0$,
the pullback bundle of the tangent bundle $T\mathcal{H}$ by $({\bf u}, {\bf d})$, $({\bf u}, {\bf d})^{*}T\mathcal{H}$, 
is defined by
\begin{align*}
({\bf u}, {\bf d})^{*}T\mathcal{H}&=\Big\{
({\bm{\omega}},{\bm\phi})\big|\ {\bm\omega}\in C([0,T],\mathbf{H})\cap L^2_tH^1_x(Q_T),\\
&\ \ \ \  \qquad\qquad{\bm{\phi}}\in C([0,T], H^1(\Omega,\mathbb R^3))\cap L^2_tH^2_x(Q_T),\\
& \ \ \ \qquad\qquad\langle{\bm\phi}, {\bf d}\rangle=0 \ {\rm{a.e.\ in }}\ Q_T, ({\bm\omega}, {\bm\phi})\big|_{t=0}=(0,0)\Big\},
\end{align*}
which is equipped with the norm $\|\cdot\|_{\mathcal{W}}$.

\begin{definition}\label{f-deriv} Given a $({\bf u}_0, {\bf d}_0)\in \mathbf{V}\times H^2(\Omega, \mathbb S^{2}_+)$.
For any ${\bf h}\in \widetilde{\mathcal{U}}$,  
let $({\bf u}, {\bf d})\in \mathcal{H}$ be the unique 
strong solution of \eqref{eq1.1}--\eqref{eq1.3} under the initial condition $({\bf u}_0, {\bf d}_0)$
and the boundary condition $(0, {\bf h})$, we say that the control to state map $\mathcal{S}:
\widetilde{\mathcal{U}}\mapsto\mathcal{H}$ is Fr\'echet differentiable at ${\bf h}$, if there exists a
linear map $\mathcal{S}'(h): {\bf h}^*T\widetilde{\mathcal{U}}\mapsto {\mathcal{S}({\bf h})}^*T\mathcal{H}$,
called the Fr\'echet derivative of $\mathcal{S}$ at ${\bf h}$, 
such that  for any $\epsilon>0$ there exists a $\delta>0$ so that
\begin{align}\label{Dif-F}
\big\|\mathcal{S}(\exp_{\bf h}{\bm\xi})
-\mathcal{S}(\mathbf{h})
-\mathcal{S}'({\bf h})(\bm{\xi})
\big\|_{\mathcal{W}}\le \epsilon\|\bm{\xi}\|_{\mathcal{U}},
\end{align}
whenever ${\bm\xi}$ is any section of ${\bf{h}}^*T\widetilde{\mathcal{U}}$ satisfying
both $\|{\bm\xi}\|_{\mathcal{U}}\le\delta$ and $\exp_{\bf h}(\bm\xi)\in \widetilde{\mathcal{U}}$.
Here $\exp_{\bf h}(\bm\xi)(x,t)$ is the exponential map on $\mathbb S^2$
from ${\bf h}(x,t)$ and in the direction $\bm\xi(x,t)$ for any $(x,t)\in Q_T$.
\end{definition}

Let us make two comments on Definition \ref{f-deriv}.

\begin{remark} {\rm If we denote the strict upper half space by
$${\mathbb S}^{2,\circ}_+=\mathbb{S}^2_+\setminus \partial \mathbb{S}^2_+=\big\{y\in \mathbb S^2: \ y^3>0\big\}.$$
Then for any function ${\bf h}\in \widetilde{\mathcal{U}}$ satisfying 
$${\bf h}(x,t)\in \mathbb{S}^{2,\circ}_+, \ \forall (x,t)\in\Gamma_T,$$
there exists  $\delta=\delta({\bf h})>0$ such that if ${\bm\xi}$ is a section ${\bf h}^*T\widetilde{\mathcal{U}}$
such that
$$\|{\bm\xi}\|_{\mathcal{U}}\le\delta,$$
then the exponential map $(\exp_{\bf h}{\bm \xi})(x,t)=\exp_{{\bf h}(x,t)}{\bm\xi}(x,t): Q_T\mapsto \mathbb S^2$
has the same regularity as ${\bf h}$ and has its third component  $(\exp_{\bf h}{\bm\xi})^3>0$
on $\Gamma_T$. Hence $(\exp_{\bf h}{\bm\xi})(x,t)\in \mathbb{S}^{2,\circ}_+$, for $(x,t)\in\Gamma_T$,
so that $\exp_{\bf h}{\bm\xi}\in \widetilde{\mathcal{U}}$.}

\end{remark}

\begin{remark}{\rm For ${\bf d}_0\in H^2(\Omega,\mathbb{S}^{2,\circ}_+)$ and
${\bf h}\in \widetilde{\mathcal{U}}$ with ${\bf h}(\Gamma_T)\subset\mathbb{S}^{2,\circ}_+$, 
there exist $\delta_1>0$, $\delta_2$, and $\delta_3>0$ depending on $\|{\bf d}_0\|_{H^2(\Omega)}$
and $\|{\bf h}\|_{\mathcal{U}}$ such that
$${\bf d}_0^3(x)\ge\delta_1 \ \ \ \forall x\in\Omega; \ \ \ \  {\bf h}^3(y,t)\ge \delta_1 \ \ \ \forall (y,t)\in\Gamma_T.$$
Hence $({\bf u}, {\bf d})=\mathcal{S}({\bf h})\in\mathcal{H}$ enjoys the property that
${\bf d}\in C^0(Q_T)$ and
$${\bf d}^3(x,t)\ge \delta_2, \ \forall (x,t)\in Q_T.$$
Therefore, for any section ${\bm\xi}$ of ${\bf h}^*T\widetilde{\mathcal{U}}$, if $\|{\bm\xi}\|_{\mathcal{U}}\le\delta_3$
then $\exp_{\bf h}{\bm\xi}$ maps $\Gamma_T$ to $\mathbb{S}^{2}_+$. In particular,
$\exp_{\bf h}{\bm\xi}\in \widetilde{\mathcal{U}}$ and $\mathcal{S}(\exp_{\bf h}{\bm\xi})\in\mathcal{H}$ 
is well-defined in \eqref{Dif-F}.}
\end{remark}

Now we want to study the linearized equation of the system of \eqref{eq1.1}--\eqref{eq1.3}.
\subsubsection{The linearized system}
 
For a fixed $(\mathbf{u}_{0},\mathbf{d}_{0})\in \mathbf{V}\times H^{2}(\Omega,\mathbb{S}_{+}^{2})$, let $\mathbf{h}\in \widetilde{\mathcal{U}}$ be given and  $(\mathbf{u}, \mathbf{d})=\mathcal{S}(\mathbf{h})$
be the unique global strong solution to the system \eqref{eq1.1}--\eqref{eq1.3}, with the initial condition
$(\mathbf{u}_{0},\mathbf{d}_{0})$ and the boundary condition
$(0, {\bf h})$,  given by Theorem \ref{gsolu}.

The linearized system of \eqref{eq1.1}--\eqref{eq1.3} near $((\mathbf{u},\mathbf{d}),\mathbf{h})$,
along a section ${\bm\xi}$ of ${\bf h}^*T\widetilde{\mathcal{U}}$,  seeks a section $(\bm{\omega},\bm{\phi})$
of ${({\bf u}, {\bf d})}^*T\mathcal{H}$ that solves  
\begin{align}\label{NLC-L1}
\begin{cases}
 \partial_{t}\bm{\omega}-\Delta \bm{\omega} +\nabla \mathcal{P}+(\mathbf{u}\cdot\nabla)\bm{\omega} + (\bm{\omega}\cdot\nabla) \mathbf{u} \\
 =-\nabla\cdot(\nabla \bm{\phi}\odot\nabla \mathbf{d})-\nabla\cdot(\nabla \mathbf{d}\odot\nabla \bm{\phi}),\\
 \nabla \cdot \bm{\omega} = 0,\\
 \partial_{t} \bm{\phi}-\Delta \bm{\phi}+(\mathbf{u}\cdot\nabla)\bm{\phi}+(\bm{\omega}\cdot\nabla) \mathbf{d}\\
 =|\nabla\mathbf{d}|^{2}\bm{\phi}+2\langle\nabla\mathbf{d},\nabla \bm{\phi} \rangle \mathbf{d},
\end{cases}
{\rm{in}}\ Q_T
 \end{align}
under the boundary and initial condition
 \begin{align}\label{NLC-L2}
 \begin{cases}
(\bm{\omega}, \ \bm{\phi})=(0, \bm{\xi}),\quad &\text{ on } \Gamma_T,\\
 (\bm{\omega}, \ \bm{\phi})=(0, 0),\quad &\text{ in }\Omega\times\{0\}.
 \end{cases}
 \end{align}

We have the following theorem.

\begin{theorem}\label{Line}
 For any  section $\bm{\xi}$ of ${\mathbf{h}^{*}}T\widetilde{\mathcal{U}}$,
 the system \eqref{NLC-L1} and \eqref{NLC-L2} admits a unique strong solution 
 $(\bm{\omega}, \bm{\phi})$, which is a section of ${({\bf u}, {\bf d})}^*T{\mathcal{H}}$,
 that satisfies the following estimate: 
\begin{align}\label{line1}
&\|(\bm{\omega}, \nabla\bm{\phi})\|_{L^\infty_tL^{2}_x(Q_T)}^{2}
+\int_{0}^{T} (\|\nabla \bm{\omega}(\tau)\|_{L^{2}(\Omega)}^2+
\|\nabla^2 \bm{\phi}(\tau)\|_{L^{2}(\Omega)}^{2})\,d\tau\nonumber\\
&\leq C_{T} \big(\|\partial_{t}\bm{\xi}\|_{L^2_tH^{\frac{3}{2}}_x(\Gamma_T)}^{2}
+\|\bm{\xi}\|_{H^{\frac{5}{2},\frac54}(\Gamma_T)}^{2}\big)
\nonumber\\
&\leq C_{T}\|\bm{\xi}\|_{{\mathcal{U}}}^{2},
\end{align}
where $C_{T}>0$ depends only on $\|\mathbf{u}_{0}\|_{H^{1}(\Omega)}$,  $\|\mathbf{d}_{0}\|_{H^{2}(\Omega)}$,
$\|\mathbf{h}\|_{{\mathcal{U}}}$, $\Omega$ and $T$.
\end{theorem}

%
\proof Since the existence of a strong solution $(\bm{\omega},\bm{\phi})$ can be shown by the standard Galerkin method (see \cite{CRW2017} Proposition 4.1) and the uniqueness 
of $(\bm{\omega},\bm{\phi})$ follows from the estimate \eqref{line1}. We will only prove \eqref{line1}.
First let $\bm{\xi}_{P}$ be the parabolic lift function of $\bm{\xi}$, i.e.,
\begin{align*}
\begin{cases}
\partial_{t}\bm{\xi}_{P}-\Delta\bm{\xi}_{P}=\mathbf{0} \quad &\text{ in } Q_T,\\
\bm{\xi}_{P}=\bm{\xi} \quad &\text{ on }\Gamma_T,\\
\bm{\xi}_{P}={0} \quad &\text{ in }\Omega\times\{0\}.
\end{cases}
\end{align*}
Then we have 
\begin{align}\label{xi_P}
\|\nabla\bm\xi_P\|_{L^\infty_tL^2_x(Q_T)}
+\|\nabla^2\bm\xi_P\|_{L^2(Q_T)}+\|\bm{\xi}_{P}\|_{H^{3, \frac32}(Q_T)}\le C\|\bm{\xi}\|_{\mathcal{U}}.
\end{align}
Multiplying $\eqref{NLC-L1}_{1}$ by $\bm{\omega}$ and $\eqref{NLC-L1}_{3}$ by $\Delta\widetilde{\phi}$,
where $\widetilde{\bm{\phi}}=\bm{\phi}-\bm{\xi}_{P}$, and adding the two resulting equations, 
and using  H\"{o}lder's inequality, the Sobolev embedding Theorem and Poincar\'{e} inequality, 
we obtain
\begin{align*}
&\frac{1}{2}\frac{d}{dt}(\|\bm{\omega}\|_{L^{2}(\Omega)}^{2}
+\|\nabla\widetilde{\bm{\phi}}\|_{L^{2}(\Omega)}^{2})
+(\|\nabla\bm{\omega}\|_{L^{2}(\Omega)}^{2}
+\|\Delta \widetilde{\bm{\phi}}\|_{L^{2}(\Omega)}^{2})
\nonumber\\
&=\int_{\Omega} [-(\bm{\omega}\cdot\nabla)\mathbf{u}\cdot\bm{\omega}
+(\nabla\widetilde{\bm{\phi}}\odot\nabla\mathbf{d}
+\nabla\mathbf{d}\odot \nabla\widetilde{\bm{\phi}}):\nabla\bm{\omega}]
\nonumber\\
&\ \ \ +\int_{\Omega}
(\mathbf{u}\cdot\nabla\widetilde{\bm{\phi}}
+
\bm{\omega}\cdot\nabla\mathbf{d}
-
|\nabla \mathbf{d}|^{2}\widetilde{\bm{\phi}}
-2\langle \nabla\mathbf{d},\nabla\widetilde{\bm{\phi}} \rangle
\mathbf{d})\cdot\Delta\widetilde{\bm{\phi}}
\nonumber\\
&\ \ \ +\!\int_{\Omega}\!
(\nabla \bm{\xi}_{P}\!\odot\!\nabla\mathbf{d}
\!+\!\nabla\mathbf{d}\!\odot\! \nabla \bm{\xi}_{P}):\nabla\bm{\omega}\\
&\ \ \ \!+\!
\int_\Omega(\mathbf{u}\cdot\nabla \bm{\xi}_{P}
\!-\!
|\nabla \mathbf{d}|^{2} \bm{\xi}_{P}
\!-\!2\langle \nabla\mathbf{d},\nabla \bm{\xi}_{P} \rangle
\mathbf{d})\!\cdot\!\Delta\widetilde{\bm{\phi}}
\nonumber\\
&\leq C\Big[
\|\nabla \mathbf{u}\|_{L^{2}(\Omega)}
\|\bm{\omega}\|_{L^{4}(\Omega)}^{2}
\!+\!
\|\nabla \bm{\omega}\|_{L^{2}(\Omega)}
\|\nabla \mathbf{d}\|_{L^{4}(\Omega)}
\|\nabla \widetilde{\bm{\phi}}\|_{L^{4}(\Omega)}\\
&\ \ \ \!+\!
\|\Delta\widetilde{\bm{\phi}}\|_{L^{2}(\Omega)}
\|\mathbf{u}\|_{L^{4}(\Omega)}
\|\nabla \widetilde{\bm{\phi}}\|_{L^{4}(\Omega)}
\nonumber\\
&\ \ \ +
\|\Delta\widetilde{\bm{\phi}}\|_{L^{2}(\Omega)}
(\|\bm{\omega}\|_{L^{4}(\Omega)}
\|\nabla \mathbf{d}\|_{L^{4}(\Omega)}\\
&\ \ \ +
\|\nabla \mathbf{d}\|_{L^{8}(\Omega)}^{2}
\|\widetilde{\bm{\phi}}\|_{L^{4}(\Omega)}
+
\|\nabla \mathbf{d}\|_{L^{4}(\Omega)}
\|\nabla \widetilde{\bm{\phi}}\|_{L^{4}(\Omega)})
\nonumber\\
&\ \ \ +
\|\nabla \bm{\omega}\|_{L^{2}(\Omega)}
\|\nabla \bm{\xi}_{P}\|_{L^{4}(\Omega)}
\|\nabla \mathbf{d}\|_{L^{4}(\Omega)}\\
&\ \ \ +
\|\Delta\widetilde{\bm{\phi}}\|_{L^{2}(\Omega)}
\big(\|\mathbf{u}\|_{L^{4}(\Omega)}
\|\nabla \bm{\xi}_{P}\|_{L^{4}(\Omega)}+
\nonumber\\
&\ \ \ +
\|\nabla \mathbf{d}\|_{L^{8}(\Omega)}^{2}
\|\bm{\xi}_{P}\|_{L^{4}(\Omega)}
+
\|\nabla \mathbf{d}\|_{L^{4}(\Omega)}
\|\nabla \bm{\xi}_{P}\|_{L^{4}(\Omega)}
\big)\Big]
\nonumber\\
&\leq C\Big[
\|\nabla \mathbf{u}\|_{L^{2}(\Omega)}
\|\bm{\omega}\|_{L^{2}(\Omega)}\|\nabla\bm{\omega}\|_{L^{2}(\Omega)}\\
&\ \ \ \!+\!
\|\nabla \bm{\omega}\|_{L^{2}(\Omega)}
\|\nabla\mathbf{d}\|_{L^{4}(\Omega)}
\|\nabla \widetilde{\bm{\phi}}\|_{L^{2}(\Omega)}^\frac12
\|\Delta \widetilde{\bm{\phi}}\|_{L^{2}(\Omega)}^\frac12
\nonumber\\
&\ \ \ +\!
\|\Delta\widetilde{\bm{\phi}}\|_{L^{2}(\Omega)}^{\frac{3}{2}}
\|{\bf u}\|_{L^{4}(\Omega)}
\|\nabla \widetilde{\bm{\phi}}\|_{L^{2}(\Omega)}^{\frac{1}{2}}\\
&\ \ \ +
\|\Delta\widetilde{\bm{\phi}}\|_{L^{2}(\Omega)}
\|{\bm\omega}\|_{L^2(\Omega)}^\frac12\|\nabla\bm{\omega}\|_{L^{2}(\Omega)}^\frac12
\|\nabla\mathbf{d}\|_{L^{4}(\Omega)}
\nonumber\\
&\ \ \ +
\|\Delta\widetilde{\bm{\phi}}\|_{L^{2}(\Omega)}
\|\nabla\mathbf{d}\|_{L^{8}(\Omega)}^2
\|\nabla\widetilde{\bm{\phi}}\|_{L^{2}(\Omega)}
\\
&\ \ \ +
\|\Delta\widetilde{\bm{\phi}}\|_{L^{2}(\Omega)}^{\frac{3}{2}}
\|\nabla\mathbf{d}\|_{L^{4}(\Omega)}
\|\nabla \widetilde{\bm{\phi}}\|_{L^{2}(\Omega)}^{\frac{1}{2}}
\nonumber\\
&\ \ \ +
\|\nabla \bm{\omega}\|_{L^{2}(\Omega)}
\| \nabla\bm{\xi}_{P}\|_{L^{4}(\Omega)}
\|\nabla\mathbf{d}\|_{L^{4}(\Omega)}\\
&\ \ \ +
\|\Delta\widetilde{\bm{\phi}}\|_{L^{2}(\Omega)}
\big(\|\mathbf{u}\|_{L^{4}(\Omega)}
\| \nabla\bm{\xi}_{P}\|_{L^{4}(\Omega)}
\nonumber\\
&\ \ \ +\|\nabla\mathbf{d}\|_{L^{8}(\Omega)}^{2}
\|\bm{\xi}_{P}\|_{L^4(\Omega)}+
\| \nabla\mathbf{d}\|_{L^{4}(\Omega)}
\| \nabla\bm{\xi}_{P}\|_{L^{4}(\Omega)}
\big)\Big]
\nonumber\\
&\leq
\frac{1}{2}(\|\nabla\bm{\omega}\|_{L^{2}(\Omega)}^{2}
\!+\!\|\Delta\widetilde{\bm{\phi}}\|_{L^{2}(\Omega)}^{2})\nonumber\\
&\ \ \ +\!C(\|\bm{\omega}\|_{L^{2}(\Omega)}^{2}
\!+\!\|\nabla \widetilde{\bm{\phi}}\|_{L^{2}(\Omega)}^{2})\cdot\Big(\|{\bf u}\|_{L^4(\Omega)}^4+
\| \nabla\mathbf{u}\|_{L^{2}(\Omega)}^{2}
+\|\nabla\mathbf{d}\|_{L^{8}(\Omega)}^{4}\!\Big)\\
&\ \ \ \ \!+\!
C\Big[\!
\| \nabla\bm{\xi}_{P}\|_{L^{4}(\Omega)}^{2}
(\|\nabla\mathbf{d}\|_{L^{4}(\Omega)}^{2}
\!+\!
\|\mathbf{u}\|_{L^{4}(\Omega)}^{2})
\!+\!
\|\nabla\mathbf{d}\|_{L^{8}(\Omega)}^{4}
\|\bm{\xi}_{P}\|_{L^{4}(\Omega)}^{2}
\Big].
\end{align*}
This implies that
\begin{align*}
&\frac{d}{dt}(\|\bm{\omega}\|_{L^{2}(\Omega)}^{2}
+\|\nabla\widetilde{\bm{\phi}}\|_{L^{2}(\Omega)}^{2})
+(\|\nabla\bm{\omega}\|_{L^{2}(\Omega)}^{2}
+\|\Delta \widetilde{\bm{\phi}}\|_{L^{2}(\Omega)}^{2})
\nonumber\\
\leq&
C\big(\|{\bf u}\|_{L^4(\Omega)}^4+
\| \nabla\mathbf{u}\|_{L^{2}(\Omega)}^{2}
\!+\!
\|\nabla\mathbf{d}\|_{L^{8}(\Omega)}^{4}\!\big)
(\|\bm{\omega}\|_{L^{2}(\Omega)}^{2}
\!+\!\|\nabla \widetilde{\bm{\phi}}\|_{L^{2}(\Omega)}^{2})\nonumber\\
+&\!
C\Big[\!
\|\bm{\xi}_{P}\|_{H^{2}(\Omega)}^{2}
(\|\mathbf{d}\|_{H^{2}(\Omega)}^{2}
\!+\!
\|\mathbf{u}\|_{H^{1}(\Omega)}^{2})
\!+\!
\|\bm{\xi}_{P}\|_{H^{1}(\Omega)}^{2}\|\mathbf{d}\|_{H^{2}(\Omega)}^{4}\Big].
\end{align*}
Since $(\mathbf{u},\mathbf{d})$ is a strong solution obtained by Theorem \ref{gsolu},
it follows from \eqref{globstrong}  that
$$\|{\bf u}\|_{L^\infty_tH^1_x(Q_T)}+\|{\bf d}\|_{L^\infty_tH^2_x(Q_T)}
\le C_T,$$
where $C_T>0$ depends on $T, \Omega, \|{\bf u}_0\|_{H^1(\Omega)}$, 
$\|{\bf d}_0\|_{H^2(\Omega)}$,
$\|{\bf h}\|_{H^{\frac52,\frac54}(\Gamma_T)}$, and $\|\partial_t{\bf h}\|_{L^2_tH^\frac32_x(\Gamma_T)}$.
Hence by Sobolev's embedding theorem we have that
$$
\|{\bf u}\|_{L^\infty_tL^4_x(Q_T)}+\|\nabla {\bf d}\|_{L^\infty_tL^8_x(Q_T)}\le C_T.
$$
Thus we obtain that
$$\int_0^T\big(\|{\bf u}\|_{L^4(\Omega)}^4+
\| \nabla\mathbf{u}\|_{L^{2}(\Omega)}^{2}
\!+\!
\|\nabla\mathbf{d}\|_{L^{8}(\Omega)}^{4}
\!\big)\,dt\le C_T.$$
Since $(\bm\omega, \bm\phi)|_{t=0}=(0,0)$,  
by applying Gronwall's inequality we obtain that
\begin{align} \label{line3}
&\sup_{0\le t\le T}\|(\bm{\omega}, \nabla{\widetilde{\bm\phi}})(t)\|_{L^{2}(\Omega)}^{2}
+\int_{0}^{T}(\|\nabla\bm{\omega}(\tau)\|_{L^{2}(\Omega)}^{2}
+\|\nabla^2 \widetilde{\bm{\phi}}(\tau)\|_{L^{2}(\Omega)}^{2})\,d\tau
\nonumber\\
&\leq C\exp\{C\!\!\int_{0}^{T}
\big(\|{\bf u}\|_{L^4(\Omega)}^4+\| \nabla\mathbf{u}\|_{L^{2}(\Omega)}^{2}
+\|\nabla\mathbf{d}\|_{L^{8}(\Omega)}^{4}\big)(\tau)\,d\tau\}
\nonumber\\
&\ \ \ \cdot
\int_{0}^{T}
\Big[\!
\|\bm{\xi}_{P}\|_{H^{2}(\Omega)}^{2}
(\|\mathbf{d}\|_{H^{2}(\Omega)}^{2}
\!+\!
\|\mathbf{u}\|_{H^{1}(\Omega)}^{2})
\!+\!
\|\bm{\xi}_{P}\|_{H^{1}(\Omega)}^{2}\|\mathbf{d}\|_{H^{2}(\Omega)}^{4}\Big]\,d\tau
\nonumber\\
&\leq C_{T}\|\bm{\xi}\|_{\mathcal{U}}^2. 
\end{align}
Thus  \eqref{line1} was proven.
\medskip

To show that $(\bm\omega, \bm\phi)$ is a section of 
$(\mathbf{u},\mathbf{d})^*T\mathcal{H}$, we need to verify that
$$\langle\bm\phi,{\bf d}\rangle(x,t)=0, \ {\rm{for}}\ (x,t)\in Q_T.$$
To see this, observe that 
 $\langle\bm{\phi}, \mathbf{d}\rangle$ satisfies
\begin{align} \label{bmphi}
&\partial_{t}\langle\bm{\phi}, \mathbf{d}\rangle+\mathbf{u}\cdot\nabla \langle\bm{\phi}, \mathbf{d}\rangle
-\Delta \langle\bm{\phi}, \mathbf{d}\rangle\nonumber\\
&= \langle\partial_{t}\bm{\phi}+\mathbf{u}\cdot\nabla\bm{\phi}-\Delta \bm{\phi}, \mathbf{d}\rangle
+\langle\partial_{t}\mathbf{d}+\mathbf{u}\cdot\nabla \mathbf{d}-\Delta\mathbf{d},  \bm{\phi}\rangle
\nonumber\\
&\ \ \ -2\langle\nabla \bm{\phi}, \nabla\mathbf{d}\rangle
\nonumber\\
&=2|\nabla \mathbf{d}|^{2}\langle\bm{\phi}, \mathbf{d}\rangle,
\end{align}
and 
\begin{align*}
\langle\bm{\phi}, \mathbf{d}\rangle=0\ \ {\rm{on}}\ \ \partial_pQ_T.
\end{align*}
Hence, by the parabolic maximum principle, we  conclude that
\begin{align*}
\langle\bm{\phi}, \mathbf{d}\rangle\equiv 0\ \ \ {\rm{in}}\ \ \ Q_T.
\end{align*}
This completes the proof of Theorem \ref{Line}. \qed

\medskip
To facilitate the discussion, we also introduce a linear map associated with an element 
${\bf h}\in \widetilde{\mathcal{U}}$, $\mathcal{L}_{\mathbf{h}}: {\bf h}^*T\widetilde{\mathcal{U}}\mapsto 
(\mathcal{S}({\bf h}))^*T\mathcal{H}$ that is defined by
\begin{align}
\mathcal{L}_{\mathbf{h}}(\bm{\xi})=(\bm{\omega},\bm{\phi}),
\label{Lh*}
\end{align}
where $(\bm{\omega},\bm{\phi})$ is the unique global strong solution to the linearized system
\eqref{NLC-L1}--\eqref{NLC-L2} on $Q_T$, with $({\bf u}, {\bf d})=\mathcal{S}({\bf h})$,  obtained
by Theorem \ref{Line}.
\medskip

It follows directly from the estimate \eqref{line1} that

\begin{corollary}
For any ${\bf h}\in \widetilde{\mathcal{U}}$, the linear map $\mathcal{L}_{\mathbf{h}}:{\bf h}^*T\widetilde{\mathcal{U}}
\mapsto (\mathcal{S}({\bf h}))^*T\mathcal{H}$ is Lipschitz continuous.
\end{corollary}

\subsubsection{Differentiability of  $\mathcal{S}$}

In this subsection, we will prove the Fr\'echet differentiability of $\mathcal{S}$. More precisely we have
\begin{theorem} \label{F-Scal}
Given $(\mathbf{u}_{0}, \mathbf{d}_{0})\in \mathbf{V}\times H^{2}(\Omega, \mathbb{S}_{+}^{2})$,
if $\mathbf{h}\in \widetilde{\mathcal{U}}$ 
then the control to state map $\mathcal{S}$ is Fr\'echet differentiable at ${\bf h}$ in the sense of \eqref{Dif-F}. 
Moreover, the Fr\'echet derivative $\mathcal{S}'(\mathbf{h})$ is given by
 \begin{align}
{\mathcal{S}}'({\bf h})({\bm\xi})=\mathcal{L}_{\bf h}({\bm\xi}), \ {\rm{for\ any\ section}}\ {\bm\xi} 
\ {\rm{of}}\ {\bf h}^*T\widetilde{\mathcal{U}}\ {\rm{with}}\ \exp_{\bf h}{\bm\xi}\in\widetilde{\mathcal{U}}.
 \end{align}
\end{theorem}

\begin{proof}
Let $(\mathbf{u},\mathbf{d})$ be the unique global strong solution to the system
\eqref{eq1.1}--\eqref{eq1.3}, obtained by Theorem \ref{gsolu},
with the initial data $(\mathbf{u}_{0},\mathbf{d}_{0})$ and the boundary data 
$({0},\mathbf{h})$, namely,
\begin{align*}
(\mathbf{u},\mathbf{d})=\mathcal{S}(\mathbf{h}). 
\end{align*}
If ${\bm\xi}$ is a section of $\mathbf{h}^{*} T\widetilde{\mathcal{U}}$ such that
$\exp_{\bf h}{\bm\xi}\in\widetilde{\mathcal{U}}$, then we can define
a new boundary data $\widehat{\bf h}=\exp_{\bf h}{\bm\xi}$, which satisfies $\widehat{\bf h}(\Gamma_T)\subset
\mathbb{S}^{2}_+$ and belongs to ${\widetilde{\mathcal{U}}}$.
Let $(\widehat{\mathbf{u}},\widehat{\mathbf{d}})\in\mathcal{H}$ be the unique global strong solution to 
the problem \eqref{eq1.1}--\eqref{eq1.3} under the initial condition
$(\mathbf{u}_{0},\mathbf{d}_{0})$ and the boundary condition $({0},\widehat{\mathbf{h}})$, i.e.
$(\widehat{\mathbf{u}},\widehat{\mathbf{d}})=\mathcal{S}(\widehat{\mathbf{h}}).$

Let $(\bm{\omega},\bm{\phi})=\mathcal{L}_{\bf h}(\bm\xi)\in {\mathcal{S}}(\bf h)^*T\mathcal{H}$, which
is the unique solution to problem \eqref{NLC-L1}--\eqref{NLC-L2} obtained by Theorem \ref{Line},
under the initial condition $({0},{0})$ and the boundary condition
$({0},\bm{\xi})$.

By Theorem \ref{gsolu} and Theorem \ref{Line},  we have the following estimates:
\begin{align} \label{Unif-bdd}
\begin{cases}
\big\|(\mathbf{u}, \mathbf{d})\big\|_{\mathcal{H}}\leq C\big(T, \|({\bf u}_0, \nabla{\bf d}_0)\|_{H^1(\Omega)},
\|{\bf h}\|_{\mathcal{U}}\big),\\
\big\|(\widehat{\mathbf{u}}, \widehat{\mathbf{d}})\big\|_{\mathcal{H}}
\leq C\big(T, \|({\bf u}_0, \nabla{\bf d}_0)\|_{H^1(\Omega)},
\|{\bf h}\|_{\mathcal{U}}\big) ,\\
\big\|(\bm{\omega}, \bm{\phi})\big\|_{\mathcal{W}}\leq  C\big(T, \|({\bf u}_0, \nabla{\bf d}_0)\|_{H^1(\Omega)},
\|{\bf h}\|_{\mathcal{U}}\big)\big\|{\bm\xi}\big\|_{\mathcal{U}}.
\end{cases}
\end{align}
Moreover, we can infer from Theorem \ref{gsoluC} that
\begin{align}\label{conti1}
&\big\|\mathbf{u}-\widehat{\mathbf{u}}\big\|_{L^\infty_tH^{1}_x(Q_T)}^{2}
+\big\|\mathbf{d}-\widehat{\mathbf{d}}\big\|_{L^\infty_t H^{2}_x(Q_T)}^{2}\nonumber\\
+&\big\|\mathbf{u}-\widehat{\mathbf{u}}\big\|_{L^{2}_tH^{2}_x(Q_T)}^{2}
+\big\|\mathbf{d}-\widehat{\mathbf{d}}\big\|_{L^{2}_t H^{3}_x(Q_T)}^2 \nonumber\\
\leq &C\big(T, \|({\bf u}_0, \nabla{\bf d}_0)\|_{H^1(\Omega)}, \|{\bf h}\|_{\mathcal{U}}\big)
\big\|\mathbf{h}-\widehat{\mathbf{h}}\big\|_{\mathcal{U}}^{2}\nonumber\\
\leq &C\big(T, \|({\bf u}_0, \nabla{\bf d}_0)\|_{H^1(\Omega)}, \|{\bf h}\|_{\mathcal{U}}\big)
\big\|\bm\xi\big\|_{\mathcal{U}}^{2}.
\end{align}

Now we set 
$$\mathbf{w}=\widehat{\mathbf{u}}-\mathbf{u}-\bm{\omega}
\  {\rm{and}}\ \mathbf{e}=\widehat{\mathbf{d}}-\mathbf{d}-\bm{\phi}.$$
By direct calculations, $({\bf w}, {\bf e})$ solves, in $Q_T$, 
\begin{align} \label{NLC-WE}
\begin{cases}
\partial_{t}\mathbf{w}-\Delta\mathbf{w}+\nabla \widetilde{P} 
+(\widehat{\mathbf{u}}-\mathbf{u})\cdot\nabla (\widehat{\mathbf{u}}-\mathbf{u})+ 
(\mathbf{u}\cdot\nabla)\mathbf{w} + (\mathbf{w}\cdot\nabla) \mathbf{u}
\\
\ \ \ \ \ \ \ =-\nabla \cdot[\nabla (\widehat{\mathbf{d}}-\mathbf{d})\odot\nabla(\widehat{\mathbf{d}}-\mathbf{d})
+\nabla \mathbf{d}\odot\nabla \mathbf{e}+\nabla \mathbf{e}\odot\nabla \mathbf{d}],
 \\
 \nabla \cdot \mathbf{w} = 0,\\
 \partial_{t}\mathbf{e}-\Delta \mathbf{e}+(\widehat{\mathbf{u}}-\mathbf{u})\cdot\nabla(\widehat{\mathbf{d}}-\mathbf{d})
 +\mathbf{u}\cdot\nabla\mathbf{e}+\mathbf{w}\cdot\nabla\mathbf{d}\\
=|\nabla\mathbf{d}|^{2}\mathbf{e}
 +|\nabla(\widehat{\mathbf{d}}-\mathbf{d})|^{2}\widehat{\mathbf{d}}
 +2\langle\nabla\mathbf{d},\nabla\mathbf{e} \rangle
 \widehat{\mathbf{d}}+2\langle \nabla\mathbf{d},\nabla\bm{\phi}\rangle(\widehat{\mathbf{d}}-\mathbf{d}),
\end{cases}
\end{align}
with the boundary and initial  condition
\begin{align} \label{NLC-WE1}
\begin{cases}
(\mathbf{w},\mathbf{e})= ({0}, \exp_{\mathbf{h}}{\bm{\xi}}
-\mathbf{h}-\bm{\xi})\quad &\text{ on } \Gamma_T\\
(\mathbf{w},\mathbf{e})=({0},{0})
\quad &\text{ in }\Omega\times\{0\}.
\end{cases}
\end{align}
Define  the parabolic lifting function $\bm{\chi}:Q_T \mapsto \mathbb{R}^{3}$ by
\begin{align*}
\begin{cases}
\partial_{t}\bm{\chi}-\Delta \bm{\chi}=\mathbf{0}
 \quad &\text{ in } Q_T,\\
\bm{\chi}=\exp_{\mathbf{h}}{\bm{\xi}}
-\mathbf{h}-\bm{\xi}\quad &\text{ on } \Gamma_T,\\
\bm{\chi} = \mathbf{0}\quad &\text{ in } \Omega\times\{0\}.
\end{cases}
\end{align*}
By direct calculations, we find that 
$$
\big\|\exp_{\mathbf{h}}{\bm{\xi}}
-\mathbf{h}-\bm{\xi}\big\|_{\mathcal{U}}\le C\big\|{\bm\xi}\big\|_{\mathcal{U}}^2
$$
and hence 
\begin{align}\label{chi}
\|\nabla\bm\chi\|_{L^\infty_tL^2_x(Q_T)}
+\|\nabla^2\bm\chi\|_{L^2(Q_T)}+\|\bm\chi\|_{H^{3, \frac32}(Q_T)}
\le C\|\bm{\xi}\|_{\mathcal{U}}^2.
\end{align}

Next we define $\widetilde{\mathbf{e}}=\mathbf{e}-\bm{\chi}$. Then $(\mathbf{w},\widetilde{\mathbf{e}})$ satisfies
in $Q_T$:
\begin{align} \label{NLC-WE3}
\begin{cases}
\partial_{t}\mathbf{w}-\Delta\mathbf{w}+\nabla \widetilde{P} +
\overline{\mathbf{u}}\cdot\nabla \overline{\mathbf{u}}+ 
(\mathbf{u}\cdot\nabla)\mathbf{w} + (\mathbf{w}\cdot\nabla) \mathbf{u}
\\
=-\nabla \cdot[\nabla \overline{\mathbf{d}}\odot\nabla \overline{\mathbf{d}} 
+\nabla \mathbf{d}\odot\nabla \mathbf{e}
 +\nabla \mathbf{e}\odot\nabla \mathbf{d}],\\
 \nabla \cdot \mathbf{w} = 0,\\
 \partial_{t}\widetilde{\mathbf{e}}-\Delta \widetilde{\mathbf{e}}+\overline{\mathbf{u}}\cdot\nabla\overline{\mathbf{d}}
 +\mathbf{u}\cdot\nabla\mathbf{e}+\mathbf{w}\cdot\nabla \mathbf{d}\\
 =|\nabla\mathbf{d}|^{2}\mathbf{e}
 +|\nabla\overline{\mathbf{d}}|^{2}\widehat{\mathbf{d}}
 +2\langle\nabla\mathbf{d},\nabla\mathbf{e} \rangle
 \mathbf{d}
 +
 2\langle \nabla\mathbf{d},\nabla\bm\phi\rangle\overline{\mathbf{d}},
\end{cases}
\end{align}
with the boundary and initial condition
\begin{align} \label{NLC-WE4}
\begin{cases}
(\mathbf{w},\widetilde{\mathbf{e}})= ({0}, {0})\quad &\text{ on } \Gamma_T\\
(\mathbf{w},\widetilde{\mathbf{e}})=({0},{0})
\quad &\text{ in }\Omega\times\{0\}.
\end{cases}
\end{align}
Here $\overline{\mathbf{u}}=\widehat{\mathbf{u}}-\mathbf{u}$ and 
$\overline{\mathbf{d}}=\widehat{\mathbf{d}}-\mathbf{d}$. 

Multiplying $\eqref{NLC-WE3}_{1}$ by $\mathbf{w}$, and $\eqref{NLC-WE3}_{3}$ by $-\Delta \widetilde{\mathbf{e}}$,  integrating over $\Omega$, and adding the two resulting equations, we deduce
\begin{align}\label{deri-est1}
&\frac{1}{2}\frac{d}{dt}(\|\mathbf{w}\|_{L^{2}(\Omega)}^{2}
+\|\nabla \widetilde{\mathbf{e}}\|_{L^{2}(\Omega)}^{2})
+(\|\nabla \mathbf{w}\|_{L^{2}(\Omega)}^{2}
+\|\Delta\widetilde{\mathbf{e}}\|_{L^{2}(\Omega)}^{2})
\nonumber\\
&=
\Big[\int_{\Omega} (\overline{\mathbf{u}}\cdot\nabla\mathbf{w}\cdot\overline{\mathbf{u}}
+\mathbf{w}\cdot\nabla\mathbf{w}\cdot\mathbf{u})\\
&\ \ \ +\!\int_{\Omega} [\nabla \overline{\mathbf{d}}\odot\nabla \overline{\mathbf{d}} 
+\nabla \mathbf{d}\odot\nabla \widetilde{\mathbf{e}}
 +\widetilde{\nabla \mathbf{e}}\odot\nabla \mathbf{d}]:\nabla\mathbf{w}
\nonumber\\
&\ \ \ +\int_\Omega[\nabla \mathbf{d}\odot\nabla {\bm\chi}
 +\nabla \bm\chi\odot\nabla \mathbf{d}]:\nabla\mathbf{w}
\nonumber\\
&\ \ \ +\!\!\int_{\Omega} \!(\overline{\mathbf{u}}\cdot\!\nabla\overline{\mathbf{d}}
+\mathbf{u}\!\cdot\!\nabla\widetilde{\mathbf{e}}
+\mathbf{w}\!\cdot\!\nabla \mathbf{d})
\!\cdot\!\Delta\widetilde{\mathbf{e}}\nonumber\\
&\ \ \ -\!\!\int_{\Omega}\!
(|\nabla\mathbf{d}|^{2}\widetilde{\mathbf{e}}
 \!+\!|\nabla\overline{\mathbf{d}}|^{2}\widehat{\mathbf{d}}
 \!+\!2\langle\nabla\mathbf{d},\nabla\widetilde{\mathbf{e}}\rangle
 \mathbf{d}
 \!+\!
 2\langle \nabla\mathbf{d},\nabla{\bm\phi}\rangle\overline{\mathbf{d}})
 \!\cdot\!\Delta\widetilde{\mathbf{e}}
 \nonumber\\
 &\ \ \ -\int_\Omega (|\nabla{\bf d}|^2\bm\chi+2\langle\nabla{\bf d}, \nabla\bm\chi\rangle{\bf d}
 -{\bf u}\cdot\nabla\bm\chi)\Delta\widetilde{\bm e}\Big]
 \nonumber\\
  &=I+II+III+IV+V+VI.\nonumber
 \end{align}
$I,\cdots, VI$ can be estimated as follows.
 $$|I|\le \frac1{12}\|\nabla{\bf w}\|_{L^2(\Omega)}^2+C\|{\bf u}\|_{L^4(\Omega)}^4\|{\bf w}\|_{L^2(\Omega)}^2
 +C\|\overline{\bf u}\|_{L^4(\Omega)}^4,$$
 \begin{align*}
 |II|&\le \frac1{12}(\|\nabla{\bf w}\|_{L^2(\Omega)}^2+\|\Delta{\widetilde{\bf e}}\|_{L^2(\Omega)}^2)\\
 &\ \ +C\|\nabla{\bf d}\|_{L^4(\Omega)}^4\|\nabla\widetilde{\bf e}\|_{L^2(\Omega)}^2
 +C\|\nabla\overline{\bf d}\|_{L^4(\Omega)}^4,
 \end{align*}
 \begin{align*}
 |III|\leq \frac12\|\nabla{\bf w}\|_{L^2(\Omega)}^2
 +C\|\nabla{\bf d}\|_{L^4(\Omega)}^2\|\nabla{\bm\chi}\|_{L^4(\Omega)}^2,
 \end{align*}
\begin{align*}
|IV|&\le \frac1{12}(\|\nabla{\bf w}\|_{L^2(\Omega)}^2+\|\Delta{\widetilde{\bf e}}\|_{L^2(\Omega)}^2)\\
&\ +C(\|\overline{\bf u }\|_{L^4(\Omega)}^4+\|\nabla\overline{\bf d}\|_{L^4(\Omega)}^4)\\
&\ +C(\|\nabla{\bf d}\|_{L^4(\Omega)}^4\|{\bf w}\|_{L^2(\Omega)}^2+\|{\bf u}\|_{L^4(\Omega)}^4\|\nabla\widetilde{\bf e}\|_{L^2(\Omega)}^2),
\end{align*}
\begin{align*}
|V|&\le \frac1{12}\|\Delta{\widetilde{\bf e}}\|_{L^2(\Omega)}^2+ C\|\nabla{\bf d}\|_{L^8(\Omega)}^4\|\nabla\widetilde{\bf e}\|_{L^2(\Omega)}^2\\
&\ +C\|\nabla\overline{\bf d}\|_{L^4(\Omega)}^4
+C\|\nabla{\bf d}\|_{L^8(\Omega)}^2\|\nabla{\bm{\phi}}\|_{L^4(\Omega)}^2\|\overline{\bf d}\|_{L^8(\Omega)}^2,
\end{align*}
and
\begin{align*}
|VI|&\le \frac12\|\Delta\widetilde{\bf e}\|_{L^2(\Omega)}^2+C\|\nabla{\bf d}\|_{L^8(\Omega)}^4\|\bm\chi\|_{L^4(\Omega)}^2\\
&\ +C(\|\nabla{\bf d}\|_{L^4(\Omega)}^2+\|{\bf u}\|_{L^4(\Omega)}^2)\|\nabla\bm\chi\|_{L^4(\Omega)}^2.
\end{align*}
Substituting these estimates into \eqref{deri-est1}, we obtain
\begin{align*}
&\frac{d}{dt}(\|\mathbf{w}\|_{L^{2}(\Omega)}^{2}
+\|\nabla \widetilde{\mathbf{e}}\|_{L^{2}(\Omega)}^{2})
+(\|\nabla \mathbf{w}\|_{L^{2}(\Omega)}^{2}
+\|\Delta\widetilde{\mathbf{e}}\|_{L^{2}(\Omega)}^{2})\\
&\le C (\|{\bf u}\|_{L^4(\Omega)}^4+ \|\nabla{\bf d}\|_{L^8(\Omega)}^4)(\|\mathbf{w}\|_{L^{2}(\Omega)}^{2}
+\|\nabla \widetilde{\mathbf{e}}\|_{L^{2}(\Omega)}^{2})\\
& \ +C(\|\overline{\bf u}\|_{L^4(\Omega)}^4+ \|\nabla\overline{\bf d}\|_{L^4(\Omega)}^4)+
\|\nabla{\bf d}\|_{L^8(\Omega)}^2\|\nabla{\bm\phi}\|_{L^4(\Omega)}^2\|\overline{\bf d}\|_{L^8(\Omega)}^2)\\
& \ +C\big[\|\nabla{\bf d}\|_{L^8(\Omega)}^4\|\bm\chi\|_{L^4(\Omega)}^2
+(\|\nabla{\bf d}\|_{L^4(\Omega)}^2+\|{\bf u}\|_{L^4(\Omega)}^2)\|\nabla\bm\chi\|_{L^4(\Omega)}^2\big].
\end{align*}
From  \eqref{Unif-bdd}, \eqref{conti1} and \eqref{chi}, we have that
\begin{align*}
\int_0^T\|\nabla{\bf d}\|_{L^8(\Omega)}^4\,dt
&\le C\|\nabla{\bf d}\|_{L^4(Q_T)}^3\|\nabla^2{\bf d}\|_{L^4(Q_T)}\\
&\le C\|\nabla{\bf d}\|_{L^4(Q_T)}^3\|\nabla^2{\bf d}\|_{L^\infty_tL^2_x(Q_T)}^\frac12
\|{\bf d}\|_{L^2_tH^3_x(Q_T)}^\frac12\\
&\le C_T,
\end{align*}
$$
\int_0^T\|{\bf u}\|_{L^4(\Omega)}^4\,dt \le \|{\bf u}\|_{L^4(Q_T)}^4\le C_T,
$$
$$
\int_0^T(\|\overline{\bf u}\|_{L^4(\Omega)}^4+ \|\nabla\overline{\bf d}\|_{L^4(\Omega)}^4)\,dt
\le C_T\|{\bm\xi}\|_{\mathcal{U}}^4,
$$
\begin{align*}
&\int_0^T\|\nabla{\bf d}\|_{L^8(\Omega)}^2\|\nabla{\bm \phi}\|_{L^4(\Omega)}^2\|\overline{\bf d}\|_{L^8(\Omega)}^2\,dt\\
&\le \|\nabla{\bf d}\|_{L^4_tL^8_x(Q_T)}^2\|\nabla{\bm\phi}\|_{L^4(Q_T)}^2
\|\overline{\bf d}\|_{L^\infty_tL^8_x(Q_T)}^2\\
&\le  C_T \|{\bm\xi}\|_{\mathcal{U}}^4,
\end{align*}
and
\begin{align*}
&\int_0^T\big[\|\nabla{\bf d}\|_{L^8(\Omega)}^4\|\bm\chi\|_{L^4(\Omega)}^2
+(\|\nabla{\bf d}\|_{L^4(\Omega)}^2+\|{\bf u}\|_{L^4(\Omega)}^2)\|\nabla\bm\chi\|_{L^4(\Omega)}^2\big]\,dt\\
&\le C(\|\nabla{\bf d}\|_{L^4_tL^8_x(Q_T)}^4+\|\nabla{\bf d}\|_{L^4(Q_T)}^2
+\|{\bf u}\|_{L^4(Q_T)}^2)\|\bm\chi\|_{L^\infty_tH^2_x(Q_T)}^2\\
&\le C_T\|{\bm\xi}\|_{\mathcal{U}}^4.
\end{align*}
These estimates, combined with the fact that $({\bf w}, \widetilde{\bf e})=0$ at $t=0$ and Gronwall's
inequality, imply that
\begin{align*}
&\sup_{0\le t\le T}(\|\mathbf{w}\|_{L^{2}(\Omega)}^{2}
+\|\nabla \widetilde{\mathbf{e}}\|_{L^{2}(\Omega)}^{2})
+\int_0^T(\|\nabla \mathbf{w}\|_{L^{2}(\Omega)}^{2}
+\|\Delta\widetilde{\mathbf{e}}\|_{L^{2}(\Omega)}^{2})\,dt\\
&\le C_T \|{\bm\xi}\|_{\mathcal{U}}^4.
\end{align*}
This, with the help of \eqref{chi},  yields that
$$
\|({\bf w}, {\bf e})\|_{\mathcal{W}}\le C\|{\bm\xi}\|_{\mathcal{U}}^2.
$$
Hence $\mathcal{S}$ is differentiable at ${\bf h}$, and its Fr\'echet derivative
$\mathcal{S}'({\bf h})(\bm\xi)=\mathcal{L}_{\bf h}(\bm\xi)$ whenever $\bm\xi\in {\bf h}^*T\widetilde{\mathcal{U}}$
is such that $\exp_{\bf h}{\bm\xi}\in \widetilde{\mathcal{U}}$. 
This completes the proof.
\end{proof}

\subsection{Existence and necessary condition of boundary optimal control}
Here we will consider both the existence and a necessary condition of an optimal boundary control
for the problem \eqref{CF1}.

\subsubsection{The existence of an optimal boundary control}

We will establish the existence of an optimal boundary control for 
the problem (\ref{CF1}).  

\begin{theorem}\label{Exist-BCP}
Under the conditions (A1) and (A2), 
let $(\mathbf{u}_0, \mathbf{d}_0)\in \mathbf{V}\times H^{2}(\Omega,\mathbb{S}_{+}^{2})$ be given.
For $M>0$, if $\widetilde{\mathcal{U}}_{M}\not=\emptyset$, then (\ref{CF1}) admits a solution $((\mathbf{u},\mathbf{d}), \mathbf{h})$, 
where $\mathbf{h}\in \widetilde{\mathcal{U}}_{M}$
and $(\mathbf{u},\mathbf{d})=\mathcal{S}(\mathbf{h})$ is the unique strong solution to  \eqref{eq1.1}--\eqref{eq1.3}
with the initial condition $(\mathbf{u}_{0}, \mathbf{d}_{0})$ and the boundary condition $(0, {\bf h})$.
\end{theorem}

\proof Let  $\{((\mathbf{u}^{i}, \mathbf{d}^{i}), \mathbf{h}^{i})\}_{i=1}^{\infty}$ 
be a minimizing sequence of the cost functional $\mathcal{C}$ in \eqref{CF1}
over $\widetilde{\mathcal{U}}_{M}$, i.e.,
\begin{equation} \label{minJ}
\lim_{i\rightarrow \infty} \mathcal{C}(({\bf u}^i, {\bf d}^i), \mathbf{h}^{i})
=\inf_{\mathbf{h}\in \widetilde{\mathcal{U}}_{M}, 
({\bf u}, {\bf d})=\mathcal{S}({\bf h})}\mathcal{C}(({\bf u}, {\bf d}), \mathbf{h}),
\end{equation}
where $\mathbf{h}^{i}\in \widetilde{\mathcal{U}}_{M}$
and $(\mathbf{u}^{i},\mathbf{d}^{i})=\mathcal{S}(\mathbf{h}^{i})\in \mathcal{H}$ 
is the unique strong solution to the initial boundary value problem 
of \eqref{eq1.1}--\eqref{eq1.3} with the initial condition $(\mathbf{u}_0, \mathbf{d}_0)$ and
the boundary condition $(\mathbf{0},\mathbf{h}^{i})$. Then we have
$$\big\|\Phi({\bf h}^i)\big\|_{\mathcal{U}}\le M,
\ {\rm{and}}\ \big\|({\bf u}^i, {\bf d}^i)\big\|_{\mathcal{H}}\le CM, \ \forall\ i\ge 1.$$
From the weak compactness of $\widetilde{\mathcal{U}}_{M}\hookrightarrow\mathcal{U}$,
we may assume, after passing to subsequences, that there exist ${\bf h}^*\in \widetilde{\mathcal{U}}_{M}$
and $({\bf u}^*, {\bf d}^*)\in \mathcal{H}$ such that 
\begin{align*}
\mathbf{h}^{i} \rightharpoonup \mathbf{h}^* \quad\text{in\ }\ \  H^{\frac52, \frac54}(\Gamma_T),
\ \ \ \partial_t {\bf h}^i \rightharpoonup \partial_t\mathbf{h}^*\quad\text{in\ }\ \  L^2_tH^{\frac32}_x(\Gamma_T),
\end{align*}
and
\begin{align*}
&\mathbf{u}^{i}\stackrel{*}\rightharpoonup \mathbf{u}^*\ \ \ \text{ in }\ \ L^{\infty}([0, T], \mathbf{V}),\ 
\mathbf{u}^{i}\rightharpoonup \mathbf{u}^*\ \ \ \text{ in }\ \ \ L^{2}_tH^{2}_x(Q_T),\nonumber\\
&\mathbf{d}^{i}\stackrel{*}\rightharpoonup \mathbf{d}^* \ \ \ \text{ in }\ \ \
L^\infty_tH^2_x(Q_T),\ \mathbf{d}^{i}\rightharpoonup \mathbf{d}^*\ \ \ \text{ in }\ \ \ H^{3,\frac{3}{2}}(Q_T).
\end{align*}

Observe also that by the Aubin-Lions Lemma, 
\begin{align*}
\mathbf{u}^{i}\rightarrow \mathbf{u}^{*}\quad \text{ in }\ \ C([0, T], L^2(\Omega)),\ \ 
\mathbf{d}^{i}\rightarrow \mathbf{d}^{*}\quad \text{ in } \ \ C([0, T], H^{1}(\Omega)).
\end{align*}
It is not hard to see that $({\bf u}^*, {\bf d}^*)\in \mathcal{H}$ is a strong solution of
the system \eqref{eq1.1}--\eqref{eq1.3}, with the initial condition $({\bf u}_0, {\bf d}_0)$
and the boundary condition $(0, {\bf h}^*)$. By the uniqueness theorem of strong solutions,
we conclude that $({\bf u}^*, {\bf d}^*)=\mathcal{S}({\bf h}^*)$. 

Since the cost functional  $\mathcal{C}$ is weakly lower semi-continuous in $(({\bf u}, {\bf d}), {\bf h})
\in\mathcal{H}\times \mathcal{U}$, we have 
\begin{align}\label{minJ1}
 \liminf_{i\rightarrow\infty}\mathcal{C}((\mathbf{u}^{i},\mathbf{d}^{i}), \mathbf{h}^{i})
\geq \mathcal{C}((\mathbf{u}^{*}, \mathbf{d}^{*}),\mathbf{h}^{*}).
\end{align}
On the other hand, since ${\bf h}^*\in \widetilde{\mathcal{U}}_{M}$ and $({\bf u}^*, {\bf d}^*)=\mathcal{S}({\bf h}^*)$,
we also have
\begin{align}\label{minJ2}
\inf_{\mathbf{h}\in \widetilde{\mathcal{U}}_{M}, 
({\bf u}, {\bf d})=\mathcal{S}({\bf h})}\mathcal{C}(({\bf u}, {\bf d}), \mathbf{h})
\le \mathcal{C}((\mathbf{u}^{*}, \mathbf{d}^{*}),\mathbf{h}^{*}).
\end{align}
It follows directly from \eqref{minJ}, \eqref{minJ1}, and \eqref{minJ2}
that $(({\bf u}^*, {\bf d}^*), {\bf h}^*)$ achieves
$$
\inf_{\mathbf{h}\in \widetilde{\mathcal{U}}_{M}, 
({\bf u}, {\bf d})=\mathcal{S}({\bf h})}\mathcal{C}(({\bf u}, {\bf d}), \mathbf{h}).
$$
This completes the proof. \qed

\subsubsection{The first-order necessary optimality condition}
In this subsection, we will derive  the first-order necessary condition for the optimal control problem (\ref{CF1}) based on the Fr\'echet differentiability of the control-to-state operator $\mathcal{S}$ established
in the previous section.

Now we are ready to prove the following theorem that gives a necessary condition of boundary optimal control.
\begin{theorem}\label{thm-cond1}
Assume both (A1) and (A2). For $M>0$,  let $(\mathbf{u}_{0}, \mathbf{d}_{0})\in \mathbf{V}\times
H^{2}(\Omega,\mathbb{S}_{+}^{2})$ be given. If  $\widetilde{\mathcal{U}}_{M}\not=\emptyset$
and $\mathbf{h}\in\widetilde{\mathcal{U}}_{M}$ 
is a minimizer of the optimal control for problem (\ref{CF1}) over the admissible set
$\widetilde{\mathcal{U}}_{M}$, 
with the associated state map $(\mathbf{u},\mathbf{d})=\mathcal{S}(\mathbf{h})\in \mathcal{H}$.
Then for any boundary data ${\bf h}_{*}\in {\widetilde{\mathcal U}}_{M}$,
let $\bm\xi=\bm\xi_{{\bf h}, {\bf h}_*}$ be the section of
${\bf h}^*T{\widetilde{\mathcal{U}}}$ given by
\begin{align}\label{section}
{\bm\xi}=\frac{d}{ds}\big|_{s=0} \Pi^{-1}(s\Pi({\bf h})+(1-s)\Pi({\bf h}^*)),
\end{align}
and $(\bm{\omega}, \bm{\phi})=\mathcal{S}'(\mathbf{h})(\bm{\xi})$, i.e., the unique
global strong solution to the linearized problem \eqref{NLC-L1}--\eqref{NLC-L2} 
associated with ${\bm\xi}$,   the following variational inequality holds:
\begin{align}\label{cond1}
&\int_{Q_{T}} \big(\beta_1\langle\mathbf{u}-\mathbf{u}_{Q_T}, \bm{\omega}\rangle
  + \beta_{2} \langle\mathbf{d}-\mathbf{d}_{Q_T}, \bm{\phi}\rangle\big)\nonumber\\
 & + \int_\Omega(\beta_{3} \langle\mathbf{u}(T)-\mathbf{u}_{\Omega},
 \bm{\omega}(T)\rangle+\beta_{4} \langle\mathbf{d}(T)
    -\mathbf{d}_{\Omega}, \bm{\phi}(T)\rangle\big)\nonumber\\
& + \int_{\Gamma_{T}} \beta_5 \langle\mathbf{h}-\mathbf{e}_{3},
 \bm{\xi}\rangle  \geq 0.
\end{align}
\end{theorem}

\begin{proof} Note that $(({\bf u}, {\bf d}), {\bf h})=(\mathcal{S}({\bf h}), {\bf h})$
is a minimizer of ${\mathcal{C}}$ over $\widetilde{\mathcal{U}}_{M}$. 
For any ${\bf h}_*\in \widetilde{\mathcal{U}}_{M}$, let ${\bm\xi}$ 
be the section of ${\bf h}^*T\widetilde{\mathcal{U}}$ given by \eqref{section}.
Then  
$${\bf h}(s)=\Pi^{-1}(s\Pi({\bf h})+(1-s)\Pi({\bf h}^*))\in C^1([0,1], \widetilde{\mathcal{U}})$$
is a $C^1$-family of maps from $\Gamma_T$ to $\widetilde{\mathcal{U}}_M$ joining
${\bf h}$ to ${\bf h}^*$.
If we let 
$({\bf u}(s), {\bf d}(s))=\mathcal{S}({\bf h}(s))$ for $s\in [0,1]$.
Then it is not hard to verify that 
$({\bf u}(s), {\bf d}(s))\in C^1([0,1], \mathcal{H})$ and
$$
\frac{d}{ds}\big|_{s=0}({\bf u}(s), {\bf d}(s))={\mathcal{S}}'({\bf h})({\bm\xi})=(\bm\omega,\bm\phi)
\ {\rm{in}}\ Q_T.
$$
Since 
$$\mathcal{C}(({\bf u}(s), {\bf d}(s)), {\bf h}(s))\ge \mathcal{C}(({\bf u}(0), {\bf d}(0)), {\bf h}(0))
=\mathcal{C}(({\bf u}, {\bf d}), {\bf h}), \ \forall\ s\in [0,1],$$
we conclude that
$$\frac{d}{ds}\big|_{s=0} \mathcal{C}(({\bf u}(s), {\bf d}(s)), {\bf h}(s))\ge 0.$$
It follows directly from the chain rule that
\begin{align*}
&\frac{d}{ds}\big|_{s=0} \int_{Q_T}(\beta_1 |{\bf u}(s)-{\bf u}_{Q_T}|^2+\beta_2 |{\bf d}(s)-{\bf d}_{Q_T}|^2)\\
&=2\int_{Q_T}(\beta_1 \langle{\bf u}-{\bf u}_{Q_T}, \bm\omega\rangle+\beta_2 \langle{\bf d}-{\bf d}_{Q_T},
\bm\phi\rangle),
\end{align*}
\begin{align*}
&\frac{d}{ds}\big|_{s=0} \int_{\Omega}(\beta_3 |{\bf u}(T)-{\bf u}_\Omega|^2+\beta_4 |{\bf d}(T)-{\bf d}_\Omega|^2)\\
&=2\int_{\Omega}(\beta_3 \langle{\bf u}(T)-{\bf u}_\Omega, \bm\omega(T)\rangle+\beta_4 \langle{\bf d}(T)-{\bf d}_\Omega,
\bm\phi(T)\rangle),
\end{align*}
and
$$\frac{d}{ds}\big|_{s=0} \int_{\Gamma_T}\beta_5|{\bf h}(s)-{\bf e}_3|^2
=2\int_{\Gamma_T}\beta_5\langle{\bf h}(s)-{\bf e}_3, \bm\xi\rangle.
$$
Putting these together yields \eqref{cond1}. This completes the proof.
\end{proof}

\subsection{First-order necessary  condition via adjoint states}

In this subsection, we will eliminate the pair $(\bm{\omega}, \bm{\phi})$ from the variational inequality \eqref{cond1}
and derive a first-order necessary condition in terms of the optimal solution together with its adjoint states.
For this purpose, we will first derive the corresponding adjoint system of the control problem (\ref{CF1}). 
Since this section is similar to section 6 of \cite{CRW2017}, we will only sketch it here.

\subsubsection{Formal derivation of the adjoint system} 

The Lagrange functional $\mathcal{G}$ for the control problem (\ref{CF1}), with 
Lagrange multipliers ${\bf p}_1$, ${\bf p}_2$, ${\bf \pi}$, ${\bf q}_1$ and ${\bf q}_2$, is given by
\begin{align}\label{Gcal}
&\mathcal{G}((\mathbf{u}, \mathbf{d}), \mathbf{h}, {\bf p}_1, {\bf p}_2, {\bf \pi}, {\bf q}_1, {\bf q}_2)=\nonumber\\
&\mathcal{C}((\mathbf{u}, \mathbf{d}), \mathbf{h})
          - \int_{Q_{T}} \langle\partial_{t}\mathbf{u}+ \mathbf{u}\cdot\nabla \mathbf{u} -\Delta \mathbf{u}
          +\nabla P+ \nabla\cdot(\nabla \mathbf{d}\odot\nabla \mathbf{d}),  {\bf p}_1\rangle
\nonumber\\
 &-\!\!\int_{Q_{T}}\!\! (\nabla\!\cdot \mathbf{u}){\bf \pi}
  -\!\!\int_{Q_{T}} \!\!\langle\partial_{t}\mathbf{d} +\mathbf{u}\cdot\!\nabla \mathbf{d}-\Delta \mathbf{d}-|\nabla\mathbf{d}|^{2}\mathbf{d},  {\bf p}_2\rangle\nonumber\\
 &-\!\!\int_{\Gamma_{T}}\!\! \langle\mathbf{u}, {\bf q}_1\rangle
         -\!\!\int_{\Gamma_{T}}\!\! \langle\mathbf{d}\!-\!\mathbf{h},{\bf q}_2\rangle,
\end{align}
for any $\mathbf{h}\in \widetilde{\mathcal{U}}_{M}$
and $(\mathbf{u},\mathbf{d})\in \mathcal{H}$. 
Here, we will eliminate the five constraints due to the state problem \eqref{eq1.1}--\eqref{eq1.3} by
five corresponding Lagrange multipliers ${\bf p}_1, {\bf p}_2, {\bf \pi}, {\bf q}_1, {\bf q}_2$.

\medskip

For $M>0$, 
let $((\mathbf{u},\mathbf{d}), \mathbf{h})$ be a minimizer of the optimal control problem (\ref{CF1}) such that
$\mathbf{h}\in \widetilde{\mathcal{U}}_{M}$
and $(\mathbf{u},\mathbf{d})=\mathcal{S}(\mathbf{h})\in \mathcal{H}$.
Then we expect that $(\mathbf{u}, \mathbf{d})$ and $\mathbf{h}$ together with
the corresponding Lagrange multipliers ${\mathbf{p}}_1$, ${\mathbf{p}}_{2}$, ${\bf \pi}$, ${\mathbf{q}_1}$, 
${\mathbf{q}}_{2}$
satisfy the optimality conditions associated with the minimization problem for the Lagrange functional $\mathcal{G}$, i.e.,
\begin{align} \label{minG}
\min \mathcal{G}((\mathbf{u}, \mathbf{d}), \mathbf{h}, ({\mathbf{p}}_1, {\mathbf{p}}_{2}, {\bf \pi}, {\mathbf{q}}_1, 
{\mathbf{q}}_{2})), \ {\rm{with}}\ (\mathbf{u}, \mathbf{d})\ \textit{unconstrained}\ {\rm{and}}\ \mathbf{h} \in 
\widetilde{\mathcal{U}}_{M}.
\end{align}
Then we have that
\begin{align}\label{dG1}
\mathcal{G}'_{(\mathbf{u}, \mathbf{d})}((\mathbf{u}, \mathbf{d}), \mathbf{h},
({\mathbf{p}}_1, {\mathbf{p}}_{2}, {\bf \pi}, {\mathbf{q}}_1, {\mathbf{q}}_{2}))(\bm{\omega}, \bm{\phi})=0,
\end{align}
for all smooth functions $(\bm{\omega}, \bm{\phi})$ satisfying
\begin{align}\label{conss}
& \bm{\omega}|_{t=0}=\mathbf{0},\quad \bm{\phi}|_{t=0}=\mathbf{0}, \quad \text{in}\ \Omega.
\end{align}
Here \eqref{conss}  follows from the fact the initial data $(\mathbf{u}_{0},\mathbf{d}_{0})$
of (\ref{CF1}) is fixed.
 
Similar to  the derivation of \eqref{cond1}, it follows from \eqref{dG1} that
\begin{align} \label{adjoint1}
0&= \beta_{1}\!\int_{Q_{T}} \!\langle\mathbf{u} -\mathbf{u}_{Q_{T}}, \bm{\omega}\rangle
   + \beta_{2}\!\int_{Q_{T}}\! \langle\mathbf{d}-\mathbf{d}_{Q_T}, \bm{\phi}\rangle\nonumber\\
 &\ \ \  +\beta_{3}\!\int_{\Omega}\! \langle\mathbf{u}(T)-\mathbf{u}_{\Omega},\bm{\omega}(T)\rangle
  +\beta_{4}\!\int_{\Omega}\! \langle\mathbf{d}(T)-\mathbf{d}_{\Omega}, \bm{\phi}(T)\rangle \nonumber\\
&\ \ \  -\!\int_{Q_{T}}\! \langle\partial_{t}\bm{\omega}+\mathbf{u}\cdot\nabla\bm{\omega}-\Delta \bm{\omega}
+\nabla \widetilde{P}+ \bm{\omega}\cdot\nabla \mathbf{u}\nonumber\\
&\ \ \ \ \ \ \ \ \ \ \ \ -\nabla\!\cdot\!(\nabla \bm{\phi}\odot\nabla \mathbf{d}+\nabla \mathbf{d}\odot\nabla \bm{\phi}), {\mathbf{p}}_1\rangle
\nonumber\\
& \ \ \ \ -\!\int_{Q_{T}}\! (\nabla\cdot \bm{\omega}){\bf \pi}
  -\!\int_{Q_{T}}\! \langle\partial_{t} \bm{\phi}-\Delta \bm{\phi}+\mathbf{u}\cdot\nabla\bm{\phi}
  +\bm{\omega}\cdot\nabla\mathbf{d}\nonumber\\
&\ \ \ \ \ \ \ \ \ \ \ \ \ \ \ \ \ \ \ \ \ \ \ \ \ \ \ \ \ \ -|\nabla\mathbf{d}|^{2}\bm{\phi}-2\langle\nabla\mathbf{d},\nabla\bm{\phi}\rangle\mathbf{d}, {\mathbf{p}}_2\rangle
\nonumber\\
&\ \ \ \ -\!\int_{\Gamma_{T}} \langle\bm{\omega}, {\mathbf{q}}_{1}\rangle
 -\!\int_{\Gamma_{T}} \langle\bm{\phi}, {\mathbf{q}}_{2}\rangle.
\end{align}
Performing integration by parts, using the condition \eqref{conss}, and regrouping the relevant terms in the same
way as \cite{CRW2017} page 1065, we can obtain the adjoint system for
${\bf p}_1, {\bf p}_2, {\bf \pi}, {\bf q}_1,$ and ${\bf q}_2$ in $Q_T$:
\begin{align}\label{adjoint3}
\begin{cases}
\partial_{t} {\mathbf{p}}_1+\Delta {\mathbf{p}}_1
      +\nabla {\bf \pi }+\mathbf{u}\cdot \nabla{\mathbf{p}}_1
      -(\nabla \mathbf{u}){\mathbf{p}}_1
      -(\nabla\mathbf{d}){\mathbf{p}}_2\\
       =- \beta_{1} (\mathbf{u} -\mathbf{u}_{Q_{T}}),\\
\nabla \cdot {\mathbf{p}}_1=0,\\
\partial_{t} {\mathbf{p}}_2+\Delta {\mathbf{p}}_2
   +\mathbf{u}\cdot\nabla {\mathbf{p}}_2
   -\partial_{i}(\partial_{j}\mathbf{d}\partial_{j}{\bf p}_1^{i})
   -\partial_{j}(\partial_{i}\mathbf{d}\partial_{j}{\bf p}_1^{i})\\
=  -|\nabla\mathbf{d}|^{2}{\mathbf{p}}_2
   +2\nabla\cdot(\nabla\mathbf{d}(\mathbf{d}\cdot\mathbf{p}_2))
   -\beta_2(\mathbf{d}-\mathbf{d}_{Q_{T}}),
\end{cases}
\end{align}
with the following boundary and terminal conditions
\begin{align}\label{adjoint4}
\begin{cases}
{\mathbf{p}}_1={0}, \quad {\mathbf{p}}_2={0}\quad &\text{ on } \Gamma_{T},\\
{\mathbf{p}}_1|_{t=T}=\beta_3(\mathbf{u}(T)-\mathbf{u}_{\Omega}),
\quad {\mathbf{p}}_2|_{t=T}=\beta_4(\mathbf{d}(T)-\mathbf{d}_{\Omega})
\quad &\text{ in } \Omega.
\end{cases}
\end{align}
 Furthermore, the Lagrange multipliers $({\mathbf{q}}_{1},  {\mathbf{q}}_{2})$ can be
uniquely determined by $({\mathbf{p}}_1, {\bf\pi}, {\mathbf{p}}_2)$  through
\begin{align}\label{adjoint5}
\begin{cases}
{\mathbf{q}}_1+\partial_{\bm{\nu}}{\mathbf{p}}_1
+{\bf\pi}\bm{\nu}={0}\quad &\text{ on } \Gamma_{T},\\
{\bf q}_{2}^k+(\partial_{\bm{\nu}} {\bf p}_2)^{k}
 =\partial_{j}{\bf d}^{k}\partial_{j}{\bf p}_1^{i}{\bm\nu}^{i}
 +\partial_{i}{\bf d}^{k}\partial_{j}{\bf p}_1^{i}{\bm\nu}^{j}, \ (k=1,2,3)\ &\text{ on } \Gamma_{T}.
\end{cases}
\end{align}

\subsubsection{Solvability of the adjoint system}
In this part, we will show the existence of a unique solution of \eqref{adjoint3} and \eqref{adjoint4}.
To do it, set
\begin{align}\label{reset}
\widetilde{\mathbf{p}}_1(t)={\bf p}_1(T-t),\ \ \widetilde{\mathbf{p}}_2={\mathbf{p}}_2(T-t),\ \ \text{ and }\ \ \
\widetilde{\bf\pi}(t) ={\bf\pi}(T-t).
\end{align}
Then \eqref{adjoint3}--\eqref{adjoint4} becomes
\begin{align}\label{adjoint6}
\begin{cases}
\partial_{t} \widetilde{\mathbf{p}}_1-\Delta \widetilde{\mathbf{p}}_1-\nabla \widetilde{\bf\pi}-\mathbf{u}(T-t)
\cdot \nabla\widetilde{\mathbf{p}}_1\\
=-\nabla\mathbf{u}(T-t)\widetilde{\mathbf{p}}_1-\nabla\mathbf{d}(T-t)\widetilde{\mathbf{p}}_2
+\beta_{1} (\mathbf{u} -\mathbf{u}_{Q_{T}})(T-t),\\
\nabla \cdot \widetilde{\mathbf{p}}_1=0,\\
\partial_{t} \widetilde{\mathbf{p}}_2-\Delta \widetilde{\mathbf{p}}_2-\mathbf{u}(T-t)\cdot\nabla \widetilde{\mathbf{p}}_2\\
+\partial_{i}(\partial_{j}\mathbf{d}(T-t)\partial_{j}\widetilde{{\bf p}}_1^{i})
 +\partial_{j}(\partial_{i}\mathbf{d}(T-t)\partial_{j}\widetilde{{\bf p}}_1^{i})\\
=  |\nabla\mathbf{d}|^{2}(T-t)\widetilde{\mathbf{p}}_2
   -2\nabla\cdot((\nabla\mathbf{d}\mathbf{d})(T-t)\cdot\widetilde{\mathbf{p}}_2))
   +\beta_2(\mathbf{d}-\mathbf{d}_{Q_{T}})(T-t),
\end{cases}
\end{align}
in $Q_T$, under the boundary and initial condition:
\begin{align}\label{adjoint7}
\begin{cases}
\big(\widetilde{\mathbf{p}}_1, \widetilde{\mathbf{p}}_2\big)=(0,0) &\text{ on } \Gamma_{T},\\
\big(\widetilde{\mathbf{p}}_1,\widetilde{{\bf p}}_2\big)=\big(\beta_{3}(\mathbf{u}(T)-\mathbf{u}_{\Omega}),
\beta_{4}(\mathbf{d}(T)-\mathbf{d}_{\Omega})\big) 
&\text{ in } \Omega\times\{0\}.
\end{cases}
\end{align}
We have the following existence result to \eqref{adjoint6}--\eqref{adjoint7}.

\begin{theorem}\label{bdd-pq}
Assume (A1) and (A2) hold, let $(\mathbf{u},\mathbf{d})\in \mathcal{H}$ and $({\bf u}_\Omega, {\bf d}_\Omega)$
satisfy
\begin{align}
\begin{cases}
\mathbf{u}_{\Omega} \in \mathbf{V},\quad &\text{ if }  \beta_{3}>0,\\
 \mathbf{d}_{\Omega} \in H^{1}(\Omega, \mathbb{S}_{+}^{2})\ \text{ with }\ 
 (\mathbf{d}(T)-\mathbf{d}_{\Omega})|_{\Gamma}={0}, &\text{ if } \beta_{4}>0.
\end{cases}
 \end{align}
 Then the system \eqref{adjoint6}--\eqref{adjoint7} admits a unique weak solution 
 $(\widetilde{\mathbf{p}}_1, \widetilde{\bf\pi}, \widetilde{\mathbf{p}}_2)$ such that 
 \begin{align*}
&\widetilde{\mathbf{p}}_1\in C([0, T], \mathbf{V})\cap L^{2}_tH^{2}_x(Q_T),\nonumber\\
&\widetilde{\bf\pi}\in L^{2}_tH^{1}_x(Q_T) \ \ \text{with} \ \  \int_{\Omega} \widetilde{\bf\pi}(x,t)\,dx=0,\nonumber\\
&\widetilde{\mathbf{p}}_2\in
C([0, T], L^{2}(\Omega,\mathbb{R}^{3})) \cap L^{2}_tH^{1}_{0}(Q_T).
\end{align*}
Moreover, it holds that
\begin{align}\label{bdd-pq0}
&\|\widetilde{\mathbf{p}}_1(t)\|_{L^\infty_tH^{1}_x(Q_T)}+\|\widetilde{\mathbf{p}}_2(t)\|_{L^\infty_tL^{2}_x(Q_T)}\nonumber\\
&+\int_{0}^{T} (\|\widetilde{\mathbf{p}}_1(\tau)\|_{H^{2}(\Omega)}^{2}+\|\widetilde{\mathbf{p}}_2(\tau)\|_{H^{1}(\Omega)}^{2})
\leq C_{T},
\end{align}
where $C_T>0$ is a constant depending on  $\|(\mathbf{u},\mathbf{d})\|_{\mathcal{H}}$, $\beta_{1}\|\mathbf{u}-\mathbf{u}_{Q_{T}}\|_{L^{2}(Q_T)}$, $\beta_{2}\|\mathbf{d}- \mathbf{d}_{Q_{T}}\|_{L^{2}(Q_T)}$, $\beta_{3}\|\mathbf{u}_{\Omega}\|_{H^{1}(\Omega)}$, $\beta_{4}\|\mathbf{d}_{\Omega}\|_{H^{1}(\Omega)}$, $\Omega$, and $T$.
For any $s\in (0,2)$, it also holds that
\begin{align}\label{s-reg}
\partial_t\widetilde{\bf p}_2,\nabla^2\widetilde{\bf p}_2\in L^{2-s}(Q_T).
\end{align}
\end{theorem}

\begin{proof}
The existence of weak solutions follows from the Faedo-Galerkin method, similar to \cite{CRW2017} Proposition 4.1,
which is left for the readers. Here we sketch the proof of a priori estimates.

Multiplying \eqref{adjoint6} by $\Delta\widetilde{\mathbf{p}}_1$, \eqref{adjoint7} by 
$\widetilde{\mathbf{p}}_2$, and adding the resulting equations, we obtain
\begin{align} \label{pq-est1}
&\frac{1}{2}\frac{d}{dt} (\|\nabla \widetilde{\mathbf{p}}_1\|_{L^{2}(\Omega)}^{2}
+\|\widetilde{\mathbf{p}}_2\|_{L^{2}(\Omega)}^{2})
+(\|\Delta\widetilde{\mathbf{p}}_1\|_{L^{2}(\Omega)}^{2}
+\|\nabla\widetilde{\mathbf{p}}_2\|_{L^{2}(\Omega)}^{2})\\
&=-\int_\Omega \langle\mathbf{u} (T-t)\cdot \nabla\widetilde{\mathbf{p}}_1,  \Delta\widetilde{\mathbf{p}}_1\rangle
 +\int_{\Omega} \langle(\nabla\mathbf{u}(T-t))\widetilde{\mathbf{p}}_1, \Delta\widetilde{\mathbf{p}}_1\rangle\nonumber\\
& \ \ \ +\int_{\Omega} \langle(\nabla\mathbf{d}(T-t))\widetilde{\mathbf{p}}_2,  \Delta \widetilde{\mathbf{p}}_1\rangle 
-\int_{\Omega} \beta_{1} \langle(\mathbf{u}-\mathbf{u}_{Q_T})(T-t), \Delta\widetilde{\mathbf{p}}_1\rangle\nonumber\\
&\ \ \  -\!\int_{\Omega}\!\langle \partial_{i}(\partial_{j}\mathbf{d}(T-t)\partial_{j}\widetilde{\bf p}^{i}_1)
 +\partial_{j}(\partial_{i}\mathbf{d}(T-t)\partial_{j}\widetilde{\bf p}_1^{i}), \widetilde{\mathbf{p}}_2\rangle\nonumber\\
&\ \ \  +\!\int_{\Omega}\langle|\nabla \mathbf{d}|^{2}(T-t)\widetilde{\mathbf{p}}_2, \widetilde{\mathbf{p}}_2\rangle
-2\!\int_{\Omega}\langle\nabla\cdot(\nabla\mathbf{d}\mathbf{d}(T-t)\cdot\widetilde{\mathbf{p}}_2), \widetilde{\mathbf{p}}_2\rangle\nonumber\\
&\ \ \ +\beta_{2}\!\int_{\Omega}\langle(\mathbf{d}\!-\!\mathbf{d}_{Q_{T}})(T-t), \widetilde{\mathbf{p}}_2\rangle
 \nonumber\\
 &=\sum_{i=1}^8 I_i.\nonumber
 \end{align}
We can estimate $I_i$ ($1\le i\le 8$) as follows. 
 \begin{align*}
 |I_1|&\le C\|{\bf u}(T-t)\|_{L^4(\Omega)}\|\nabla\widetilde{\bf p}_1\|_{L^4(\Omega)}
 \|\Delta\widetilde{\bf p}_1\|_{L^2(\Omega)}\\
 &\le \frac{1}{16}\|\Delta\widetilde{\bf p}_1\|_{L^2(\Omega)}^2
  +C\|{\bf u}(T-t)\|_{H^1(\Omega)}^4\|\nabla\widetilde{\bf p}_1\|_{L^2(\Omega)}^2,
  \end{align*}
 \begin{align*}
 |I_2|&\le C\|\nabla{\bf u}(T-t)\|_{L^4(\Omega)}\|\widetilde{\bf p}_1\|_{L^4(\Omega)}
  \|\Delta\widetilde{\bf p}_1\|_{L^2(\Omega)}\\
  &\le \frac{1}{16}\|\Delta\widetilde{\bf p}_1\|_{L^2(\Omega)}^2
  +C\|{\bf u}(T-t)\|_{H^2(\Omega)}^2\|\nabla\widetilde{\bf p}_1\|_{L^2(\Omega)}^2,
  \end{align*}
 \begin{align*}
 |I_3|&\le C\|\nabla{\bf d}(T-t)\|_{L^4(\Omega)}\|\widetilde{\bf p}_1\|_{L^4(\Omega)}
  \|\Delta\widetilde{\bf p}_1\|_{L^2(\Omega)}\\
  &\le \frac{1}{16}\|\Delta\widetilde{\bf p}_1\|_{L^2(\Omega)}^2
  +C\|{\bf d}(T-t)\|_{H^2(\Omega)}^2\|\nabla\widetilde{\bf p}_1\|_{L^2(\Omega)}^2,
 \end{align*}
 \begin{align*}
 |I_4|&\le C\beta_1\|({\bf u}-{\bf u}_{Q_T})(T-t)\|_{L^2(\Omega)} \|\Delta\widetilde{\bf p}_1\|_{L^2(\Omega)}\\
 &\le \frac{1}{16}\|\Delta\widetilde{\bf p}_1\|_{L^2(\Omega)}^2
+C\beta_1^2\|({\bf u}-{\bf u}_{Q_T})(T-t)\|_{L^2(\Omega)}^2,
 \end{align*}
 \begin{align*}
 |I_5|&
 \le C\|\nabla^2{\bf d}(T-t)\|_{L^2(\Omega)}\|\nabla\widetilde{\bf p}_1\|_{L^4(\Omega)}\|\widetilde{\bf p}_2\|_{L^4(\Omega)}\\
 &\ \ +C\|\nabla{\bf d}(T-t)\|_{L^4(\Omega)}\|\Delta\widetilde{\bf p}_1\|_{L^2(\Omega)}\|\widetilde{\bf p}_2\|_{L^4(\Omega)}\\
 &\le \frac1{16}\big(\|\Delta\widetilde{\bf p}_1\|_{L^2(\Omega)}^2+\|\nabla\widetilde{\bf p}_2\|_{L^2(\Omega)}^2\big)\\
 &\ \ +C(1+\|{\bf d}(T-t)\|_{H^2(\Omega)}^2)\|{\bf d}(T-t)\|_{H^2(\Omega)}^2\\
 &\ \ \ \ \ \ \ \ \cdot(\|\nabla\widetilde{\bf p}_1\|_{L^2(\Omega)}^2
 +\|\widetilde{\bf p}_2\|_{L^2(\Omega)}^2),
 \end{align*}
 \begin{align*}
 |I_6|&\le C\|\nabla{\bf d}(T-t)\|_{L^4(\Omega)}^2\|\widetilde{\bf p}_2\|_{L^4(\Omega)}^2\\
 &\le  \frac1{16}\|\nabla\widetilde{\bf p}_2\|_{L^2(\Omega)}^2+
 C\|{\bf d}(T-t)\|_{H^2(\Omega)}^4\|\widetilde{\bf p}_2\|_{L^2(\Omega)}^2,
 \end{align*}
 \begin{align*}
 |I_7|&\le C\|\nabla^2{\bf d}(T-t)\|_{L^2(\Omega)}\|\widetilde{\bf p}_2\|_{L^4(\Omega)}^2
 +C\|\nabla{\bf d}(T-t)\|_{L^4(\Omega)}^2\|\widetilde{\bf p}_2\|_{L^4(\Omega)}^2\\
 &\ \ +C\|\nabla{\bf d}(T-t)\|_{L^4(\Omega)}\|\nabla\widetilde{\bf p}_2\|_{L^2(\Omega)}\|\widetilde{\bf p}_2\|_{L^4(\Omega)}\\
 &\le \frac1{16}\|\nabla\widetilde{\bf p}_2\|_{L^2(\Omega)}^2
 +C\|{\bf d}(T-t)\|_{H^2(\Omega)}^2\|\widetilde{\bf p}_2\|_{L^2(\Omega)}^2\\
 &\ \ +C\|{\bf d}(T-t)\|_{H^2(\Omega)}^4\|\widetilde{\bf p}_2\|_{L^2(\Omega)}^2,
 \end{align*}
 \begin{align*}
 |I_8|&\le C\beta_2\|({\bf d}-{\bf d}_{Q_T})(T-t)\|_{L^2(\Omega)}\|\widetilde{\bf p}_2\|_{L^2(\Omega)}\\
 &\le \|\widetilde{\bf p}_2\|_{L^2(\Omega)}^2+C\beta_2^2\|({\bf d}-{\bf d}_{Q_T})(T-t)\|_{L^2(\Omega)}^2.
 \end{align*}
 Putting these estimates into \eqref{pq-est1}, we obtain
 \begin{align*}
 &\frac{d}{dt} (\|\nabla \widetilde{\mathbf{p}}_1\|_{L^{2}(\Omega)}^{2}
+\|\widetilde{\mathbf{p}}_2\|_{L^{2}(\Omega)}^{2})
+(\|\Delta\widetilde{\mathbf{p}}_1\|_{L^{2}(\Omega)}^{2}
+\|\nabla\widetilde{\mathbf{p}}_2\|_{L^{2}(\Omega)}^{2})\\
&\le C(1+\|{\bf d}(T-t)\|_{H^2(\Omega)}^2)\|{\bf d}(T-t)\|_{H^2(\Omega)}^2(\|\nabla\widetilde{\bf p}_1\|_{L^2(\Omega)}^2
 +\|\widetilde{\bf p}_2\|_{L^2(\Omega)}^2)\\
 &\ \ +C(\|{\bf u}(T-t)\|_{H^1(\Omega)}^4+\|{\bf u}(T-t)\|_{H^2(\Omega)}^2)(\|\nabla\widetilde{\bf p}_1\|_{L^2(\Omega)}^2
 +\|\widetilde{\bf p}_2\|_{L^2(\Omega)}^2)\\
&\ \ +C\big(\beta_1^2\|({\bf u}-{\bf u}_{Q_T})(T-t)\|_{L^2(\Omega)}^2
+\beta_2^2\|({\bf d}-{\bf d}_{Q_T})(T-t)\|_{L^2(\Omega)}^2\big).
 \end{align*}
Since $({\bf u}, {\bf d})\in\mathcal{H}$, we have that
\begin{align*}
&\int_0^T (\|{\bf u}(T-t)\|_{H^1(\Omega)}^4+\|{\bf d}(T-t)\|_{H^2(\Omega)}^4)\,dt\le C\|({\bf u}, {\bf d})\|_{\mathcal{H}}^4,\\
&\int_0^T(\|{\bf u}(T-t)\|_{H^2(\Omega)}^2+\|{\bf d}(T-t)\|_{H^2(\Omega)}^2)\,dt\le C\|({\bf u}, {\bf d})\|_{\mathcal{H}}^2.
\end{align*}
Hence we can apply Gronwall's inequality
to show \eqref{bdd-pq0}. Note that \eqref{bdd-pq0}, together with \eqref{adjoint6}--\eqref{adjoint7}, implies that 
$\partial_{t}\widetilde{\mathbf{p}}_1\in L^{2}([0,T], \mathbf{H})$, 
$\partial_{t}\widetilde{\mathbf{p}}_2\in L^{2}_tH^{-1}_x(Q_T,\mathbb{R}^{3})$,
and $\nabla\widetilde{\bf\pi}\in L^{2}(Q_T)$. 

Observe that 
$$\partial_{t} \widetilde{\mathbf{p}}_2-\Delta \widetilde{\mathbf{p}}_2={\bf F},$$
where
\begin{align*}
|{\bf F}|&\le C\big[(|{\bf u}|+|\nabla{\bf d}|)(T-t)|\nabla\widetilde{\bf p}_2|
+|\nabla^2{\bf d}|(T-t)|\nabla\widetilde{\bf p}_1|\\
&\ \ \ \ \ \ +|\nabla{\bf d}|(T-t)|\nabla^2\widetilde{\bf p}_1|
+|\nabla^2{\bf d}|(T-t)|\widetilde{\bf p}_2|+|({\bf d}-{\bf d}_{Q_T})|(T-t)\big].
\end{align*}
It follows easily from \eqref{bdd-pq} and the fact that $({\bf u}, {\bf d})\in\mathcal{H}$ that
$$|{\bf F}|\in L^{2-s}(Q_T),$$
for any $0<s<2$. 

Hence we can apply the standard $L^{2-s}$-theory of parabolic equations
to deduce $\partial_t{\widetilde{\bf p}}_2, \nabla^2{\widetilde{\bf p}}_2\in L^{2-s}(Q_T)$.
This completes the proof. \end{proof}

From the relations \eqref{reset} and \eqref{adjoint5}, we have
\begin{corollary}
Under the same assumptions of Theorem \ref{bdd-pq}, the adjoint system \eqref{adjoint3}--\eqref{adjoint4} admits a unique weak solution $(\mathbf{p}_1, {\bf\pi}, \mathbf{p}_2)$, satisfying the same properties as for the weak solution 
$(\widetilde{\mathbf{p}}_1, \widetilde{\bf\pi}, \widetilde{\mathbf{p}}_2)$
to the system \eqref{adjoint6}--\eqref{adjoint7} stated in Theorem \ref{bdd-pq}. 
Moreover, the Lagrange multipliers $(\mathbf{q}_1, \mathbf{q}_2)$ are
uniquely determined by \eqref{adjoint5} such that
\begin{align}\label{bdd-pq1}
\mathbf{q}_1\in L^2_tH^{\frac{1}{2}}_x(\Gamma_T,\mathbb{R}^{2}),
\quad \mathbf{q}_2\in L^{1}_tH^{\frac{1}{2}}_x(\Gamma_T, \mathbb{R}^{3}).
\end{align}
\end{corollary}

\begin{proof}
It suffices to prove \eqref{bdd-pq1}. 
From \eqref{reset}, \eqref{bdd-pq},  and the trace Theorem, we have that
${\mathbf{p}}_1\in L^{2}_tH^{\frac{3}{2}}_x(\Gamma_T,\mathbb{R}^{2})$. This, combined with
$\eqref{adjoint5}_{1}$,  implies that $\mathbf{q}_1\in L^2_tH^{\frac{1}{2}}_x(\Gamma_T,\mathbb{R}^{2})$.

To estimate $\mathbf{q}_{2}$, first observe that \eqref{reset}, \eqref{bdd-pq},  and  the trace Theorem imply that
$\partial_{\bm{\nu}}{\mathbf{p}}_2\in L^{2}_t H^{\frac{1}{2}}_x(\Gamma_T).$
Also we can estimate
\begin{align*}
&\|\partial_{j}{\bf d}^{k}\partial_{j}{\bf p}_1^{i}
 +\partial_{i}{\bf d}^{k}\partial_{j}{\bf p}_1^{i}\|_{L^1_tH^{1}_x(Q_T)}\\
& \leq C\big\|(|\nabla{\bf d}|+|\nabla^2 {\bf d}|)|\nabla{\bf p}_1|+|\nabla{\bf d}||\nabla^2{\bf p}_1|\big\|
 _{L^1_tL^2_x(\Omega)}\\
 &\le C\|{\bf d}\|_{L^\infty_tH^2_x(Q_T)}\|{\bf d}\|_{L^2_tH^3_x(Q_T)}
\|{\bf p}_1\|_{L^\infty_tH^1_x(Q_T)}\|{\bf p}_1\|_{L^2_tH^2_x(Q_T)}<\infty.
 \end{align*}
 Hence we have that
\begin{align*}
\partial_{j}{\bf d}^{k}\partial_{j}{\bf p}_1^{i}{\bm\nu}^{i}
 +\partial_{i}{\bf d}^{k}\partial_{j}{\bf p}_1^{i}{\bm\nu}^{j}
 \in L^{1}_tH^{\frac{1}{2}}_x(\Gamma_T),
\end{align*}
which, together with $\eqref{adjoint5}_{2}$, yields 
that ${\mathbf{q}}_2\in L^{1}_tH^{\frac{1}{2}}_x(\Gamma_T)$.
\end{proof}

\subsubsection{The first-order necessary condition via adjoint systems}
With the help of previous subsections, 
we are able to formulate another necessary condition for optimal boundary control
in terms of adjoint systems. More precisely, we have the following theorem.

\begin{theorem}\label{nece2}
 Assume (A1) and (A2). For $M>0$, 
 let $(\mathbf{u}_0, \mathbf{d}_0)\in \mathbf{V}\times \mathbf{H}^2(\Omega,\mathbb{S}_{+}^{2})$  and
\begin{align*}
& \mathbf{u}_{\Omega} \in \mathbf{V},\quad \mathrm{if}\  \beta_{3}>0,\\
& \mathbf{d}_{\Omega} \in H^{1}(\Omega, \mathbb{S}_{+}^{2})\quad \text{with}\ (\mathbf{d}(T)-\mathbf{d}_{\Omega})|_{\Gamma}=\mathbf{0},\quad \mathrm{if}\ \beta_4>0.
 \end{align*}
Let  $\mathbf{h}$ be an optimal boundary control for (\ref{CF1}) in $\widetilde{\mathcal{U}}_{M}$,
with the associate state $(\mathbf{u},\mathbf{d})=\mathcal{S}(\mathbf{h})\in \mathcal{H}$ 
and the adjoint state $({\mathbf{p}}_1, {\mathbf{p}}_2)$ given by \eqref{adjoint3} and \eqref{adjoint4}.
For any $\widehat{\mathbf{h}}\in \widetilde{\mathcal{U}}_{M}$, if $\bm\xi$ is the section
of ${\bf h}^*T\widetilde{\mathcal{U}}$ given by 
$$\bm\xi=\frac{d}{ds}\big|_{s=0} \Pi^{-1}(s\Pi({\bf h})+(1-s)\Pi(\widehat{\bf h})),$$
then the following variational inequality holds:
\begin{align}\label{cond2}
&\beta_5\!\int_{\Gamma_{T}}\!\langle\mathbf{h}-\mathbf{e}_{3}, \bm{\xi}\rangle
+\int_{\Gamma_{T}}\! \langle\partial_{j}\mathbf{d}\partial_{j}{\bf p}_1^{i}\bm\nu_{i}
 \!+\!\partial_{i}\mathbf{d}\partial_{j}{\bf p}_1^{i}\bm\nu_{j}-\partial_{\bm\nu}{\bf p}_2, \bm{\xi}\rangle \geq 0.
\end{align}
\end{theorem}

\begin{proof} Set $${\bf h}(s)=\Pi^{-1}(s\Pi({\bf h})+(1-s)\Pi(\widehat{\bf h}))\in C^1([0,1], \widetilde{\mathcal{U}}_M),$$
and 
$$({\bf u}(s), {\bf d}(s))=\mathcal{S}({\bf h}(s))\ {\rm{for}}\ s\in [0,1].$$ 
Then ${\bf h}(0)={\bf h}$, $({\bf u}, {\bf d})=({\bf u}(0), {\bf d}(0))$,
and ${\bf h}(1)=\widehat{\bf h}$.  From the minimality of
$\mathcal{G}$ at ${\bf h}$, we obtain that
\begin{align*}
0&\le \frac{d}{ds}\big|_{s=0} \mathcal{G}(({\bf u}, {\bf d}), {\bf h}(s), {\bf p}_1, {\bf p_2},
{\bf\pi}, {\bf q}_1, {\bf q}_2)\\
&=\beta_5\int_{\Gamma_T} \langle {\bf h}-{\bf e}_3, \bm\xi\rangle
+\int_{\Gamma_T}\langle{\bf q}_2, \bm\xi\rangle.
\end{align*}
This, combined with \eqref{adjoint5}$_2$, gives rise to \eqref{cond2}.
\end{proof}

\medskip


\noindent\textbf{Acknowledgments}. 
Liu is partially supported by NSF of China 11401202. Wang is partially supported by NSF DMS 1764417. Zhang and Zhou are supported
by a Chinese Council of Scholarship. This work was completed while Liu, Zhang, and Zhou were visiting Department of Mathematics, Purdue University. They would like to
express their gratitude to the department for its hospitality.


\vspace{0.2cm}\sc

\noindent Key Laboratory of High Performance Computing and Stochastic Information Processing, 
College of Mathematics \& Statistics, Hunan Normal Univ., Changsha, Hunan 410081, P.R.China; {\tt liuqao2005@163.com}

\vspace{0.2cm}\sc
\noindent Department of Mathematics, Purdue Univ., West Lafayette, IN 47907, USA;
{\tt wang2482@purdue.edu}

\vspace{0.2cm}\sc
\noindent School of Mathematical Sciences, Fudan Univ., Shanghai 200433, P. R. China;
{\tt 13110180015@fudan.edu.cn}

\vspace{0.2cm}\sc
\noindent School of Mathematical Sciences, Xiamen Univ., Xiamen, Fujian 361005,
P. R. China; {\tt jfzhouxmu@163.com}

\end{document}